\documentclass[intlim,righttag,10pt]{amsart}
\usepackage{url}
\usepackage{amscd,bbm}
\usepackage{amssymb}
\usepackage[all]{xy}
\RequirePackage{mathrsfs}
\def\jcdot{\scriptscriptstyle\bullet}

\def\invlim{\mathop{\vtop{\ialign{##\crcr$\hfill{\lim}\hfil$\crcr
\noalign{\kern1pt\nointerlineskip}\leftarrowfill\crcr\noalign
{\kern -3pt}}}}\limits}
\def\dirlim{\mathop{\vtop{\ialign{##\crcr$\hfill{\lim}\hfil$\crcr
\noalign{\kern1pt\nointerlineskip}\rightarrowfill\crcr\noalign
{\kern -3pt}}}}\limits}
\def\lomapr#1{\smash{\mathop{\relbar\joinrel\longrightarrow}\limits^{#1}}}
 \def\verylomapr#1{\smash{\mathop{\relbar\joinrel\relbar\joinrel\relbar\joinrel\longrightarrow}\limits^{#1}}}
\def\veryverylomapr#1{\smash{\mathop{\relbar\joinrel\relbar\joinrel\relbar
\joinrel\relbar\joinrel\relbar\joinrel\longrightarrow}\limits^{#1}}}
\def\phi{\varphi}
\def\epsilon{\varepsilon}
\let\mathcal\mathscr

\newtheorem{theorem}{Theorem}[section]
 \newtheorem{lemma}[theorem]{Lemma}
 \newtheorem{proposition}[theorem]{Proposition}
 \newtheorem{corollary}[theorem]{Corollary}

\theoremstyle{definition}
\newtheorem{definition}[theorem]{Definition}
\newtheorem{remark}[theorem]{Remark}
\newtheorem{construction}[theorem]{Construction}
\newtheorem{recall}[theorem]{Recall}

\newtheorem{example}[theorem]{Example}
\newtheorem*{acknowledgments}{Acknowledgments}
\newtheorem{num}[theorem]{}
\numberwithin{equation}{section}
\newtheorem*{thm*}{Theorem}

\newcommand{\Qp}{\mathbf{Q}_p}

\renewcommand{\phi}{\varphi}

\newcommand{\eff}{\operatorname{eff}}

\newcommand{\R}{\mathrm {R} }

\newcommand{\pst}{\operatorname{pst}}

\newcommand{\bq}{{\mathbf Q}}
\newcommand{\bz}{{\mathbf Z}}

\newcommand{\ovk}{\overline{K} }

\newcommand{\prim}{\operatorname{prim} }

\newcommand{\hk}{\operatorname{HK} }
\newcommand{\bk}{\operatorname{BK} }

\newcommand{\dr}{\operatorname{dR} }
\newcommand{\ad}{\operatorname{ad} }

\newcommand{\Ind}{\operatorname{Ind} }

 \newcommand{\op}{\operatorname{op} }
  \newcommand{\coker}{\operatorname{coker} }
 \newcommand{\coim}{\operatorname{coim} }

 \newcommand{\End}{\operatorname{{End}} } 
 \newcommand{\holim}{\operatorname{holim} }
 \newcommand{\hocolim}{\operatorname{hocolim} }

  \newcommand{\proeet}{\operatorname{pro\acute{e}t}  }

 \newcommand{\eet}{\operatorname{\acute{e}t} }

 \newcommand{\nr}{\operatorname{nr} }
 
 \newcommand{\Spec}{\operatorname{Spec} } 
 \newcommand{\Imm}{\operatorname{Im} } 
 
 \newcommand{\gm}{\operatorname{gm} }

 \newcommand{\Hom}{\operatorname{Hom} }

 \newcommand{\Ext}{\operatorname{Ext} }
 \newcommand{\Rep}{\operatorname{Rep} }

 \newcommand{\Gal}{\operatorname{Gal} }
 
 \newcommand{\can}{ \operatorname{can} }
 
 \newcommand{\id}{ \operatorname{Id} }
\newcommand{\synt}{ \operatorname{syn} }
 
 \newcommand{\Cone}{\operatorname{Cone} }

\newcommand{\st}{\operatorname{st} }

 \newcommand{\kr}{^{\scriptscriptstyle\bullet}}

 \newcommand{\sh}{{\mathcal{H}}}

 \newcommand{\scc}{{\mathcal{C}}}

 \newcommand{\so}{{\mathcal O}}
 
 \newcommand{\se}{{\mathcal{E}}}
 \newcommand{\sa}{{\mathcal{A}}}

 \newcommand{\srr}{{\mathcal{R}}}
\newcommand{\sd}{{\mathcal{D}}}
\newcommand{\sm}{{\mathcal{M}}}

 \newcommand{\wt}{\widetilde}

 \newcommand{\Z}{ {\mathbf Z} }
    
   \newcommand{\Q}{ {\mathbf Q}}
   \newcommand{\N}{{\mathbf N}}
 \newcommand{\B}{{\mathbf B}}
 
\DeclareMathOperator{\Sh}{Sh}
\DeclareMathOperator{\Sch}{Sch}
\DeclareMathOperator{\Sm}{Sm}
\newcommand{\Sma}{\operatorname{Sm}^{aff}}
\newcommand{\TT}{\mathcal T}
\renewcommand{\AA}{\mathbb A}
\newcommand{\PP}{\mathbb P}
\DeclareMathOperator{\uHom}{\underline{Hom} }
\newcommand{\synsp}{\mathcal E_{\synt}}
\newcommand{\synspx}[1]{\mathcal E_{\synt,#1}}
\newcommand{\usMod}{\synsp\!-\!\underline{\operatorname{\mathcal Mod}}}

\newcommand{\sMod}{\synsp\!-\!\operatorname{\mathcal Mod}}
\newcommand{\smod}{\synsp\!-\!\operatorname{mod}}
\newcommand{\smodx}[1]{#1\!-\!\operatorname{mod}}
\DeclareMathOperator{\derR}{R}
\DeclareMathOperator{\derL}{L}
\newcommand{\DMgm}{DM_{gm}}

\newcommand{\MM}{\operatorname{MM}} %Nori cohomological motives
\newcommand{\MMe}{\operatorname{EMM}} %effective Nori cohomological motives
\newcommand{\HMMe}{\operatorname{EHM}} %effective Nori homological motives
\newcommand{\Hm}{H_{\mathrm{mot}}} % Nori cohomological motive of a good pair

\newcommand{\D}{{\mathcal D}}
\newcommand{\T}{{\mathcal T}}

\newcommand{\un}{\mathbbm 1}

\numberwithin{equation}{section}

\begin{document}
 \title[On $p$-adic absolute Hodge cohomology and syntomic coefficients, I.]{On $p$-adic absolute Hodge cohomology and syntomic coefficients, I.}
 \author{Fr\'ed\'eric D\'eglise, Wies{\l}awa Nizio{\l}}
 \date{\today}
\thanks{The authors' research was supported in part by the ANR
 (grants ANR-12-BS01-0002 and ANR-14-CE25, respectively).}
 \email{frederic.deglise@ens-lyon.fr, wieslawa.niziol@ens-lyon.fr}
 \begin{abstract}
We interpret syntomic cohomology defined in \cite{NN} as a $p$-adic absolute Hodge cohomology. This is analogous to the interpretation of Deligne-Beilinson  cohomology as an absolute Hodge cohomology by Beilinson \cite{BE0} and generalizes the results of Bannai  \cite{Ban} and  Chiarellotto, Ciccioni, Mazzari \cite{CCM} in the good reduction case, and of Yamada \cite{Yam} in the semistable reduction case. This interpretation  yields a simple construction of the syntomic descent spectral sequence and its degeneration for proper and smooth varieties. We introduce syntomic coefficients and show that in dimension zero they form a full triangulated subcategory of the derived category of potentially semistable Galois representations.

  Along the way, we obtain $p$-adic realizations of mixed motives including $p$-adic comparison isomorphisms. We apply this to the motivic fundamental group generalizing results of Olsson and Vologodsky \cite{Ol}, \cite{Vo}.
 \end{abstract}
 \maketitle
 \tableofcontents
\section{Introduction}
In \cite{BE0}, Beilinson gave an interpretation of Deligne-Beilinson cohomology as an \emph{absolute Hodge cohomology}, i.e., as derived Hom  in the derived category of mixed Hodge structures. This approach is advantageous: absolute Hodge cohomology allows coefficients. It follows that Deligne-Beilinson cohomology can  be  interpreted as derived Hom between Tate twists in the derived category of Saito's mixed Hodge modules \cite[A.2.7]{HW}.

Syntomic cohomology is a  $p$-adic analog of Deligne-Beilinson  cohomology. The purpose of this paper is to give an analog of the above results for syntomic cohomology. Namely, we will show that
the syntomic cohomology introduced in \cite{NN} is a $p$-adic absolute Hodge cohomology, i.e., it can be expressed  as derived Hom   in the derived   category of $p$-adic Hodge structures, and we will begin the study of syntomic coefficients - an approximation of $p$-adic Hodge modules. This generalizes the results of Bannai  \cite{Ban} and  Chiarellotto, Ciccioni, Mazzari \cite{CCM} in the good reduction case, and of Yamada \cite{Yam} in the semistable reduction case.
 
Let  $K$ be a complete discrete valuation  field
  of mixed characteristic $(0,p)$ with perfect
residue field $k$.  Let  $G_K=\Gal(\overline {K}/K)$ be the Galois group of $K$. For the category of $p$-adic Hodge structures we take  the abelian category $DF_K $ of
 (weakly) admissible filtered $(\varphi,N,G_K)$-modules defined by Fontaine. For a variety $X$ over $K$, we construct a complex $\R\Gamma_{DF_K}(X_{\ovk},r)\in D^b(DF_K)$, $r\in \Z$. The absolute Hodge cohomology of $X$ is then by definition
 $$\R\Gamma_{\sh}(X,r):=\R\Hom_{D^b(DF_K)}(K(0),\R\Gamma_{DF_K}(X_{\ovk},r)),\quad r\in \Z. $$ For $r\geq 0$, it coincides with the syntomic cohomology $\R\Gamma_{\synt}(X,r)$ defined in  \cite{NN}. Recall that the latter was defined  as the following mapping fiber
$$\R\Gamma_{\synt}(X,r)=[\R\Gamma^B_{\hk}(X)^{\phi=p^r,N=0}\lomapr{\iota_{\dr}}\R\Gamma_{\dr}(X)/F^r],
$$
where $\R\Gamma^B_{\hk}(X)$ is the Beilinson-Hyodo-Kato cohomology from \cite{BE1}, $\R\Gamma_{\dr}(X)$ is the Deligne de Rham cohomology, and  the map $\iota_{\dr}$ is the Beilinson-Hyodo-Kato map.

   We present two approaches to the definition of  the complex $\R\Gamma_{DF_K}(X_{\ovk},r)$. In the first one,  we follow Beilinson's construction of the complex of mixed Hodge structures associated to a variety \cite{BE0}. Thus,
 we build the dg category $\sd_{pH}$ of \emph{$p$-adic Hodge complexes} (an analog of Beilinson's  mixed Hodge complexes)
 which is obtained
 by  gluing two dg categories, one, corresponding morally to
 the special fiber, whose objects are equipped with an action of a Frobenius
 and a monodromy operator, and the other one, corresponding
 to the generic fiber, whose objects are equipped with a filtration thought of as the Hodge
 filtration on de Rham cohomology. It contains a dg subcategory of \emph{admissible $p$-adic Hodge complexes} with cohomology groups belonging to $DF_K$. The category $\sd^{\ad}_{pH}$ admits a natural $t$-structure
 whose heart is the category $DF_K$ and   $\sd^{\ad}_{pH}$ is 
 equivalent to the derived category of its heart. That is, we have the following  equivalences of categories $$
\theta: DF_K \stackrel{\sim}{\rightarrow }\sd_{pH}^{\ad,\heartsuit},\quad \theta:  \sd^b(DF_K) \stackrel{\sim}{\rightarrow} \sd^{\ad}_{pH}.
$$ 
  The interest of the category $\sd^{\ad}_{pH}
$ lies in the fact that, for $r\in\Z$, a variety $X$ over $K$ gives rise to the admissible $p$-adic Hodge complex
$$
\R\Gamma_{pH}(X_{\ovk},r):=( \R\Gamma^B_{\hk}(X_{\ovk},r), (\R\Gamma_{\dr}(X),F^{\jcdot+r}), \iota_{\dr})\in\sd^{\ad}_{pH}
$$
 We define $\R\Gamma_{DF_K}(X_{\ovk},r):=\theta^{-1}\R\Gamma_{pH}(X_{\ovk},r)$.

   Since the category $DF_K$ is equivalent to that of potentially semistable representations \cite{CF}, i.e., we have a functor  $V_{\pst}: DF_K\stackrel{\sim}{\to}\Rep_{\pst}(G_K)$, we can also write
\begin{align*}
\R\Gamma_{\sh}(X,r) = \Hom_{\sd^b(\Rep_{\pst}(G_K))}(\Qp,\R\Gamma_{\pst}(X_{\ovk},r)),
  \end{align*}
 for $\R\Gamma_{\pst}(X_{\ovk},r):=V_{\pst}\R\Gamma_{DF_K}(X_{\ovk},r)$.
 Using Beilinson's comparison theorems \cite{BE1} we prove that $\R\Gamma_{\pst}(X_{\ovk},r)\simeq \R\Gamma_{\eet}(X_{\ovk},\Qp(r))$ as Galois modules. It follows that there is a  functorial syntomic descent  spectral sequence (constructed originally by a different, more complicated,  method in \cite{NN})
\begin{equation*}
{}^{\sh}E^{i,j}_2:=H^i_{\st}(G_K,H^j_{\eet}(X_{\ovk},\mathbf {Q}_p(r)))\Rightarrow H^{i+j}_{\sh}(X,r),
\end{equation*}
where $H^i_{\st}(G_K,\cdot):=\Ext^i_{\Rep_{\pst}(G_K)}(\Qp,\cdot)$. By a  classical  argument of Deligne \cite{De1}, it follows from Hard Lefschetz Theorem, that  it degenerates at $E_2$ for $X$ proper and smooth.
 
 \bigskip
  
  A more direct definition of the complex $\R\Gamma_{DF_K}(X_{\ovk},r)$, or, equivalently, of the complex $\R\Gamma_{\pst}(X_{\ovk},r)$  of potentially semistable representations  associated to a variety   was proposed by Beilinson \cite{BS} using Beilinson's Basic Lemma. This lemma allows one to associate a potentially semistable analog of a cellular complex (of a CW-complex) to an affine variety $X$ over $K$:  one stratifies the variety by closed subvarieties such that consecutive relative geometric \'etale cohomology is concentrated in the top degree (and is a potentially semistable representation). For a general $X$ one obtains Beilinson's potentially semistable complex  by a \v{C}ech gluing argument. 
  
 All the $p$-adic cohomologies mentioned above  (de Rham, \'etale,  Hyodo-Kato, and syntomic) behave well hence they lift to
  realizations of both Nori's abelian 
 and Voevodsky's triangulated category of mixed motives. 
 We also lift  the comparison maps between them, thus obtaining
 comparison theorems for mixed
 motives.
 We illustrate this construction by two applications.
 The first one is a $p$-adic realization of the motivic fundamental group including a potentially semistable comparison theorem. We rely on Cushman's motivic (in the sense of Nori) theory of the fundamental group \cite{Cu2}. This generalizes results obtained earlier for curves and proper varieties with good reduction \cite{Ha}, \cite{AIK}, \cite{Vo}, \cite{Ol}.
 The second is a compatibility result. We show that Beilinson's $p$-adic comparison theorems (with compact support or not) are compatible with 
 Gysin morphisms and (possibly mixed) products.

   To define a well-behaved notion of syntomic coefficients (i.e., coefficients for syntomic cohomology) we use Morel-Voevodsky \emph{motivic homotopy theory},
 and more precisely the concept of modules over (motivic) ring spectra. Recall 
 that objects of motivic stable homotopy theory, called spectra,
 represent cohomology theories with suitable properties.
 A multiplicative structure on the cohomology theory corresponds to
 a monoid structure on the representing spectrum, which is
 then called a \emph{ring spectrum}. These objects have to
 be thought of as a generalization of
 ($h$-sheaves of) differential graded algebras.
 In fact, as we will only consider ordinary cohomology theories
 (as opposed to K-theory or algebraic cobordism with
  integral coefficients), we will always restrict to this
        later concept.
 Therefore modules over ring spectra have to be understood
 as the more familiar concept of modules over
 differential graded algebras.  
        
   One of the basic examples
 of a representable cohomology theory is de Rham cohomology
 in characteristic $0$. Denote the corresponding motivic ring spectrum
 by $\mathcal E_{\dr}$. By \cite{CD3}, \cite{Drew}, 
working relatively to a fix
 complex variety $X$, modules over $\mathcal E_{\dr,X}$
 satisfying a suitable finiteness condition correspond naturally
 to (regular holonomic) $\mathcal D_X$-modules of geometric origin.

  In \cite{NN} it is shown that syntomic cohomology can be represented by a motivic  dg algebra  $\se_{\synt}$, i.e., we have
\begin{equation}
\label{form}
\R\Gamma_{\synt}(X,r)=\R\Hom_{DM_h(K,\Qp)}(M(X),\se_{\synt}(r)),
\end{equation}
where $M(X)$ is the Voevodsky's motive associated to $X$ and $DM_h(K,\Qp)$ is the category of $h$-motives.
 So we have the companion notion
 of \emph{syntomic modules}, that is, modules over the
 motivic dg-algebra $\synsp$.
 The main advantage of this definition is that the link with mixed
 motives is rightly given by the construction and, most of all,
 the 6 functors formalism follows easily from the motivic one.
 
     Now the crucial question is to understand how the category of
 syntomic modules is  related to the category  of  filtered
 $(\varphi,N,G_K)$-modules,  the existing candidates for syntomic smooth sheaves  \cite{Fa0}, \cite{Fa1}, \cite{Ts}, \cite{Sch}, and the category of syntomic coefficients  introduced in \cite{DegMaz} by a method analogous to the one we use  but based on 
 Gros-Besser's version of syntomic cohomology.
In this paper we study this question only  in dimension zero, i.e., for syntomic modules over the base field.
 With a suitable notion of finiteness for syntomic modules,
 called \emph{constructibility}, we prove the 
 following theorem.
\begin{thm*}[Theorem \ref{thm:compute_synt_modl}]
The triangulated monoidal category
 of constructible syntomic modules over 
 a $p$-adic field $K$ is equivalent to a full subcategory
 of the derived category of admissible filtered
 $(\varphi,N,G_K)$-modules.
\end{thm*}
It implies, by adjunction from (\ref{form}), that $p$-adic absolute Hodge cohomology coincides with derived Hom in the  (homotopy) category of syntomic modules, i.e., we have
$$
\R\Gamma_{\sh}(X,r)=\R\Hom_{\smod_X}(\se_{\synt,X},\se_{\synt,X}(r)).
$$

In the conclusion of the paper, we use syntomic modules
 to introduce new notions of $p$-adic Galois representations (Definition \ref{df:motivic_rep}). We define
  \emph{geometric representations} which correspond
 to the common intuition of  representations
 associated to  (mixed) motives, and  \emph{constructible representations},
 corresponding to cohomology groups of  Galois realizations of syntomic modules.

We expect that the categories of geometric, constructible, and potentially
 semistable representations are not the same. This is at least what is predicted
 by the current general conjectures. Note that this is in contrast to
 the case of number fields where the analogs of
 these  notions are conjectured to coincide with the known definition
 of ``representations coming from geometry''  \cite{FonMaz}.

\subsubsection{Notation}Let  $\so_K$ be a complete discrete valuation ring with fraction field
$K$  of characteristic 0, with perfect
residue field $k$ of characteristic $p$. Let $\ovk$ be an algebraic closure of $K$. Let
$W(k)$ be the ring of Witt vectors of $k$ with
 fraction field $K_0$ and denote by $K_0^{\nr}$ the maximal unramified extension of $K_0$.  Set $G_K=\Gal(\overline {K}/K)$ and let $I_K$ denote its inertia subgroup.  Let $\phi$ be the absolute
Frobenius on $K_0^{\nr}$.
 We will denote by $\so_K$,
$\so_K^{\times}$, and $\so_K^0$ the scheme $\Spec (\so_K)$ with the trivial, canonical
(i.e., associated to the closed point), and $({\mathbf N}\to \so_K, 1\mapsto 0)$
log-structure respectively. For a scheme $X$ over $W(k)$, $X_n$ will denote its reduction mod $p^n$, $X_0$ will denote its special fiber.

 Let ${\mathcal Var}_K$ denote the category of varieties over $K$. For a dg category $\scc$ with a $t$-structure, 
 we will denote by $\scc^{\heartsuit}$ the heart of the $t$-structure. 
%All the dg categories we work with are strongly pretriangulated and over %$\Qp$ ??. That means that the dg nerve construction carries them to stable %$\infty$ categories
%\cite[Chapter 1, 1.3.1]{HA}, \cite{GF} identifying their homotopy categories %as triangulated categories. 
\begin{acknowledgments} We would like to thank 
Alexander Beilinson, Laurent Berger, Bhargav Bhatt, Fran\c{c}ois Brunault, Denis-Charles Cisinski, Pierre Colmez, Gabriel Dospinescu, Bradley Drew, Veronika Ertl, Tony Scholl, and Peter Scholze for helpful discussions related to the subject of this paper. We thank Madhav Nori for sending us the thesis of Matthew Cushman. Special thanks go to Alexander  Beilinson for  explaining to us his construction of syntomic cohomology, for allowing  us to include it in this paper, and for sending us Nori's notes on Nori's motives. 
\end{acknowledgments}

\section{A $p$-adic absolute Hodge cohomology, I}
\subsection{The derived category of admissible filtered $(\phi, N, G_K)$-modules}
%   Unless otherwise stated, we work in the category of integral quasi-coherent % log-schemes.

\begin{num}
    For a field $K$, let $V_{K}$ denote the category of $K$-vector spaces. It is an abelian category. We will denote by $\sd^b(V_K)$ its bounded derived dg category and by $D^b(V_K)$ -- its bounded derived category. 
   Let $V^K_{\dr}$ denote  the category of $K$-vector spaces with a descending exhaustive separated filtration $F^{\scriptscriptstyle\bullet}$.
The category $V^K_{\dr}$ (and the category of bounded complexes $C^b(V^K_{\dr})$) is additive but not abelian. It is an exact category in the sense of Quillen \cite{Qu}, where  short exact sequences are exact sequences of $K$-vector spaces with  {\em strict morphisms}
 (recall that a morphism $f: M\to N$ is {\em strict} if $f(F^iM)=F^iN\cap \Imm(f)$). It is also a quasi-abelian category in the sense of \cite{Sn} (see \cite[2]{SS} for a quick review). Thus its derived category
 can be studied as usual (see \cite{BBD}). 

   An object $M\in C^b(V^K_{\dr})$ is called  a {\em strict complex} if its differentials are strict.
   There are canonical truncation functors on $C^b(V^K_{\dr})$: 
\begin{align*}
\tau_{\leq n}M & :=\cdots \to M^{n-2}\to M^{n-1}\to 
\ker(d^n)\to 0\to\cdots\\
 \tau_{\geq n} M & :=\cdots \to 0\cdots \to
\coim(d^{n-1}) \to M^n\to M^{n+1}\to\cdots
\end{align*}
The cohomology object $$
\tau_{\leq n}\tau_{\geq n}(M)=\cdots \to 0 \to \coim(d^{n-1})\to \ker(d^n)\to 0\to\cdots
$$
     We will denote the bounded derived dg category of $V^K_{\dr}$ by $\sd^b(V^K_{\dr})$. It is defined as the dg quotient \cite{Dr} of the dg category  $C^b(V^K_{\dr})$ by the full dg 
     subcategory of strictly exact complexes \cite{N}.  A map of complexes is a quasi-isomorphism if and only if it is a quasi-isomorphism on the grading. 
     The homotopy category of $\sd^b(V^K_{\dr})$ is the bounded filtered derived category $D^b(V^K_{\dr})$. 
  
 For $n\in \Z$, let $D^b_{\leq n}(V^K_{\dr})$ (resp. $D^b_{\geq n}(V^K_{\dr})$) denote the full subcategory of $D^b(V^K_{\dr})$  of complexes that are strictly exact in degrees $k >n$ (resp. $k<n$). The above truncation maps extend to truncations functors
  $\tau_{\leq n}: D^b(V^K_{\dr})\to D^b_{\leq n}(V^K_{\dr})$ and $\tau_{\geq n}: D^b(V^K_{\dr})\to D^b_{\geq n}(V^K_{\dr})$. The pair $(D^b_{\leq n}(V^K_{\dr}),D^b_{\geq n}(V^K_{\dr})$) defines a t-structure on $D^b(V^K_{\dr})$ by \cite{Sn}. The heart $D^b_{\leq 0}(V^K_{\dr})^{\heartsuit}$  is an abelian category $LH(V^K_{\dr})$. We have an embedding
$V^K_{\dr}\hookrightarrow LH(V^K_{\dr})$ 
that induces an equivalence $D^b(V^K_{\dr})\stackrel{\sim}{\to} D^b(LH(V^K_{\dr}))$. This t-structure pulls back to a t-structure on the derived dg category $\sd^b(V^K_{\dr})$.
\end{num}
\begin{num}
 
       Let the field $K$ be again as at the beginning of this article. A $\phi$-module over $K_0$ is a pair $(D,\phi)$, where $D$ is a  $K_0$-vector space and the Frobenius
$\phi=\phi_D$ is a $\phi$-semilinear endomorphism of $D$. We will usually write $D$ for $(D,\phi)$.  The category $M_{K_0}(\phi)$ of $\phi$-modules over $K_0$ is abelian and we will denote by $\sd^b_{K_0}(\phi)$ its bounded derived dg category.

   For $D_1,D_2\in M_{K_0}(\phi)$, let $\Hom_{K_0,\phi}(D_1,D_2)$ denote the group of Frobenius morphisms. We have the exact sequence
\begin{align}
\label{exact}
0\to \Hom_{K_0,\phi}(D_1,D_2)\to \Hom_{K_0}(D_1,D_2)\to  \Hom_{K_0}(D_1,\phi_*D_2), \end{align}
where the last map is $\delta: x\mapsto \phi_{D_2}x-\phi_*(x)\phi_{D_1}$. Set $\Hom^{\sharp}_{K_0,\phi}(D_1,D_2):=\Cone(\delta)[-1]$. Beilinson proves the following lemma.
\begin{lemma}(\cite[1.13,1.14]{BE1}) 
\item For $D_1,D_2\in \sd^b_{K_0}(\phi)$, the map $R\Hom_{K_0,\phi}(D_1,D_2)\to \Hom^{\sharp}_{K_0,\phi}(D_1,D_2)$ is a quasi-isomorphism, i.e., 
$$
R\Hom_{K_0,\phi}(D_1,D_2)=\Cone(\Hom_{K_0}(D_1,D_2)\stackrel{\delta}{\to} \Hom_{K_0}(D_1,\phi_*D_2))[-1]
$$
\end{lemma}
\begin{proof} Note that, 
for $D_1,D_2\in \sd^b_{K_0}(\phi)$, from the exact sequence (\ref{exact}), we get a map
$$
\alpha: R\Hom_{K_0,\phi}(D_1,D_2)\to\Cone(R\Hom_{K_0}(D_1,D_2)\stackrel{\delta}{\to}R\Hom_{K_0}(D_1,R\phi_*D_2))[-1]
$$ We will show that it is a quasi-isomorphism.

     The forgetful functor $M_{K_0}(\phi)\to V_{K_0}$ has a right adjoint $M\to M_{\phi}$, where the $\phi$-module $M_{\phi}:=\prod_{n\geq 0}\phi^n_*M$ with Frobenius $\phi_{M_{\phi}}: (x_0,x_1,\ldots,)\to (x_1,x_2,\ldots)$. The functor $M\to M_{\phi}$ is left exact and preserves injectives. Since all $K_0$-modules are injective, the map $M\to M_{\phi}$, $m\mapsto (m,\phi(m),\phi^2(m),\ldots)$, embeds $M$ into an injective $\phi$-module.  It suffices thus to check that the map $\alpha$ is a quasi-isomorphism for $D_1$ any $\phi$-module and $D_2=G_{\phi}$. We calculate
\begin{align*}
R\Hom_{K_0,\phi}(D_1,G_{\phi}) & \stackrel{\sim}{\leftarrow}\Hom_{K_0,\phi}(D_1,G_{\phi})
\stackrel{\sim}{\to}
\Cone(\Hom_{K_0}(D_1,G_{\phi})\stackrel{\delta}{\to}\Hom_{K_0}(D_1,\phi_*G_{\phi}))[-1]\\
& \stackrel{\sim}{\to}\Cone(R\Hom_{K_0}(D_1,G_{\phi})\stackrel{\delta}{\to}R\Hom_{K_0}(D_1,R\phi_*G_{\phi}))[-1]
\end{align*}
This proves the lemma.
\end{proof}
\end{num}
\begin{num}
   A $(\phi,N)$-module is a triple $(D,\phi_D,N)$ (abbreviated often to $D$), where $(D,\phi_D)$ is a finite rank $\phi$-module over $K_0$ and $\phi_D$ is an automorphism, and $N$ is a $K_0$-linear endomorphism of $D$ such that $N\phi_D=p\phi_D N$ (hence $N$ is nilpotent). The category $M_{K_0}(\phi,N)$  of $(\phi,N)$-modules is naturally a Tannakian tensor ${\mathbf Q}_p$-category and $(M,\phi_M,N)\mapsto M$ is a fiber functor over $K_0$. Denote by $\sd^b_{\phi,N}(K_0)$ and $D^b_{\phi,N}(K_0)$ the
   corresponding bounded derived dg category and bounded derived  category, respectively.

   For $(\phi, N)$-modules $M,T$, let $Hom_{\phi,N}(M,T)$ be the group of $(\phi,N)$-module morphisms.  Let $Hom^{\sharp}(M,T)$ be the complex \cite[1.15]{BE1}
  $$Hom_{K_0}(M,T)  \to Hom_{K_0}(M,\phi_*T)\oplus Hom_{K_0}(M,T)\to Hom(M,\phi_*T)$$
  beginning in degree $0$ and with the following differentials 
  \begin{align*}
  d_0:\quad  & x\mapsto (\phi_2 x - x\phi_1, N_2x - x N_1);\\
  d_1:\quad & (x,y)\mapsto (N_2x-px N_1-p\phi_2y+y\phi_1)
  \end{align*}
  Clearly, we have $\Hom_{\phi,N}(M,T)=H^0\Hom^{\sharp}_{\phi,N}(M,T)$.  Complexes $\Hom^{\sharp}_{\phi,N}$ compose naturally and supply 
  a dg category structure on the category of bounded complexes of $(\phi, N)$-modules.

    Beilinson states the following fact.
 \begin{lemma}(\cite[1.15]{BE1})
 \label{diagram1}
For $D_1,D_2\in \sd^b_{K_0}(\phi,N)$, the map $R\Hom_{\phi,N}(D_1,D_2)\to \Hom^{\sharp}_{\phi,N}(D_1,D_2)$ is a quasi-isomorphism, i.e., 
$$
R\Hom_{\phi,N}(D_1,D_2)= \left[\begin{aligned}
\xymatrix{\Hom_{K_0}(D_1,D_2)\ar[r]^{\delta_1}\ar[d]^{\delta_2} &  \Hom_{K_0}(D_1,\phi_*D_2))\ar[d]^{\delta_2^{\prime}}\\
\Hom_{K_0}(D_1,D_2)\ar[r] ^{\delta_1^{\prime}}& \Hom_{K_0}(D_1,\phi_*D_2))
}
\end{aligned}
\right]
$$
Here
\begin{align*}
\delta_1: x\mapsto \phi_2x-x\phi_1,\quad \delta^{\prime}_1: x\mapsto p\phi_2 x - x \phi_1;\\
\delta_2: x\mapsto N_2x-xN_1,\quad \delta^{\prime}_2: x\mapsto N_2 x - px N_1
\end{align*}
The outside bracket denotes the homotopy limit (the total complex of the double complex in this case).
  \end{lemma}
  \begin{proof}
  We pass first to the  category of  $(\phi,N)$-modules defined as above but with modules of any rank and with  the Frobenius being just an endomorphism. It suffices to prove our lemma in this (larger) category.  The forgetful functor $M_{K_0}(\phi,N)\to V_{K_0}$ has a right adjoint $M\mapsto M_{\phi,N}$, where $M_{\phi,N}$ is the product $\prod _{i\geq 0}M^i_{\phi}, $ $M^i_{\phi}=M_{\phi}$. We will often write $M_{\phi,N}=\prod_{i,j\geq 0}M_{ij}$. The $K_0$-action on $M_{\phi,N}$ is twisted by Frobenius: $am_{ij}=\phi^i(a)m_{ij}$, $a\in K_0$; the Frobenius is defined by $\phi=\id: M_{ij}\to \phi_*M_{i,j-1}$ and the monodromy -- as $N=p^{-j}: M_{ij}\to M_{i-1,j}$.  The functor $M\mapsto M_{\phi,N}$ is exact and preserves injectives. We have the adjunction map $M\to M_{\phi,N}$, $m\mapsto (m_{ij}=N^i\phi^j(m))$ that induces the isomorphism $\Hom_{K_0}(T,M)\stackrel{\sim}{\to}\Hom_{\phi,N}(T,M_{\phi,N})$, $T\in V_{K_0}, M\in M_{K_0}(\phi,N)$.
  
     Consider the following injective resolution of
  a $(\phi, N)$-module $M$
  $$ 0\to M\stackrel{}{\to}M_{\phi,N}\stackrel{\alpha}{\to}\phi_*M_{\phi,N}\oplus M_{\phi,N}\stackrel{\beta}{\to}\phi_*M\to 0
  $$
  with 
  \begin{align*}
  \alpha(m_{ij}) & =(\phi(m_{ij})-m_{i,j+1},Nm_{ij}-p^{-j}m_{i+1,j}), \\
  \beta(x,y) & =((i,j)\mapsto Nx_{ij}-p^{1-j}x_{i+1,j}-p\phi y_{i,j}+y_{i,j+1})
  \end{align*}
   Hence 
  $R\Hom_{\phi,N}(T,M)$ for $M,T\in M_{K_0}(\phi,N)$,  can be computed by the complex
  $$
  \Hom_{\phi,N}(T,M_{\phi,N})\stackrel{\alpha}{\to}\Hom_{\phi,N}(T,\phi_*M_{\phi,N})\oplus\Hom_{\phi,N}(T,M_{\phi,N})\stackrel{\beta}{\to}\Hom_{\phi,N}(T,\phi*M_{\phi,N})
  $$
 Passing to the $K_0$-linear morphisms we get the complex $\Hom^{\sharp}(T,M)$. This suffices to prove  the lemma. 
  \end{proof}
\end{num}
\begin{num}

    A filtered $(\phi,N)$-module  is a tuple $(D_0,\phi,N,F^{\scriptscriptstyle\bullet})$, where $(D_0,\phi,N)$ is a $(\phi, N)$-module and
    $F^{\scriptscriptstyle\bullet} $  is a decreasing finite filtration of $D_K:=D_0\otimes_{K_0}K$ by $K$-vector
spaces. There is a notion of a {\em (weakly) admissible} filtered $(\phi, N)$-module \cite{CF}. Denote by
$$MF^{\ad}_K(\phi,N)\subset MF_K(\phi,N)\subset M_{K_0}(\phi,N)$$ the categories of admissible filtered $(\phi, N)$-modules, 
filtered $(\phi, N)$-modules, and $(\phi, N)$-modules, respectively. We know \cite{CF} that the
pair of functors 
\begin{align*}
D_{\st}(V) & =(\B_{\st}\otimes_{{\mathbf Q}_p}V)^{G_K},\quad D_{K}(V)=(\B_{\dr}\otimes_{{\mathbf Q}_p}V)^{G_K};\\
 V_{\st}(D) & =(\B_{\st}\otimes_{K_0}D_0)^{\phi=\id,N=0}\cap F^0(\B_{\dr}\otimes_{K}D_K)
 \end{align*}
defines an equivalence of categories $MF_K^{\ad}(\phi,N)\simeq \Rep_{\st}(G_K)\subset\Rep(G_K)$, where the last two categories denote the subcategory of semistable Galois
representations \cite{F2} of the category of finite dimensional $\Qp$-linear representations of the Galois group $G_K$. The rings $\B_{\st}$ and $\B_{\dr}$ are the semistable and de Rham period rings of Fontaine \cite{F2}. 
The category $MF^{\ad}_K(\phi,N)$   is naturally a Tannakian tensor ${\mathbf Q}_p$-category and $(D_0,\phi,N,F^{\scriptscriptstyle\bullet})\mapsto D_0$ is a fiber functor over $K_0$.

 A filtered $(\phi,N,G_K)$-module is a  tuple $(D_0,\phi,N,\rho, F^{\scriptscriptstyle\bullet})$, where
\begin{enumerate}
\item $D_0$ is a  finite dimensional $K_0^{\nr}$-vector
space;
\item  $\phi : D_0 \to D_0$ is a Frobenius  map;
\item $N : D_0 \to D_0$ is  a $K_0^{\nr}$-linear monodromy map such that $N\phi = p\phi N$;
\item  $\rho$ is a $K_0^{\nr}$-semilinear $G_K$-action on $D_0$  (hence $\rho|I_K$ is linear) that factors through a finite quotient of the inertia $I_K$ and that commutes with $\phi$ and $N$;
 \item  $F^{\scriptscriptstyle\bullet}$ is a decreasing finite filtration  of $D_K:=(D\otimes _{K_0^{\nr}}\ovk)^{G_K}$ by $K$-vector
spaces.
\end{enumerate}
Morphisms between filtered $(\phi,N, G_K)$-modules are $K_0^{\nr}$-linear maps preserving all
structures. There is a notion of a {\em (weakly) admissible} filtered $(\phi, N,G_K)$-module \cite{CF}, \cite{F3}. 
Denote by $$DF_K:=MF_K^{\ad}(\phi,N,G_K)\subset MF_K^{}(\phi,N,G_K)\subset M_K(\phi,N,G_K)$$ the categories of
admissible filtered $(\phi, N,G_K)$-modules ($DF$ stands for Dieudonn\'e-Fontaine), filtered  $(\phi, N,G_K)$-modules, and $(\phi, N,G_K)$-modules, respectively.
The last category is built from tuples $(D_0,\phi,N,\rho)$ having properties 1, 2, 3, 4 above. We know \cite{CF} that the
pair of functors 
\begin{align*}
D_{\pst}(V) & =\injlim_H(\B_{\st}\otimes_{{\mathbf Q}_p}V)^{H},\quad
H\subset G_K - \text{an open subgroup,}\quad D_K(V):=(V\otimes_{\Qp}\B_{\dr})^{G_K}; \\
V_{\pst}(D) & =(\B_{\st}\otimes_{K_0^{\nr}}D_0)^{\phi=\id,N=0}\cap F^0(\B_{\dr}\otimes_{K}D_K)
\end{align*}
define an equivalence of categories $MF_K^{\ad}(\phi,N,G_K)\simeq \Rep_{\pst}(G_K)$, where the last
category denotes the category of potentially semistable Galois representations \cite{F2}. 
We have the abstract period isomorphisms
\begin{equation}
\label{abstract}
\rho_{\pst}: D_{\pst}(V)\otimes_{K_0^{\nr}}\B_{\st}\simeq V\otimes_{\Qp}\B_{\st},\quad \rho_{\dr}: D_{K}(V)\otimes_{K}\B_{\dr}\simeq V\otimes_{\Qp}\B_{\dr},
\end{equation}
where the first one is compatible with the action of $\phi,N$, and $G_K$, and the second one is compatible with filtration. 
The category $MF_{K}^{\pst}$   is naturally a Tannakian tensor ${\mathbf Q}_p$-category and 
$(D_0,\phi,N,\rho,F^{\scriptscriptstyle\bullet})\mapsto D_0$ is a fiber functor over $K^{\nr}_0$. We will denote by 
$\sd^b(DF_K)$ and $D^b(DF_K)$ its bounded derived dg category and bounded derived category, respectively.

   The category $M_K(\phi,N,G_K)$ is abelian. We will denote by $\sd^b_K(\phi,N,G_K)$ and $D^b_K(\phi,N,G_K)$ its bounded derived dg category and bounded derived category, 
   respectively. For $(\phi, N,G_K)$-modules $M,T$, let  $\Hom_{\phi,N,G_K}(M,T)$ be the group of $(\phi,N,G_K)$-module morphisms and let
   $\Hom_{G_K}(M,T)$  be the group of $K^{\nr}_0$-linear and $G_K$- equivariant morphisms. Let $\Hom^{\sharp}_{\phi, N, G_K}(M,T)$ be the complex
  $$\Hom_{G_K}(M,T)  \to \Hom_{G_K}(M,\phi_*T)\oplus \Hom_{G_K}(M,T)\to \Hom_{G_K}(M,\phi_*T).$$
  This complex is supported in degrees $0,1,2$ and the differentials are as above for $(\phi, N)$-modules. 
  Clearly, we have $\Hom_{\phi,N,G_K}(M,T)=H^0\Hom^{\sharp}_{\phi,N,G_K}(M,T)$.  Complexes $\Hom^{\sharp}_{\phi,N,G_K}$ compose naturally.  
   Arguing as in the proof of Lemma \ref{diagram1}, we can show that, for $M,T\in\sd^b_K(\phi,N,G_K)$, 
   \begin{equation}
    \label{isom1}
\R\Hom_{\phi,N,G_K}(M,T)\simeq \Hom_{\phi,N,G_K}^{\sharp}(M,T).
\end{equation} 
        
    Let $M$, $T$ be two complexes in $C^b(MF_K(\phi,N,G_K))$. Define the complex   $\Hom^{\flat}(M,T)$ as the following homotopy fiber
$$
 \Hom^{\flat}(M,T):=\Cone(\Hom^{\sharp}_{\phi, N, G_K}(M_0,T_0)\oplus  \Hom_{\dr}(M_K,T_K)\veryverylomapr{\can-\can}\Hom_{G_K}(M_{\ovk},T_{\ovk}) )[-1]
$$
Complexes $\Hom^{\flat}$ compose naturally.  
\begin{proposition}
\label{comp1}
 We have $\R\Hom_{DF_K}(M,T)\simeq \Hom_{}^{\flat}(M,T).$ 
 \end{proposition}
 \begin{proof}
 We follow the method of proof of Beilinson and Bannai \cite[Lemma 1.7]{BE0}, \cite[Prop. 1.7]{Ban}. Denote by $f_{M,T}$ the morphism 
 in the cone defining $\Hom_{}^{\flat}(M,T).$ We have the distinguished triangle
 $$\ker(f_{M,T})\to \Hom_{}^{\flat}(M,T)\to  \coker(f_{M,T})[-1]
 $$
We also have the functorial isomorphism 
$$\Hom_{K^b(DF_K)}(M,T[i])\stackrel{\sim}{\to}H^i(\ker(f_{M,T}))
$$ 
Hence a long exact sequence
$$\to H^{i-2}(\coker(f_{M,T}))\to \Hom_{K^b(DF_K)}(M,T[i])\to H^i(\Hom_{}^{\flat}(M,T))\to H^{i-1}(\coker(f_{M,T}))\to 
$$

  Let $I_T$ be the category whose objects are quasi-isomorphisms $s:T\to L$ in $K^b(DF_K)$ and whose morphisms are morphisms $L\to L^{\prime}$ in $K^b(DF_K)$ compatible with $s$. 
Since $\injlim_{I_T}\Hom_{K^b(DF_K)}(M,L[i])=\Hom_{D(DF_K)}(M,T[i]) $, it suffices to show that 
$\injlim_{I_T}H^i(Hom_{}^{\flat}(M,L))=H^i(Hom_{}^{\flat}(M,T))$ and that $\injlim_{I_T}H^{i}(\coker(f_{M,L}))=0$. The first fact follows from Lemma \ref{diagram1} and the second one from the  Lemma \ref{help}  below.
\end{proof}
 \begin{lemma}
 \label{help}
 Let $u\in \Hom^j_{G_K}(M_{\ovk},T_{\ovk})$. There exists a complex $E\in C^b(DF_K)$ and a quasi-isomorphism $T\to E$ such that the image of $u$ in the cokernel of the map $f$ is zero.
 \end{lemma}
 \begin{proof}We will construct an extension
 $$0\to T\to E\to \Cone(M\stackrel{\id}{\to} M)[-j-1]\to 0
 $$
 in the category of filtered $(\phi,N,G_K)$-modules. Since the category of admissible modules is closed under extension, $E$ will be admissible.
 The underlying complex of $K_0^{\nr}$-vector spaces is
 $$E_0:=\Cone(M_0[-j-1]\stackrel{(0,\id)}{\to} T_0\oplus M_0[-j-1])
 $$
 The Frobenius, monodromy operator, and Galois action are defined on $E_{0}^{i+j}:=T^{i+j}_0\oplus M_0^{i-1}\oplus M^i_0$ coordinatewise. The filtration on
  $E^{i+j}_{K}:=E^{i+j}_0\otimes_{K^{\nr}_0}\ovk$ is defined as 
 \begin{align*}
 F^nE^{i+j}_{K}=F^nT^j_{K} & \oplus \{(u^i(x),0,x)|x\in F^nM^i_{K}\}\\
  & \oplus \{(d_T(u^{i-1}(x)),-x,-d_M(x))|x\in F^nM_{K}^{i-1}\}
 \end{align*}
  Now take $\xi=(0,0,\id)+(u^i,0,\id)\in  \Hom^{\sharp}_{\phi, N, G_K}(M^i_0,E^{i+j}_0)\oplus  \Hom_{\dr}(M^i_K,E^{i+j}_K) $. We have $f(\xi)=(u^i,0,0)$, as wanted.
 \end{proof}
 \end{num}
\subsection{The category of $p$-adic Hodge complexes}
\begin{num} Let $V_{\ovk}^G$ be the category of $\ovk$-vector spaces with a smooth $\ovk$-linear action of $G_K$.  It is a Grothendieck abelian category. We will consider the following functors:
\begin{itemize}
\item $F_{\dr}:V_{\dr}^K \rightarrow V_{\ovk}^G$,
which to a filtered $K$-vector space $(E,F\kr)$ associates
 the $\ovk$-vector space $E \otimes_K \ovk$ with its natural action of $G_K$.
\item $F_0:M_K(\phi,N,G_K) \rightarrow V_{\ovk}^G$, which to a $(\phi,N,G_K)$-module
 $M$ associates the $\ovk$-vector space $M \otimes_{K_0^{\nr}} \bar K$
 whose $G_K$-action is induced by the given $G_K$-action on $M$.
\end{itemize}
Both functors are exact and monoidal. Note in particular that
 they induces functors on the respective categories of complexes
 which are dg-functors.
\end{num}
\begin{num}\label{df:DpH}
 Let $\sd^b(V_{\ovk}^G)$ and $D^b(V_{\ovk}^G)$ denote the bounded derived dg category and the bounded derived category of $V_{\ovk}^G$, respectively.
We define the dg category $\sd_{pH}$ of $p$-adic Hodge complexes as the homotopy limit 
$$\sd_{pH}:=\holim(\sd^b(M_K(\phi,N,G_K))\stackrel{F_0}{\to} \sd^b(V_{\ovk}^G)\stackrel{F_{\dr}}{\leftarrow}\sd^b(V_{\dr}^K))
$$
We denote by  $D_{pH}$  the homotopy category of $\sd_{pH}$.
By \cite[Def. 3.1]{Ta}, \cite[4.1]{BB}, an object of $\sd_{pH}$ consists of objects $M_{0}\in \sd^b(M_K(\phi,N,G_K))$,
$M_{K}\in \sd^b(V^{K}_{\dr})$,  and a quasi-isomorphism
$$
F_0(M_{0})\stackrel{a_M}{\to}  F_{\dr}(M_{K})$$
 in $\sd(V^G_{\ovk})$. We will denote the object above by 
$
M=(M_0, M_{K}, a_M).
$
 The morphisms are given by the complex $ \Hom_{\sd_{pH}}((M_0,M_{K},a_M),(N_0,N_{K},a_N))$:
\begin{equation}
\begin{split}
 &\Hom^i_{\sd_{pH}}((M_0,M_{K},a_M),(N_0,N_{K},a_N)) \\
  & \qquad\quad=\Hom^i_{\sd^b(M_K(\phi,N,G_K))}(M_0,N_0)\oplus \Hom^i_{\sd^b(V^K_{\dr})}(M_{K},N_{K})\oplus \Hom^{i-1}_{\sd^b(V_{\ovk}^G)}(F_0(M_0),F_{\dr}(N_{K}))
\end{split}
\end{equation}
 The differential is given by
 $$
 d(a,b,c)=(da,db,dc+a_NF_0(a)-(-1)^{i}F_{\dr}(b)a_M)
 $$
and the composition 
 \begin{align}
 \label{hom1}
    \Hom_{\sd_{pH}}((N_0,N_{K},a_N), & (T_0,T_{K},a_T))  \otimes  \Hom_{\sd_{pH}}((M_0,M_{K},a_M),(N_0,N_{K},a_N))\\
    & \to  \Hom_{\sd_{pH}}((M_0,M_{K},a_M),(T_0,T_{K},a_T))\notag
 \end{align}
 is given by
 $$
 (a^{\prime},b^{\prime},c^{\prime})(a,b,c)=(a^{\prime}a,b^{\prime}b,c^{\prime}F_0(a)+F_{\dr}(b^{\prime})c)
 $$
It  now follows easily that a (closed) morphism 
$(a,b,c)\in \Hom_{\sd_{pH}}((M_0,M_{K},a_M),(N_0,N_{K},a_N))$ is a quasi-isomorphism if and only so are the morphisms $a$ and $b$ (see \cite[Lemma 4.2]{BB}).

By definition, we get a commutative square of  dg categories over $\Qp$:
\begin{equation} \label{eq:specialization_pH}
\xymatrix@=10pt{
\sd_{pH}\ar^{T_{\dr}}[r]\ar_{T_{0}}[d] & \sd^b(V_{\dr}^K)\ar^{F_{\dr}}[d] \\
\sd^b(M_K(\phi,N,G_K))\ar^-{F_0}[r] & \sd^b(V_{\ovk}^G).
}
\end{equation}
Given a $p$-adic Hodge complex $M$,
 we will call $T_{\dr}(M)$ (resp. $T_0(M)$) the \emph{generic fiber}
 (resp. \emph{special fiber}) of $M$.
 As pointed out above, a morphism $f$ of $p$-adic Hodge complexes is a quasi-isomorphism
 if and only if $T_{\dr}(f)$ and $T_{0}(f)$ are quasi-isomorphisms.
\end{num}
\begin{num}

  Let us recall that,
 since the category $\sd_{pH}$ is obtained by gluing, it
 has a canonical $t$-structure \cite[Prop. 4.1.12]{HMF}.
We will denote by $\sd_{pH,\leq 0}$
 (resp. $\sd_{pH,\geq0}$) the full  dg subcategory of $\sd_{pH}$ made of non-positive (resp. non negative) $p$-adic Hodge complexes.
 Let $M$ be a $p$-adic Hodge complex.
 We define its non positive truncation $\tau_{\leq 0}(M)$ according to
 the following formula:
$$
\tau_{\leq 0}(M):=(\tau_{\leq 0}M_0,\tau_{\leq 0}M_{K},
\tau_{\leq 0}a_M).
$$
The functors $F_{\dr}$ and $F_0$ being exact, this is indeed
 a $p$-adic Hodge module. The non negative truncation is obtained
 using the same formula. According to this definition,
 we get a canonical morphism of $p$-adic Hodge complexes:
$$
\tau_{\leq 0}(M) \rightarrow M
$$
whose cone is positive.
This is all we need to get that the
 pair $(\sd_{pH,\leq 0},\sd_{pH,\geq 0})$ forms a $t$-structure
 on $\sd_{pH}$.
\begin{definition}
The $t$-structure $(\sd_{pH,\leq 0},\sd_{pH,\geq 0})$ defined above
 will be called the canonical $t$-structure on $\sd_{pH}$.
\end{definition}
\end{num}

\begin{num}
 Let  $M\in C^b(MF_K(\phi,N,G_K))$. Define $\theta(M)\in \sd_{pH}$ to be the object
$$
\theta(M):=(M_0,M_K,\id_M:M_{\ovk}\simeq M_{\ovk})
$$
Since $\theta$ preserves quasi-isomorphisms, it induces a functor
$$
\theta: \quad\sd^b(DF_K)\to \sd_{pH}
$$
This is a functor between dg categories compatible with the $t$-structures. 
Through this functor we can regard $MF_K(\phi,N,G_K)$ as  a subcategory  of the heart of the $t$-structure on $\sd_{pH}$. 
\begin{lemma}
\label{heart1}
The natural functor
$$
\theta:\quad MF_K(\varphi,N,G_K) {\rightarrow }\sd_{pH}^{\heartsuit}
$$is fully faithful.
\end{lemma}
\begin{proof}
Analogous to \cite[Prop. 4.1.12]{HMF}, \cite[1.2.27]{Sn}.
\end{proof}
\begin{definition}
We will say that a strict $p$-adic Hodge complex $M$
 is \emph{admissible} if its cohomology filtered $(\phi,N,G_K)$-modules
 $H^n(M)$ are (weakly) admissible. Denote by $\sd_{pH}^{\ad}$ the full dg subcategory of $\sd_{pH}$ of admissible $p$-adic Hodge complexes.
\end{definition}
The functor $\theta$ induces a canonical functor:
$$
\theta:\quad \sd^b(DF_K) \rightarrow \sd^{\ad}_{pH}.
$$
\begin{lemma}
\label{heart2}
The natural functor
$$
\theta:\quad DF_K \stackrel{\sim}{\rightarrow }\sd_{pH}^{\ad,\heartsuit}
$$is an equivalence of abelian categories.
\end{lemma}
\begin{proof}
By definition, a strict $p$-adic Hodge complex $M$ is in the heart of the $t$-structure
 if and only if $M$ is isomorphic to $\tau_{\leq 0} \tau_{\geq 0}(M)$.
 According to the formula for this truncation,
 we get that $M$ is isomorphic to an object $\widetilde M$ such that
 $\widetilde M_0$ is a $(\varphi,N,G_K)$-module, $\widetilde M_{K}$ is
 a filtered $K$-vector space, and one has a
 $G_K$-equivariant isomorphism
$$
\widetilde M_{0} \otimes_{K_0^{\nr}} \ovk
 \simeq \widetilde M_{K} \otimes_{K} \ovk.
$$
In particular, $\widetilde M_{0}$ has the structure of a filtered
 $(\varphi,N,G_K)$-module.
Conversely, it is clear that a filtered $(\varphi,N,G_K)$-module
 induces a $p$-adic Hodge complex concentrated in degree $0$. This proves our lemma.
\end{proof}
\begin{theorem}
\label{thm1}
The functor $\theta$ induces an equivalence of dg categories
$$
\theta: \quad\sd^b(DF_K) \stackrel{\sim}{\rightarrow} \sd^{\ad}_{pH}
$$
\end{theorem}
\begin{proof}
Since, by Lemma \ref{heart2},  we have the equivalence of abelian categories
$$
\theta: \quad DF_K=\sd^b(DF_K)^{\heartsuit}\stackrel{\sim}{\to}\sd^{\ad,\heartsuit}_{pH}
$$
it suffices to show that, given two complexes $M$, $M'$ of $C^b(MF_K^{pst})$,
 the functor $\theta$ induces a quasi-isomorphism:
$$
\theta:\quad \Hom_{\sd^b(DF_K)}(M,M')
 \rightarrow \Hom_{\sd_{pH}}(\theta(M),\theta(M'))
$$ By (\ref{isom1}) and Proposition \ref{comp1}, since $F_0(M_0)=F_{\dr}(M_K)=M_{\ovk}, F_0(M^{\prime}_0)=F_{\dr}(M^{\prime}_K)=M_{\ovk}$, 
we have the following sequence of quasi-isomorphisms 
 \begin{align*}
     \Hom_{\sd_{pH}}(\theta(M),\theta(M')) & =\Hom_{\sd_{pH}}((M_0,M_{K},\id_M)  ,(M^{\prime}_0,M^{\prime}_{K},\id_{M^{\prime}}))\\
      & \simeq (\Hom_{\sd^b(M_K(\phi, N, G_K))}(M_0,M^{\prime}_0)\stackrel{F_0}{\to}  \Hom_{\sd^b(V^G_{\overline{K}})}(M_{\ovk},M^{\prime}_{\ovk}) \stackrel{F_{\dr}}{\leftarrow}  \Hom_{\sd^b(V^K_{\dr})}(M_{K},M^{\prime}_{K}))\\
 & \simeq (\Hom^{\sharp}_{\phi, N, G_K}(M_0,M^{\prime}_0)\stackrel{F_0}{\to}  \Hom_{G_K}(M_{\ovk},M^{\prime}_{\ovk}) \stackrel{F_{\dr}}{\leftarrow}  \Hom_{\dr}(M_{K},M^{\prime}_{K}))\\
 & \simeq \Hom_{\sd^b(DF_K)}(M,M^{\prime}).
 \end{align*}
 This concludes our proof.
\end{proof}
\end{num}
\subsection{The absolute $p$-adic Hodge cohomology}
\begin{num}
Any potentially semistable $p$-adic representation is a $p$-adic
 Hodge complex. Therefore, we can define the Tate twist in $\sd_{pH}$
 as follows: given any integer $r \in \bz$, we let
 $K(r)$ be the $p$-adic Hodge complex $$K(-r)=(K_0^{\nr}, K,\id_K: \ovk\stackrel{\sim}{\to} \ovk )$$
that is equal to $K_0^{\nr}$ and $K$ concentrated in degree $0$;
the Frobenius is $\phi_{K(-r)}(a)=p^r\phi(a)$, the Galois action is canonical and the monodromy operator is zero;
 the filtration is $F^i=K$ for $i\leq r$ and zero otherwise.
% Note that $K(0)$ is just the unit for the monoidal structure
% on $\sd_{pH}$ ????.

As usual, given any $p$-adic Hodge complex $M$,
 we put $M(r):=M\otimes K(r)$. In other words,
 twisting a $p$-adic Hodge complex $r$-times
 divides the Frobenius by $p^r$, leaves unchanged the
 monodromy operator, and shifts the filtration $r$-times.
\end{num}

\begin{example}
\label{ex1}
Given any $p$-adic Hodge complex $M$, by formula (\ref{hom1}) and by (\ref{isom1}),
 we have the quasi-isomorphism of complexes of $\Qp$-vector spaces
 $$\Hom_{\sd_{pH}}(K(0),M(r))\simeq 
\Cone(M^{\sharp}_0 \oplus F^rM_K\verylomapr{a_M-\can}F_{\dr}(M_K)^{G_K})[-1],
$$
where $M^{\sharp}_0$ is defined as the homotopy limit of the following diagram (we set $\phi_i:=\phi/p^i$)
$$
M^{\sharp}_0:=\left[\begin{aligned}\xymatrix{
 M^{G_K}_0\ar[r]^{1-\phi_r}\ar[d]_N & M^{G_K}_0\ar[d]^N\\
M^{G_K}_0\ar[r]^{1-\phi_{r-1}}  & M^{G_K}_0}
\end{aligned}\right]
$$
\end{example}
\begin{num}
Let $X$ be a variety over $K$. Consider the following complex in $\sd^{\ad}_{pH}$
$$
\R\Gamma_{pH}(X_{\ovk},0):=( \R\Gamma^B_{\hk}(X_{\ovk}), (\R\Gamma_{\dr}(X),F^{\jcdot}), \R\Gamma^B_{\hk}(X_{\ovk})\stackrel{\iota_{\dr}}{\longrightarrow}\R\Gamma_{\dr}(X_{\ovk}))
$$
Here $\R\Gamma^B_{\hk}(X_{\ovk})$ is the (geometric) Beilinson-Hyodo-Kato cohomology \cite{BE1}, \cite[3.3]{NN}; by definition it is a bounded 
complex of $(\phi, N, G_K)$-modules. The filtered complex $\R\Gamma_{\dr}(X)$ is the Deligne de Rham cohomology. The map $\iota_{\dr}$ is the Beilinson-Hyodo-Kato map \cite{BE1} that induces a quasi-isomorphism 
$$\iota_{\dr}: \quad \R\Gamma^B_{\hk}(X_{\ovk})\otimes_{K_0^{\nr}}\ovk \stackrel{\sim}{\to}\R\Gamma_{\dr}(X_{\ovk}).$$ The comparison theorems of $p$-adic Hodge theory (proved in \cite{Fa1}, \cite{Ts1}, \cite{Ni2}, \cite{BE1}, \cite{Bh}) imply that the $p$-adic Hodge complex $\R\Gamma_{pH}(X_{\ovk},0)$ is admissible. 

  We will denote by $$\R\Gamma_{pH}(X_{\ovk},r):=\R\Gamma_{pH}(X_{\ovk},0)(r)\in \sd^{\ad}_{pH}$$ the  $r$'th Tate twist of $\R\Gamma_{pH}(X_{\ovk},0)$. We will call it
{\em the geometric $p$-adic Hodge cohomology of $X$}. It is a dg $\Q_p$-algebra. Since the Beilinson-Hyodo-Kato map is a map of dg $K_0^{\nr}$-algebras,
the assignment
$X\mapsto \R\Gamma_{pH}(X_{\ovk},r)$ is a presheaf of dg $\Q_p$-algebras on ${\mathcal V}ar_K$ and we also have the external product
$\R\Gamma_{pH}(X_{\ovk},r)\otimes \R\Gamma_{pH}(Y_{\ovk},s)$ in $\sd_{pH}^{\ad}$.
\begin{lemma}[K\"{u}nneth formula]\label{lm:kunneth_abs_synt}
The natural map
$$
\R\Gamma_{pH}(X_{\ovk},r)\otimes \R\Gamma_{pH}(Y_{\ovk},s)\stackrel{\sim}{\to }\R\Gamma_{pH}(X_{\ovk}\times Y_{\ovk},r+s)
$$ is a quasi-isomorphism.
\end{lemma}
\begin{proof}
This follows easily from the K\"{u}nneth formulas in the filtered de Rham cohomology and the Hyodo-Kato cohomology (use the Hyodo-Kato map to pass to de Rham cohomology).
\end{proof}
Set 
\begin{align*}
\R\Gamma_{DF_K}(X_{\ovk},r) & :=\theta^{-1}\R\Gamma_{pH}(X_{\ovk},r)\in\sd^b(DF_K),\\
\R\Gamma_{\pst}(X_{\ovk},r) & :=V_{\pst}\theta^{-1}\R\Gamma_{pH}(X_{\ovk},r)\in \sd^b(\Rep_{\pst}(G_K)).
\end{align*}
\begin{lemma}
\label{pst=et}
There exists a canonical quasi-isomorphism in $\sd^b(\Rep(G_K))$
$$\R\Gamma_{\pst}(X_{\ovk},r) \simeq \R\Gamma_{\eet}(X_{\ovk},\Q_p(r)). 
$$
\end{lemma}
\begin{proof}
To start, we note that we have the following commutative diagram  of dg categories. 
$$
\xymatrix{
\sd^b(\Rep_{\pst}(G_K))\ar[d]^{D_{\pst}}\ar[rr]^{\can} & &\sd^b(\Rep(G_K))\ar[d]^{\can}\\
\sd^b(DF_K)\ar[u]^{V_{\pst}}\ar[r]^{\theta} & \sd^{\ad}_{pH}\ar[r]^-{r_{\eet}} & \sd(\Spec(K)_{\proeet})
}
$$
 Here the functor
$$
r_{\eet}:\quad   \sd^{\ad}_{pH}\to\sd(\Spec(K)_{\proeet})
$$
associates  to a $p$-adic Hodge complex $(M_0,M_K,a_M: F_0(M_0)\to F_{\dr}(M_K))$
  the complex
\begin{align*}
& [[M_0\otimes_{K_0^{\nr}}\B_{\st}]^{\phi=\id,N=0} \oplus F^0(M_K\otimes _K\B_{\dr})
        \xrightarrow{a_M\otimes\iota-\can\otimes \iota} F_{\dr}(M_K)\otimes_{\ovk}\B_{\dr}]\\
&=[[M_0\otimes_{K_0^{\nr}}\B_{\st}]^{\phi=\id,N=0} 
        \xrightarrow{a_M\otimes\iota} (F_{\dr}(M_K)\otimes_{\ovk}\B_{\dr})/F^0]  
\end{align*}
where $\iota: \B_{\st}\hookrightarrow \B_{\dr}$ is the canonical map of period rings.
To see that the diagram commutes, recall that we have the fundamental exact sequence
\begin{equation}
\label{fund}
0\to \Qp\to \B_{\st}^{\phi=\id, N=0}\oplus F^r\B_{\dr}\stackrel{\iota}{\to} \B_{\dr}\to 0
\end{equation}
It follows that, for $V\in \sd^b(\Rep_{\pst}(G_K))$,  we have a canonical morphism
\begin{align*}
V & \simeq [V\otimes_{\Qp}\B_{\st}^{\phi=\id,N=0} \oplus V\otimes _{\Qp}F^0\B_{\dr}
        \xrightarrow{\id\otimes\iota-\can\otimes \iota} V\otimes_{\Qp}\B_{\dr}]\\
         & \simeq [[V\otimes_{\Qp}\B_{\st}]^{\phi=\id,N=0} \oplus V\otimes _{\Qp}F^0\B_{\dr}
        \xrightarrow{\id\otimes\iota-\can\otimes \iota} V\otimes_{\Qp}\B_{\dr}]\\
               & \verylomapr{(\rho_{\pst}\oplus\rho_{\dr},\rho_{\dr})}[[D_{\pst}(V)\otimes_{K_0^{\nr}}\B_{\st}]^{\phi=\id,N=0} \oplus F^0(D_K(V)\otimes _K\B_{\dr})
        \xrightarrow{a_M\otimes\iota-\can\otimes \iota} D_K(V)\otimes_{\ovk}\B_{\dr}]\\
        & \simeq r_{\eet}\theta D_{\pst}(V)
        \end{align*}
        Since the abstract period morphisms $\rho_{\pst}, \rho_{\dr}$ from (\ref{abstract}) are isomorphisms, the above morphism is a quasi-isomorphism and we have the commutativity we wanted. 
        
        The above diagram gives us the first quasi-isomorphism in the following formula.
 $$\R\Gamma_{\pst}(X_{\ovk},r) \simeq r_{\eet}\R\Gamma_{pH}(X_{\ovk},r)\simeq \R\Gamma_{\eet}(X_{\ovk},\Q_p(r)).
  $$
   It suffices now to prove the second quasi-isomorphism. But, we have 
$$\R\Gamma_{pH}(X_{\ovk},r)=(\R\Gamma_{\hk}^B(X_{\ovk},r), (\R\Gamma_{\dr}(X), F^{\jcdot+r}), \R\Gamma_{\hk}^B(X_{\ovk},r)\otimes_{K_0^{\nr}}\ovk\stackrel{\iota_{\dr}}{\to} \R\Gamma_{\dr}(X_{\ovk})),$$
where we twisted the Beilinson-Hyodo-Kato cohomology to remember the Frobenius twist.
 Recall that Beilinson has constructed period morphisms (of dg-algebras)  \cite[3.6]{BE00}, \cite[3.2]{BE1} \footnote{We will be using consistently Beilinson's definition of the period maps. It is likely that the uniqueness criterium stated in \cite{Ni3} can be used to show that these maps coincide with the other existing ones.}
 \begin{align*}
 \rho_{\pst}: & \quad \R\Gamma^B_{\hk}(X_{\ovk})\otimes _{K_0^{\nr}}\B_{\st} \simeq \R\Gamma_{\eet}(X_{\ovk},\Qp)\otimes _{\Qp}\B_{\st} ,\\
 \rho_{\dr}: & \quad  \R\Gamma_{\dr}(X_{\ovk})\otimes _{\ovk}\B_{\dr}\simeq \R\Gamma_{\eet}(X_{\ovk},\Qp)\otimes_{\Qp} \B_{\dr},
 \end{align*}
 The first morphism is compatible with Frobenius, monodromy, and $G_K$-action; the second one - with filtration.
These morphisms allow us to define a quasi-isomorphism
$$\beta: \quad r_{\eet}\R\Gamma_{pH}(X_{\ovk},r)\simeq \R\Gamma_{\eet}(X_{\ovk},\Qp(r))
$$ in $\sd(\Spec(K)_{\proeet})$
as the composition
\begin{align*}
\beta:\quad r_{\eet}\R\Gamma_{pH}(X_{\ovk},r)& =[[\R\Gamma^B_{\hk}(X_{\ovk},r)\otimes _{K_0^{\nr}}\B_{\st}]^{\phi=\id,N=0}
\lomapr{\iota_{\dr}}(\R\Gamma_{\dr}(X_{\ovk})\otimes _{\ovk}\B_{\dr})/F^r]\\
 & \veryverylomapr{(\rho_{\hk},\rho_{\dr})}
[\R\Gamma_{\eet}(X_{\ovk},\Qp(r))\otimes _{\Qp}\B_{\st}^{\phi=\id,N=0}\lomapr{\iota}(\R\Gamma_{\eet}(X_{\ovk},\Qp)\otimes _{\Qp}\B_{\dr})/F^r]\\
& \stackrel{\sim}{\leftarrow} \R\Gamma_{\eet}(X_{\ovk},\Qp(r))
\end{align*}
Here the last quasi-isomorphism follows from the fundamental exact sequence (\ref{fund}). We are done.
\end{proof}
\begin{remark}
\label{geomsc}
   The geometric $p$-adic Hodge cohomology $\R\Gamma_{pH}(X_{\ovk},r)$ we work with here is not the same as the geometric syntomic  cohomology $\R\Gamma_{\synt}(X_{\ovk,h},r)$ defined in \cite{NN}. 
While the first one, by the above corollary, represents the \'etale cohomology $\R\Gamma_{\eet}(X_{\ovk},\Q_p(r)) $, the second one represents only its piece, i.e., we have $\tau_{\leq r}\R\Gamma_{\synt}(X_{\ovk,h},r) \simeq \tau_{\leq r}\R\Gamma_{\eet}(X_{\ovk},\Q_p(r))$.
 \end{remark}
\end{num}
\begin{num}\label{num:main_construction}
The {\em  $p$-adic absolute Hodge cohomology of $X$} (also called {\em syntomic cohomology} of $X$ if this does not cause confusion) is  defined as
\begin{align}
\label{def1}
\R\Gamma_{\sh}(X,r)=\R\Gamma_{\synt}(X,r):=\Hom_{\sd_{pH}}(K(0),\R\Gamma_{pH}(X_{\ovk},r)).
\end{align}
By Theorem \ref{thm1}, we have 
\begin{align*}
\R\Gamma_{\sh}(X,r) & \simeq \Hom_{\sd^b(DF_K)}(K(0),\R\Gamma_{DF_K}(X_{\ovk},r))\\
 & \simeq \Hom_{\sd^b(\Rep_{\pst}(G_K))}(\Qp,\R\Gamma_{\pst}(X_{\ovk},r)).
\end{align*}
The assignment
$X\mapsto \R\Gamma_{\sh}(X,r)=\R\Gamma_{\synt}(X,r)$ is a presheaf of dg $\Q_p$-algebras on ${\mathcal V}ar_K$.

  Set $H^{i}_{\synt}(X,r):=H^i\R\Gamma_{\synt}(X,r)$. 
\begin{theorem}\label{thm:syntomic_descent}
\begin{enumerate}
\item There is a functorial syntomic descent  spectral sequence
\begin{equation}
\label{synts}{}^{\synt}E^{i,j}:=H^i_{\st}(G_K,H^j_{\eet}(X_{\ovk},\mathbf {Q}_p(r)))\Rightarrow H^{i+j}_{\synt}(X,r),
\end{equation}
where $H^i_{\st}(G_K,\cdot)$ is the group of (potentially) semistable extensions $\Ext^i_{\Rep_{\pst}(G_K)}(\Qp,\cdot)$ as defined in \cite[1.19]{FPR}.
\item There is a functorial syntomic period morphism
$$\rho_{\synt}:\quad \R\Gamma_{\synt}(X,r)\to \R\Gamma_{\eet}(X,\Qp(r)).
$$
\item The syntomic descent spectral sequence is compatible with the Hochschild-Serre spectral sequence
\begin{equation}
\label{syntss0}
{}^{\eet}E_2^{i,j}=H^i(G_K,H^j_{\eet}(X_{\ovk},\Qp(r)))
 \Rightarrow H^{i+j}_{\eet}(X,\Qp(r)).
\end{equation}
More specifically, there is a natural map ${}^{\synt}E_2^{i,j}\to {}^{\eet}E_2^{i,j}$ that is compatible with the syntomic period map $\rho_{\synt}$.
\end{enumerate}
\end{theorem}
\begin{proof}
From the definition (\ref{def1}) of $\R\Gamma_{pH}(X_{\ovk},r)$ we obtain the following spectral sequence
$$E^{i,j}_2=\Ext^i_{\Rep_{\pst}(G_K)}(\Qp,H^j\R\Gamma_{\pst}(X_{\ovk},r))\Rightarrow H^{i+j}\R\Gamma_{\synt}(X,r)$$
Since, by Lemma \ref{pst=et}, 
we have  $\R\Gamma_{\pst}(X_{\ovk},r))\simeq  \R\Gamma_{\eet}(X_{\ovk},\Qp(r))$, 
the first statement of our theorem  follows.

   We define the syntomic period map $\rho_{\synt}: \R\Gamma_{\synt}(X,r)\to \R\Gamma_{\eet}(X,\Qp(r))$ as the composition
\begin{align*}
\rho_{\synt}:\quad \R\Gamma_{\synt}(X,r) & =
\Hom_{\sd_{pH}}(K(0),\R\Gamma_{pH}(X_{\ovk},r)))
 \stackrel{r_{\eet}}{\to }\Hom_{\sd(\Spec(K)_{\proeet})}(\Qp,r_{\eet}\R\Gamma_{pH}(X_{\ovk},r)))\\
 & \stackrel{\beta}{\to}\Hom_{\sd(\Spec(K)_{\proeet})}(\Qp,\R\Gamma_{\eet}(X_{\ovk},\Qp(r)))
=\R\Gamma_{\eet}(X,\Qp(r)).
\end{align*}
The second statement of the theorem follows.

 Finally, since the Hochschild-Serre spectral sequence
$$
{}^{\eet}E^{i,j}_2:=H^i(G_K,H^j(X_{\ovk},\bq_p(r)))\Rightarrow H^{i+j}(X,\bq_p(r))
$$
 can be identified with the spectral sequence
$${}^{\eet}E^{i,j}_2:=H^i(\Spec(K)_{\proeet},H^j(X_{\ovk},\bq_p(r)))\Rightarrow H^{i+j}(X,\bq_p(r))
$$
we get that the syntomic descent spectral sequence 
is compatible with the  Hochschild-Serre spectral sequence via the map $\rho_{\synt}$, as wanted.
\end{proof}
\begin{theorem}
\label{compsynpH}
Let $\R\Gamma_{\synt}(X_h,r) $ be the syntomic cohomology defined in \cite[3.3]{NN}.
There exists a natural quasi-isomorphism (in the classical derived category)
$$\R\Gamma_{\synt}(X_h,r)\stackrel{\sim}{\to} \R\Gamma_{\synt}(X,r), \quad r\geq 0.
$$ It is compatible with syntomic period morphisms and the syntomic as well as the \'etale descent spectral sequences.
\end{theorem}
\begin{proof}Let $r\geq 0$. 
Recall that we have a natural quasi-isomorphism \cite[Prop. 3.18]{NN}
$$
\R\Gamma_{\synt}(X_h,r)\simeq \Cone(\R\Gamma^B_{\hk}(X)^{\phi,N}\oplus F^r\R\Gamma_{\dr}(X)\verylomapr{\iota_{\dr}-\can}\R\Gamma_{\dr}(X))[-1],
$$
where
$$
\R\Gamma^B_{\hk}(X_{h})^{\phi,N}:=\left[\begin{aligned}
\xymatrix{\R\Gamma^B_{\hk}(X)\ar[r]^{1-\phi_r}\ar[d]^N& \R\Gamma^B_{\hk}(X)\ar[d]^N\\
\R\Gamma^B_{\hk}(X)\ar[r]^{1-\phi_{r-1}} & \R\Gamma^B_{\hk}(X)}
\end{aligned}
\right]
$$
and the complex $\R\Gamma^B_{\hk}(X)$ is the (arithmetic) Beilinson-Hyodo-Kato cohomology \cite{BE1} that comes equipped with the Beilinson-Hyodo-Kato 
map $\iota_{\dr}: \R\Gamma^B_{\hk}(X)\to \R\Gamma_{\dr}(X)$ \cite[3.3]{NN}. 

    Since $\R\Gamma^B_{\hk}(X)\simeq \R\Gamma^B_{\hk}(X_{\ovk})^{G_K}$ and 
    $\R\Gamma_{\dr}(X)\simeq \R\Gamma_{\dr}(X_{\ovk})^{G_K}$ by \cite[Prop. 3.20]{NN}, Example \ref{ex1} and  Theorem  \ref{thm1} yield
\begin{align*}
\R\Gamma_{\synt}(X_h,r)\simeq \Hom_{\sd_{pH}}(K(0),\R\Gamma_{pH}(X_{\ovk},r)))\simeq
 \Hom_{\sd^b(DF_K)}(K(0),\R\Gamma_{DF_K}(X_{\ovk},r)))\simeq  \R\Gamma_{\synt}(X,r),
\end{align*}
as wanted. The last claim of the theorem is now clear.
\end{proof}
\begin{remark}
The above theorems gives an alternative construction of the syntomic descent spectral sequence from \cite[4.1]{NN} (that construction used the geometric syntomic cohomology mentioned in Remark \ref{geomsc})
and an alternative proof of its compatibility with the Hochschild-Serre spectral sequence \cite[Theorem 4.8]{NN}. In the present approach the syntomic descent spectral sequence is a genuine descent spectral sequence: from  geometric \'etale cohomology to syntomic cohomology. In the approach of 
 \cite{NN} this sequence appears as a piece of a larger descent spectral sequence that remains to be understood. 
\end{remark}
\begin{remark}
\label{compact}
In everything above, the variety $X$ can be replaced by a finite simplicial scheme or a finite diagram of schemes. In particular, we obtain statements about  cohomology with compact support: use resolutions of singularities to get a compactification of the variety with a divisor with normal crossing at infinity and then represent cohomology with compact support as a cohomology of a finite simplicial scheme built from the closed strata.
In particular, we get the syntomic descent spectral sequence with compact support:
\begin{equation*}
{}^{\synt,c}E^{i,j}_2:=H^i_{\st}\big(G_K,H^j_{\eet,c}( X_{\ovk},\Qp(r))\big)
 \Rightarrow H^{i+j}_{\synt,c}(X,r)
\end{equation*}
that is compatible with the Hochschild-Serre spectral sequence for  \'etale cohomology with compact support. 
\end{remark}
\end{num}
\begin{corollary}
For $X$ smooth and proper over $K$, the syntomic descent spectral sequence (\ref{synts})  degenerates at $E_2$.
\end{corollary}
\begin{proof}The argument proceeds along standard lines \cite[Thm 1.5]{De1}. First, we treat the case of $X$ smooth and projective, of equal  dimension $d$. 
Recall that we have the Hard Lefschetz Theorem \cite[Thm 4.1.1]{De2}: If  $L \in H^2(X_{\ovk},\Qp(1))$ is   the class of a hyperplane, then 
 for $i \leq  d$, the map $L^i : H^{d-i}(X_{\ovk},\Qp) \to H^{d+i}(X_{\ovk},\Qp(i))$, $a\mapsto a \cup L^i$, is an
isomorphism. This gives us the Lefschetz primitive decomposition
 \begin{equation}
 \label{LPD}
 H^i(X_{\ovk},\Qp(r))=\oplus_{k\geq 0}L^kH^{i-2k}_{\prim}(X_{\ovk},\Qp(r-k)).
 \end{equation}

    Take $s\geq 2$. Assume that the differentials of our spectral sequence $d_2=\cdots=d_{s-1}=0$. We want to show that $d_s=0$.
 Note that  Hard Lefschetz gives us that the differentials
  $$d_s: \quad H^j_{\st}(G_K,H^{i-2k}_{\prim}(X_{\ovk},\Qp(r-k)) )\to  H^j_{\st}(G_K,H^{i-2k}(X_{\ovk},\Qp(r-k)) )$$
  are trivial and hence that so are the differentials
  $$d_s: \quad H^j_{\st}(G_K,L^kH^i_{\prim}(X_{\ovk},\Qp(r)) )\to  H^j_{\st}(G_K,H^i(X_{\ovk},\Qp(r)) ).  $$
 By (\ref{LPD}), this gives that $d_s=0$, as wanted. 
 \begin{remark}
  In fact, we have 
 the Decomposition Theorem, i.e., there is a natural quasi-isomorphism in $D^b(\Rep_{\pst}(G_K))$
 $$\bigoplus_i H^i(X_{\ovk},\Q_p)[-i]\stackrel{\sim}{\to}\R\Gamma_{\pst}(X_{\ovk},\Q_p)
 $$
 \end{remark}
 
   For a general smooth and proper variety $X$, we first use Chow Lemma and resolution of singularities to find a birational proper map $f: Y\to X$ from a smooth and projective  variety $Y$ over $K$. It suffices now to note that the maps $f^*: H^i(X_{\ovk},\Q_p)\hookrightarrow  H^i(Y_{\ovk},\Q_p)$, $i\geq 0$, are injective to get the degeneration we wanted. 
 \end{proof}
\section{A $p$-adic absolute Hodge cohomology, II: Beilinson's definition}
In this section we will describe the definition of $p$-adic absolute Hodge cohomology due to Beilinson \cite{BS}. Beilinson associates to any variety over $K$  a canonical  complex of potentially semistable representations of $G_K$ representing the geometric \'etale cohomology of the variety as a Galois module. Then he defines $p$-adic absolute Hodge cohomology of this variety as the derived $\Hom$ in the category of potentially semistable representations from  the trivial representation to this complex.
\subsection{Potentially semistable complex of a variety}
\subsubsection{Potentially semistable cellular complexes}
The Basic  Lemma of Beilinson \cite[Lemma 3.3]{B0}  allows one,   in analogy with the cellular complex for $CW$-complexes, to associate a canonical complex of potentially semistable representations of  $G_K$ to any  affine variety over $K$. Recall that the cellular
complex associated to a $CW$-complex X is a complex of singular homology groups 
\begin{equation}
\label{cellular}
\cdots \to H^B_2(X^2, X^1)\stackrel{d_2}{\to} H^B_1(X^1, X^0)\stackrel{d_1}{\to} H^B_0(X^0,\emptyset)\stackrel{d_0}{\to} 0 
\end{equation}
where $X^j$ denotes the $j$-skeleton of X. The homology of the above  complex computes 
the singular homology of $X$: 
we have  $H^B_j(X^j/X^{j-1})\simeq H^B_j(\vee_{|I|} S^j)\simeq \sum_{i\in I} e_i\bz,$ $I$ being the index set of $j$-cells in $X$.

 We will briefly sketch the construction of potentially semistable (cohomological) cellular complexes and we refer
interested reader for details to \cite{KK}, \cite{N1}, \cite{HMS}.
\begin{definition}
\begin{enumerate}
\item A {\em pair} is a  triple $(X,Y,n)$, for   a closed $K$-subvariety $Y\subset X$ of a  $K$-variety $X$ and an integer $n$.
 \item  Pair $(X,Y,n)$  is called a {\em good pair} if the relative geometric \'etale cohomology
$$
H^j(X_{\ovk},Y_{\ovk},\bq_p)=0,\quad \text{unless } j\neq n.
$$
\item A good pair is called {\em very good} if $X$ is affine and $X\setminus Y$ is smooth and either $X$ is of dimension $n$ and $Y$ of dimension $n-1$ or $X=Y$ is of dimension less than $n$. 
\end{enumerate}
\end{definition}
\begin{lemma}(Basic Lemma)
Let $X$ be an affine variety over $K$  and let $Z\subset X$ be a closed subvariety such that $\dim(Z)< \dim(X)$. Then there is a closed subvariety $Y\supset Z$ such that  $\dim(Y)< \dim(X)$ and $(X,Y,n)$, $n:=\dim(X)$,  is a good pair, i.e.,
$$
H^j(X_{\ovk},Y_{\ovk},\bq_p)=0,\quad  j\neq n.
$$
Moreover, $X\setminus Y$ can be chosen to be smooth.
\end{lemma}
\begin{proof}
 See \cite[Lemma 3.3]{B0} (a result in any characteristic), \cite{N1}, \cite[7]{HNo}.
\end{proof}
\begin{corollary}
\label{cor1}
\begin{enumerate}

\item 
Every affine variety $X$ over $K$ has a {\em cellular stratification} 
$$
F_{\jcdot}X:\quad \emptyset=F_{-1}X\subset F_0X\subset \cdots \subset F_{d-1}X\subset F_dX=X
$$
That is, a stratification  by closed subvarieties such that the triple $(F_jX,F_{j-1}X,j)$ is very good. 
\item  Celullar stratifications  of  $X$ form a filtered system.
\item Let $f:X\to Y$ be a morphism of affine varieties over $K$. Let $F_{\jcdot}X$ be a cellular stratification on $X$. Then there exists a cellular stratification  $F_{\jcdot}Y$ such that $f(F_iX)\subset F_iY$.
\end{enumerate}
\end{corollary}
\begin{proof}
See Corollary D.11, Corollary D.12 in \cite{HMS}.
\end{proof}

  Having the above facts it is easy to associate a potentially semistable analog of the cellular complex (\ref{cellular}) to an affine variety $X$ over $K$ \cite[Appendix D]{HMS}. We just 
pick a cellular stratification  $$
F_{\jcdot}X: \quad \emptyset=F_{-1}X\subset F_0X\subset \cdots\subset F_{d-1}X\subset F_dX=X
$$ and take the complex
\begin{align*}
\R\Gamma_{\pst}(X_{\ovk},F_{\jcdot}X):=
0\to H^0(F_0X_{\ovk},\bq_p) & \to \cdots\to H^j(F_jX_{\ovk},F_{j-1}X_{\ovk},\bq_p)\\
 & \stackrel{d_j}{\to}H^{j+1}(F_{j+1}X_{\ovk},F_jX_{\ovk},\bq_p)\stackrel{d_{j+1}}{\to}
\cdots \to H^d(X_{\ovk},F_{d-1}X_{\ovk},\bq_p)\to 0
\end{align*}
This is a complex of Galois modules that, by $p$-adic comparison theorems, are potentially semistable. To get rid of the choice we take the homotopy colimit over all cellular stratifications, i.e., we 
set $$\R\Gamma^T_{\pst}(X_{\ovk}):=\hocolim_{F_{\jcdot}X}\R\Gamma_{\pst}(X_{\ovk},F_{\jcdot}X)$$
 It is a complex in $\sd(\Ind-\Rep_{\pst}(G_K)) $ whose cohomology groups are in $\Rep_{\pst}(G_K)$ hence we can think of it as being in $\sd(\Rep_{\pst}(G_K))$.

   The complex  $\R\Gamma^T_{\pst}(X_{\ovk})$ computes the \'etale cohomology groups $H^*(X_{\ovk},\bq_p)$ as Galois modules. More precisely, we have the following proposition.
 \begin{proposition}( \cite[Prop. 2.1]{KK})
 \label{prop1}
 \begin{enumerate}
 \item Let $F_{\jcdot}X$ be a cellular stratification of $X$. There is a natural quasi-isomorphism 
 $$\kappa_{(X,F_{\jcdot}X)}:\quad \R\Gamma_{\pst}(X_{\ovk},F_{\jcdot}X)\simeq     \R\Gamma_{\eet}(X_{\ovk},\Q_p)$$
 that  is compatible with the action of $G_K$.
  \item
 Let $f: Y\to X$ be a map of affine schemes  and let $F_{\jcdot}Y$ be a cellular  stratification of $Y$ such that, for all $i$, $F_iY\subset F_iX$. Then 
 the following diagram commutes (in the dg derived category)
 $$
 \xymatrix{
 \R\Gamma_{\pst}(Y_{\ovk},F_{\jcdot}Y)\ar[r]^{\kappa_{(Y,F_{\jcdot}Y)}}_{\sim} &   \R\Gamma_{\eet}(Y_{\ovk},\Q_p)\\
 \R\Gamma_{\pst}(X_{\ovk},F_{\jcdot}X)\ar[u]^{f^*}\ar[r]^{\kappa_{(X,F_{\jcdot}X)}}_{\sim}  &    \R\Gamma_{\eet}(X_{\ovk},\Q_p)\ar[u]^{f^*}
  }
  $$
  \item  There exists a natural quasi-isomorphism $$\kappa_{X}:\quad \R\Gamma^T_{\pst}(X_{\ovk})\simeq    \R\Gamma_{\eet}(X_{\ovk},\Q_p)$$
   that is compatible with the action of $G_K$. 
 \end{enumerate}
 \end{proposition}
 \begin{proof}
We have the following commutative diagram of Galois equivariant morphisms 
 $$\xymatrix@C=8pt{
 H^0(F_0X_{\ovk},\bq_p)\ar[d]^{\wr}\ar[r] & \cdots \ar[r] & H^k(F_kX_{\ovk},F_{k-1}X_{\ovk},\bq_p) \ar[r] \ar[d]^{\wr}
& \cdots  \ar[r] & H^d(X_{\ovk},F_{d-1}X_{\ovk},\bq_p) \ar[d]^{\wr} \\
 \R\Gamma_{\eet}(F_0X_{\ovk},\Q_p) \ar[r] \ar[d]& \cdots\ar[r] &  [\R \Gamma_{\eet}(F_kX_{\ovk},\Q_p)\to \R\Gamma_{\eet} (F_{k-1}X_{\ovk},\Q_p)][k] \ar[r]  \ar[d]
   &\cdots\ar[r] &  [\R\Gamma_{\eet} (X_{\ovk},\Q_p)\to \R\Gamma_{\eet}( {F_{d-1}X_{\ovk}},\Q_p)] [d]\ar[d]\\
0\ar[r] & \cdots\ar[r] & 0 \ar[r]  
  &\cdots\ar[r] & \R\Gamma_{\eet}( {X_{\ovk}},\Q_p) [d]
 }$$
 Here the square brackets denote homotopy limits. The first vertical maps are the truncations $\tau_{\leq d}\tau_{\geq d}$. 
 We obtain the map $\kappa_{(X,F_{\jcdot}X)}$ from the first statement of the proposition by taking homotopy limits of the rows of the diagram. Second  statement is  now clear. The third one  is an immediate corollary of Proposition \ref{prop1} and Corollary \ref{cor1}.
  \end{proof}
  \subsubsection{Potentially semistable complex of a variety}
  \label{pstB}
 For a general variety $X$ over $K$, one (Zariski) covers it with (rigidified) affine varieties defined over $K$, takes the associated \v{C}ech covering,
 and applies the above construction to each level of the covering \cite[D.5-D.10]{HMS}. Then, to make everything canonical, one goes to limit over such coverings. 

  Proposition \ref{prop1} implies now the following result \cite[Prop. D.3]{HMS}.
\begin{theorem}
Let $X$ be a variety over $K$.
There is a canonical complex $\R\Gamma^B_{\pst}(X_{\ovk})\in\sd^b(\Rep_{\pst})$ which represents the \'etale cohomology $\R\Gamma_{\eet}(X_{\ovk},\Q_p)$
 of 
$X_{\ovk}$ together with the action of $G_K$, i.e., there is a natural quasi-isomorphism $$\kappa_X:\quad \R\Gamma^B_{\pst}(X_{\ovk})\simeq \R\Gamma_{\eet}X_{\ovk},\Q_p),$$
that is compatible with the action of $G_K$.
\end{theorem}
\subsection{Beilinson's $p$-adic absolute Hodge cohomology}
  Beilinson \cite{BS} uses the above construction of the potentially semistable  complexes  to define his syntomic complexes. 
\begin{definition}(\cite{BS})
Let $X$ be a variety  over $K$, $r\in\Z$. 
Set $\R\Gamma^B_{\pst}(X_{\ovk},\Qp(r)):=\R\Gamma^B_{\pst}(X_{\ovk})(r)$
and $$\R\Gamma^B_{\sh}(X,r)=\R\Gamma^B_{\synt}(X,r):=\Hom_{\sd^b(\Rep_{\pst}(G_K))}(\bq_p, \R\Gamma^B_{\pst}(X_{\ovk},\Qp(r))),\quad H^i_{\synt}(X,r):=H^i\R\Gamma^B_{\synt}(X,r). $$
\end{definition}
Immediately from this  definition we obtain that 
\begin{enumerate}
\item For $X=\Spec(K)$, we have $\R\Gamma^B_{\synt}(X,r)=\Hom_{\sd^b(\Rep_{\pst}(G_K))}(\bq_p, \bq_p(r)).$
\item There is a natural syntomic descent spectral sequence
\begin{equation}
\label{syntss}
{}^{\synt}E^{i,j}_2:=H^i_{\st}(G_K,H^j(X_{\ovk},\bq_p(r)))\Rightarrow H^{i+j}_{\synt}(X,r)
\end{equation}
\item We have a natural period map 
$$\rho_{\synt}^B:\quad \R\Gamma^B_{\synt}(X,r)\to \R\Gamma_{\eet}(X,\Q_p(r))
$$
defined as the composition
\begin{align*}
\R\Gamma^B_{\synt}(X,r) & =\Hom_{\sd^b(\Rep_{\pst}(G_K))}(\bq_p, \R\Gamma^B_{\pst}(X_{\ovk},\Qp(r)))\stackrel{}{\to}
\Hom_{\sd^b(\Spec(K)_{\proeet})}(\bq_p, \R\Gamma^B_{\pst}(X_{\ovk},\Qp(r)))\\
  & \stackrel{\kappa_X}{\to}
\Hom_{\sd^b(\Spec(K)_{\proeet})}(\bq_p, \R f^X_*\Q_p(r))=\R\Gamma_{\eet}(X,\Q_p(r))
\end{align*}
 It follows  that the syntomic descent spectral sequence 
is compatible with the  Hochschild-Serre spectral sequence via the map $\rho_{\synt}^B$.
\end{enumerate}

\subsection{Comparison of the two constructions of syntomic cohomology}
We will show now that the syntomic complexes defined in \cite{NN} and by Beilinson are naturally quasi-isomorphic. 
\begin{corollary}
\label{late1}
\begin{enumerate}
\item There is a canonical quasi-isomorphism in $\sd^b(\Rep_{\pst}(G_K))$
$$\R\Gamma_{\pst}(X_{\ovk},r)\stackrel{\sim}{\to}\R\Gamma^B_{\pst}(X_{\ovk},\Qp(r)).$$
\item There is a  canonical quasi-isomorphism 
$$
\rho_{\synt}^B: \R\Gamma_{\synt}^B(X,r)\simeq \R\Gamma_{\synt}(X,r),\quad r\in\Z.
$$
It is  compatible with period maps to \'etale cohomology and the  syntomic as well as the \'etale descent spectral sequences.
\end{enumerate}
\end{corollary}
\begin{proof}The second statement follows immediately from the first one. 
  To prove the first statement, consider the complex $\R\Gamma^B_{DF_K}(X_{\ovk},r)$ in $\sd^b(DF_K)$ defined as in Proposition \ref{prop1} but starting with  $ \R\Gamma_{pH}(X_{\ovk},r)$ instead of $\R\Gamma_{\eet}(X_{\ovk},\Qp(r))$. This is possible since,  for a good pair $(X,Y,j)$, we have $$\R\Gamma_{pH}(X_{\ovk},Y_{\ovk},r)\simeq (H^j_{\hk}(X_{\ovk},Y_{\ovk},r),(H^j_{\dr}(X,Y),F^{\jcdot+r}),
 H^j_{\hk}(X_{\ovk},Y_{\ovk})\lomapr{\iota_{\dr}} H^j_{\dr}(X_{\ovk},Y_{\ovk})),$$
 and, by $p$-adic comparison theorems, this is an element of $DF_K$. Proceeding as in the proof of Proposition \ref{prop1},
 we get  a functorial quasi-isomorphism in $\sd^b(DF_K)$:  
$$\kappa_X: \quad \R\Gamma^B_{DF_K}(X_{\ovk},r)\simeq \R\Gamma_{DF_K}(X_{\ovk},r).
$$ 
  
  For good pairs 
$(X,Y,j)$, the Beilinson period maps $\rho_{\hk},\rho_{\dr}$ \cite[3.6]{BE00}, \cite[3.2]{BE1} induce 
 the period isomorphism
$ V_{\pst}\R\Gamma_{pH}(X_{\ovk},Y_{\ovk},r)\stackrel{\sim}{\to} H^j(X_{\ovk},Y_{\ovk},\Qp(r))$. This period map
lifts to a period map $$V_{\pst}\R\Gamma^B_{DF_K}(X_{\ovk},r)\stackrel{\sim}{\to} \R\Gamma^B_{\pst}(X_{\ovk},\Qp(r)). $$
We define the map $ \R\Gamma_{\pst}(X_{\ovk},r)\stackrel{\sim}{\to}\R\Gamma^B_{\pst}(X_{\ovk},\Qp(r)) $ as the following composition
  \begin{align*}
 \R\Gamma_{\pst}(X_{\ovk},r) & \stackrel{\kappa_X^{-1}}{\to }V_{\pst}\R\Gamma^B_{DF_K}(X_{\ovk},r)\simeq \R\Gamma^B_{\pst}(X_{\ovk},\Qp(r)). 
  \end{align*}
\end{proof}
\subsection{The  Bloch-Kato exponential and the syntomic descent spectral sequence}
Let $V$ be a potentially semistable representation. Let $D=D_{\pst}(V)\in DF_K$.
The  Bloch-Kato exponential $$\exp_{\bk}: D_K/F^0\to H^1(G_K,V)$$ is defined as the composition \cite[2.4]{NN}
$$
D_K/F^0\to C(G_K,C_{\pst}(D)[1]) \to C(G_K,C(D)[1]) \stackrel{\sim}{\leftarrow} C(G_K,V[1]),
$$
where $C(G_K,\cdot)$ denotes the continuous cochains cohomology of $G_K$.
The complexes $C_{\pst}(D)$, $C(D)$ are defined as follows
\begin{align*}
C_{\pst}(D): & \quad D_{\st}\verylomapr{(N,1-\phi,\iota)}D_{\st}\oplus D_{\st}\oplus D_K/F^0\verylomapr{(1-p\phi)-N}D_{\st},\\
 C(D):\quad & D\otimes_{K_0^{\nr}}\B_{\st}\verylomapr{(N,1-\phi,\iota)}D\otimes_{K_0^{\nr}}\B_{\st}\oplus D\otimes_{K_0^{\nr}}\B_{\st}\oplus (D_{\ovk}\otimes_{\ovk}\B_{\dr})/F^0\verylomapr{(1-p\phi)-N}D\otimes_{K_0^{\nr}}\B_{\st} \end{align*}
We have $C_{\pst}(D)=C(D)^{G_K}$.

  The following compatibility result is used in the study of special values of  L-functions. Its $f$-analog was proved in \cite[Theorem 5.2]{N2}\footnote{There the exponential $\exp_{\st}$ is called $l$.}.
\begin{proposition}Let $i\geq 0$. The composition
\begin{align*}
H^{i-1}_{\dr}(X)/F^r \stackrel{\partial}{\to} H^i_{\synt}(X_h,r)\stackrel{\rho_{\synt}}{\longrightarrow} H^i_{\eet}(X,\Qp(r))\to H^i_{\eet}(X_{\ovk},\Qp(r))
\end{align*}
is the zero map. The induced (from the syntomic descent spectral sequnce) map 
$$
H^{i-1}_{\dr}(X)/F^{r}\to H^1(G_K,H^{i-1}_{\eet}(X_{\ovk},\Qp(r)))
$$
 is equal to the Bloch-Kato exponential associated with the Galois representation $H^{i-1}_{\eet}(X_{\ovk},\Qp(r))$.
\end{proposition}
\begin{proof}By the compatibility of the syntomic descent spectral sequence and the Hochschild-Serre spectral sequence \cite[Theorem 4.8]{NN}, we have the commutative diagram
$$
\xymatrix{
 H^i\R\Gamma_{\synt}(X_h,r)_0\ar[r]^{\rho_{\synt}}\ar[d]^{\delta_1} & H^i_{\eet}(X,\Qp(r))_0\ar[d]^{\delta_1}\\
H^1_{\st}(G_K,H^{i-1}_{\eet}(X_{\ovk},\Qp(r)))\ar[r]^{\can} & H^1(G_K,H^{i-1}_{\eet}(X_{\ovk},\Qp(r))),
}
$$
where 
\begin{align*}
H^i\R\Gamma_{\synt}(X_h,j)_0 &:= \ker(H^i\R\Gamma_{\synt}(X_h,r)\to H^0_{\st}(G_K,H^{i}_{\eet}(X_{\ovk},\Qp(r)))),\\
  H^i_{\eet}(X,\Qp(r))_0 & :=\ker(H^i_{\eet}(X,\Qp(r))\to H^i_{\eet}(X_{\ovk},\Qp(r))).
\end{align*}
It suffices thus to show that the dotted arrow in the following diagram
$$
\xymatrix{
 & H^i\R\Gamma_{\synt}(X,r)_0\ar[d]^{\delta_1}\\
H^{i-1}_{\dr}(X)/F^r \ar@{.>}[ru]^{\partial} \ar[r]& H^1_{\st}(G_K,H^{i-1}_{\eet}(X_{\ovk},\Qp(r))))
}
$$
exists and that this diagram commutes.

   To do that, we will use freely the notation from the proof of Corollary \ref{late1}. Set $$\widetilde{\R}\Gamma^{B}_{\synt}(X,r)=\Hom_{\sd^b(DF_K)}(K(0), \R\Gamma^B_{DF_K}(X_{\ovk},r))=H^i\holim C_{\pst}(\R\Gamma^B_{DF_K}(X_{\ovk},r)).$$
Arguing as in the proof of Proposition \ref{prop1}, we get the following commutative diagram
$$
\xymatrix{
H^i\widetilde{\R}\Gamma^B_{\synt}(X,r)_0\ar[d]^{\delta_1}\ar[r]_{\sim}^{\kappa_X} & H^i\R\Gamma_{\synt}(X_h,r)_0\ar[d]^{\delta_1}\\
H^1(C_{\pst}(H^{i-1}_{\dr}(X,r)))\ar[d]_{\wr}^{(\rho_{\hk},\rho_{\dr})}\ar@{=}[r] & H^1(C_{\pst}(H^{i-1}_{\dr}(X,r)))\ar[d]_{\wr}^{(\rho_{\hk},\rho_{\dr})}\\
H^1_{\st}(G_K,H^{i-1}_{\eet}(X_{\ovk},\Qp(r)))\ar@{=}[r] & H^1_{\st}(G_K,H^{i-1}_{\eet}(X_{\ovk},\Qp(r)))
}
$$
Moreover the comparison map $\kappa_X$ is compatible with the boundary maps $\partial$ from the de Rham cohomology complexes $\R\Gamma_{\dr}(X)$ and $\R\Gamma^B_{\dr}(X)$. It suffices thus to show that the dotted arrow in the following diagram
$$
\xymatrix{
 & H^i\widetilde{\R}\Gamma_{\synt}(X,r)_0\ar[d]^{\delta_1}\\
H^{i-1}_{\dr}(X)/F^r \ar@{.>}[ru]^{\partial} \ar[r]& H^1(C_{\pst}(H^{i-1}_{\dr}(X,r)))
}
$$
exists and that this diagram commutes.

   Let 
$$\R\Gamma^B_{DF_K}(X_{\ovk},r)=D\kr=D^0\lomapr{d^0} D^1\lomapr{d^1} D^2\lomapr{d^2} \cdots 
$$
Then $\holim C_{\pst}(\R\Gamma^B_{DF_K}(X_{\ovk},r))$ is the total complex of the double complex below.
$$
\xymatrix@C=40pt{\cdots & \cdots \ar[r]& \cdots \ar[r]&\cdots\\
C_{\pst}(D^2)\ar[u]^{d^2}: & D^2_{\st}\ar[r]^-{(N,1-\phi,\iota)}\ar[u]^{d^2} & D^2_{\st}\oplus D^2_{\st}\oplus D^2_K/F^0\ar[r]^-{(1-p\phi)-N}\ar[u]^{d^2} & D^2_{\st} \ar[u]^{d^2}\\
C_{\pst}(D^1)\ar[u]^{d^1}: & D^1_{\st}\ar[r]^-{(N,1-\phi,\iota)}\ar[u]^{d^1} & D^1_{\st}\oplus D^1_{\st}\oplus D^1_K/F^0\ar[r]^-{(1-p\phi)-N}\ar[u]^{d^1} & D^1_{\st} \ar[u]^{d^1}\\
C_{\pst}(D^0)\ar[u]^{d^0}: & D^0_{\st}\ar[r]^-{(N,1-\phi,\iota)}\ar[u]^{d^0} & D^0_{\st}\oplus D^0_{\st}\oplus D^0_K/F^0\ar[r]^-{(1-p\phi)-N}\ar[u]^{d^0} & D^0_{\st} \ar[u]^{d^0}
}
$$
We note that $D\kr_{\st}=\R\Gamma_{\hk}^B(X,r)$, $D_K\kr=\R\Gamma_{\dr}^B(X)$. 
   The following facts are easy to check.
\begin{enumerate}
\item The map $\partial: \R\Gamma^B_{\dr}(X)/F^r\to \widetilde{\R}\Gamma^B_{\synt}(X,r)[1]$ is given by the canonical morphism
$$
D\kr_K/F^0\to [D\kr_{\st}\lomapr{}D\kr_{\st}\oplus D_{\st}\kr\oplus D_{K}\kr/F^0\lomapr{}D\kr_{\st}][1]
$$Similarly, the map $H^{i-1}_{\dr}(X)/F^r\to H^1(C_{\pst}(H^{i-1}_{\dr}(X,r)))$ is given by the canonical morphism
$$
H^{i-1}_{\dr}(X)/F^r\to[H^{i-1}_{\hk}(X,r)\to H^{i-1}_{\hk}(X,r)\oplus H^{i-1}_{\hk}(X,r)\oplus H^{i-1}_{\dr}(X,r)/F^0\to H^{i-1}_{\hk}(X,r)][1].
$$
\item  The map $H^i\widetilde{\R}\Gamma^B_{\synt}(X,r)\to H^0(C_{\pst}(H^i_{\dr}(X,r)))$ is induced by $(a,b,c)\mapsto a$.
\item The map $\delta_1: H^i\widetilde{\R}\Gamma^B_{\synt}(X,r)_0\to H^1(C_{\pst}(H^{i-1}_{\dr}(X,r)))$ is induced by $(a,b,c)\mapsto b-d_0a^{\prime}$, where $a^{\prime}$ is such that $d^{i}a^{\prime}=a$.
\item As a corollary of the above, we get that the composition
$$
H^{i-1}_{\dr}(X)/F^r\to H^i\widetilde{\R}\Gamma^B_{\synt}(X,r)_0\stackrel{\delta_1}{\to} H^1(C_{\pst}(H^{i-1}_{\dr}(X,r)))
$$
is induced by the map $b\mapsto (0,0,b)\mapsto b$.
\end{enumerate}
This proves our proposition.
\end{proof}

\section{$p$-adic realizations of motives}
\subsection{$p$-adic realizatons of Nori's motives} We start with  a quick review of Nori's motives. We follow \cite{HMS},  \cite{ML}, \cite{AB}, and \cite[2]{DA}.

  Take an embedding $K\hookrightarrow {\mathbf C}$ and a field $F\supset \Q$. A diagram $\Delta$ is a directed graph. A representation
$T: \Delta\to V_{F}$  assigns to every vertex in
$\Delta$ an object in
$V_{F}$ and to every edge
$e$ from
$v$ to
$v^{\prime}
$ a homomorphism
$T(e) :T(v)\to T(v^{\prime}).$ Let $\scc(\Delta, T) $ be its associated diagram category \cite[Thm 41]{ML}, \cite[2.1]{DA}:  the category of finite dimentional right $\End^{\vee}(T)$-comodules. It is the universal abelian
category together with  a unique representation $\wt{T}: \Delta\to \scc(\Delta,T)$ and a  faithful, exact, $F$-linear functor $ T: \scc(\Delta, T)\to V_{F}$  
extending  the original representation $T$. If $\Delta$ is an abelian category then we have an equivalence $\Delta\simeq \scc(\Delta,T)$.

   More specifically we have the following result of Nori.
\begin{proposition}
\label{Nori1}(Nori,  \cite[Cor. 2.2.10]{DA}, \cite[Cor. 2.2.11]{DA})
\begin{enumerate}
\item Let $\srr$ be an $F$-linear abelian category with a faithful exact functor $\rho:\srr\to V_F$. Assume  that the representation $T:\Delta\to V_F$ factors, up to natural equivalence,  as $T_1\rho$. Let  $\sa$ be  an $F$-linear abelian category equipped with a faithful exact functor $U:\sa\to\srr$. If $G:\Delta\to\sa$ is a morphism of directed graphs such that $T_1$ is equivalent to $UG$, then there exist functors $\scc(\Delta,T)\to \srr,$ $\wt{G}:\scc(\Delta,T)\to\sa$ such that the following diagram
$$
\xymatrix@C=40pt{
\Delta\ar[r]^G \ar[d]^{\wt{T}}\ar[dr]_{ T_1}&\sa\ar[d]^U\\
\scc(\Delta,T)\ar@{-->}[r] \ar[rd]^T\ar@{-->}[ur]^-{\wt{G}}&\srr\ar[d]^{\rho}\\
& V_F
}
$$
commutes up to natural equivalence. 
 \item 
For a commutative (up to natural equivalence)  diagram
$$
\xymatrix{
\Delta\ar[dd]^{\pi}\ar[r] ^{G}\ar[rrd]_-T & \sa\ar[dd]\ar[rd]^U\\
& & V_F\\
\Delta^{\prime}\ar[rru]_{T^{\prime}}\ar[r]_{G^{\prime}} & \sa^{\prime}\ar[ur]_{U^{\prime}}
}
$$
we have a commutative (up to natural equivalence) diagram 
$$
\xymatrix@C=50pt{
\scc(\Delta,T)\ar[d]^{\pi} \ar[r]^-{\wt{G}} & \sa\ar[d]\\
\scc(\Delta^{\prime},T^{\prime})\ar[r]^-{\wt{G}^{\prime}} & \sa^{\prime}}
$$
\end{enumerate}
\end{proposition}
\begin{example}
The following diagrams appear in the construction of Nori's motives.
\begin{enumerate}
 \item The diagram
$\Delta^{\eff}$ of effective pairs consists of pairs $(X, Y, i)$ and 
two types of edges:
\begin{enumerate}
\item  (functoriality) for every morphism $f: X\to X^{\prime}$, with $f(Y)\subset Y^{\prime}$, 
an
edge $f^*:(X^{\prime},Y^{\prime},i)\to (X,Y,i)$.
\item
 (coboundary) for every chain $X\supset Y\supset Z$ of closed $K$-subvarieties of $X$, an edge
 $\partial: (Y,Z,i)\to (X,Y,i+1)$. 
\end{enumerate}
\item  The diagram $\Delta^{\eff}_ {g}$ (resp. ${\Delta}^{\eff}_{vg}$) 
of effective good (resp. of
effective very good)  pairs is  the full subdiagram of $\Delta^{\eff}$
with vertices good   (resp. very good) pairs $(X, Y, i) $. 
\item 
The diagrams $\Delta$ of pairs, $\Delta_{g}$ of good pairs, and ${\Delta}_{vg}$
of very good pairs are obtained by localization with respect to the pair $({\mathbb G}_m,\{1\},1)$ \cite[B.18]{HMS}.
\end{enumerate}
\end{example} 

Let  $H^*:\Delta_{g}\to V_{F}$  be the representation which assigns to $(X, Y, i) $ the relative
singular cohomology
$H^i(X({\mathbf C}), Y({\mathbf C}),F).$
\begin{definition}\label{df:Nori_motives}
The category of (reps. effective) Nori motives
$\MM(K)_F$ (resp.
$\MM(K)_F$) is defined as the diagram category
$\scc(\Delta_{g}, H^*) $ (resp. $\scc(\Delta^{\eff}_{g}, H^*))$.
For a good pair $(X,Y,i)$,  we denote by $\Hm^i(X,Y)$ the object of $\MMe(K)_F$ (resp. $\MM(K)_F$)
 corresponding to it and we define
  the Tate object as
$${\un}(-1): =\Hm^1({\mathbb G}_{m,K},\{1\})\in \MMe(K)_F,
 \quad {\un}(-n):={\un}(-1)^{\otimes n}.
$$
\end{definition}
We have \cite[Thm 1.6, Cor. 1.7]{HMS}
\begin{itemize}
\item $\MMe(K)_F\simeq \MMe(K)_\Q\otimes_{\Q}F$ and 
$ \MM(K)_F\simeq \MM(K)_\Q\otimes_{\Q}F$.
\item As an abelian category $\MMe(K)_F$ is generated by Nori motives
 of the form $\Hm^i(X,Y)$ for  good pairs $(X,Y,i)$; every  object of $\MMe(K)_F$ is a subquotient of a finite direct sum of objects of the form $\Hm^i(X,Y)$.
\item $\MMe(K)_F\subset \MM(K)_F$ are commutative tensor categories.
\item $\MM(K)_F$ is obtained from $\MMe(K)_F$ by $\otimes$-inverting 
 ${\un}(-1)$.
\item The diagram categories of $\Delta^{\eff}$
and of ${\Delta}^{\eff}_{vg}$ with respect to singular
cohomology with coefficients in $F$ are equivalent to $\MMe(K)_F$
as abelian categories. The diagram categories
of $\Delta$ and of ${\Delta}_{vg}$ 
are equivalent to $\MM(K)_F$.\footnote{This is shown by an argument analogous to the one we have used in the construction of Beilinson's potentially semistable complex of a variety in Section \ref{pstB} : via cellular complexes and \v{C}ech coverings one lifts the representation $H^*$ from very good pairs to all pairs to a representation that canonically computes relative singular cohomology.} In particular, any pair $(X,Y,i)$
 defines a Nori motive $\Hm^i(X,Y)$.
\item Nori shows that these categories are independent of the embedding $ K\hookrightarrow {\mathbf C}$. 
\end{itemize}

  From the universal property of the  category $\MMe(K)_F$  it is easy to construct realizations. We will describe the ones coming from $p$-adic Hodge Theory.
  \begin{construction}\label{constr1}(Galois realization)
  Consider the map $\Delta^{\eff}\to \Rep(G_K)$: 
  $$(X,Y,i)\mapsto H^i(X_{\ovk},Y_{\ovk},\Qp).
 $$
 We have $H^i(X_{\ovk},Y_{\ovk},\Qp)\simeq H^i(X({\mathbf C}),Y({\mathbf C}),\Qp)$. Thus, by Theorem \ref{Nori1}, we obtain an extension which is the exact \'etale realization functor
 $$
 \R_{\eet}: \quad \MMe(K)_{\Qp}\to \Rep(G_K). 
 $$
 Note that $\R_{\eet}({\un}(-1))=H^1({\mathbb G}_{m,\ovk},\{1\},\Qp)=\Qp(-1)$. Hence the functor $\R_{\eet}$ lifts to $\MM(K)_{\Qp}$. 
 
  In analogous way we obtain the exact potentially semistable realization 
  $$
   \R_{\pst}: \quad \MM(K)_{\Qp}\to \Rep_{\pst}(G_K). 
  $$
  It factors $\R_{\eet}$ via the natural functor $\Rep_{\pst}(G_K)\to \Rep(G_K)$. 
  \end{construction}
  \begin{construction}\label{constr11}(Filtered $(\phi,N,G_K)$  realization)
  Consider the map $\Delta^{\eff}\to DF_K$: 
  $$(X,Y,i)\mapsto H^i_{DF}(X,Y):=(H^i_{\hk}(X_{\ovk},Y_{\ovk}),(H^i_{\dr}(X,Y),F^{\jcdot}),
  \iota_{\dr}:H^i_{\hk}(X_{\ovk},Y_{\ovk})\otimes_{K_0^{\nr}}\ovk\stackrel{\sim}{\to}H^i_{\dr}(X_{\ovk},Y_{\ovk})).
 $$
 By $p$-adic comparison theorems, we have $$D_{\pst}(H^i_{DF}(X,Y))\simeq H^i(X_{\ovk},Y_{\ovk},\Qp)\simeq H^i(X({\mathbf C}),Y({\mathbf C}),\Qp).
 $$Thus, by Theorem \ref{Nori1}, we obtain an extension which is the exact  filtered $(\phi,N,G_K)$ realization functor
 $$
 \R_{DF_K}: \quad \MMe(K)_{\Qp}\to DF_K. 
 $$
 Since  $\R_{DF}({\un}(-1))=K(-1)$, the functor $\R_{DF}$ lifts to $\MM(K)_{\Qp}$. 
 
  Projections yield faithful exact functors from $DF_K $ to the categories $M_K({\phi},N,G_K)$ and $V_{\dr}^K$.
  Composing them with the realization $\R_{DF}$ we get
  \begin{itemize}
  \item the exact 
   Hyodo-Kato realization
  $$\R_{\hk}: \quad \MM(K)_{\Qp} \to  M_K(\phi,N,G_K),$$
  \item the exact 
  de Rham realization
  $$\R_{\dr}:\quad  \MM(K)_{\Qp}\to V_{\dr}^K.$$
  \end{itemize}
  Composing $\R_{DF_K}$ with the projection on the third factor  of the filtered $(\phi,N,G_K)$-module, we obtain the Hyodo-Kato natural equivalence
  \begin{equation}
  \label{hk2}
  \iota_{\dr}: \quad \R_{\hk}\otimes_{K^{\nr}_0}\ovk\simeq \R_{\dr}\otimes_{K}\ovk: \quad \MM(K)_{\Qp}\to V_{\ovk} ,
  \end{equation}
  where the tensor product is taken pointwise.
  \end{construction}
  \begin{construction}\label{constr2}(Realization of period isomorphism)
 To realize period isomorphisms, we define the category of realizations $\srr (K)$. An object of $\srr (K)$ is a tuple $M:=(M_{DF},M_{\pst}, \rho_{\pst})$ consisting of 
  $M_{DF}\in DF_K$,
  $M_{\pst}\in \Rep_{\pst}(G_K)$, and  a comparison isomorphism $\rho_{\pst}:V_{\pst}M\simeq M_{\pst}$ of Galois modules.
  It is a abelian category (it is naturally equivalent to the category $ \Rep_{\pst}(G_K)$).   Projections yield faithful exact functors from $\srr (K)$ to the categories $DF_{K}$ and $\Rep_{\pst}(G_K)$.
  
  Consider the following map $\Delta^{\eff} \to \srr (K)$:
  \begin{align*}
  (X,Y,i)\mapsto  (H^i_{DF}(X,Y), H^i(X_{\ovk},Y_{\ovk},\Qp),  \rho_{\pst}:  V_{\pst}H^i_{DF}(X,Y)\simeq H^i(X_{\ovk},Y_{\ovk},\Qp)).
  \end{align*}
  Since the functor $\srr(K)\to \Rep_{\pst}(G_K)\to V_{\Q_p}$ is faithful and exact, Theorem \ref{Nori1} gives us an extension $\MMe(K)_{\Qp} \to \srr (K)$ that is compatible with the \'etale realization.
  Since
  $$({\mathbb G}_m,\{1\},1)\mapsto (K(-1),\Qp(-1), V_{\pst}K(-1)\simeq \Qp(-1)),
  $$ again  by Theorem \ref{Nori1},
  we obtain the exact realization
  $$\R_{\srr}: \quad  \MM(K)_{\Qp} \to  \srr (K) . $$
Projecting on the first two factors we get back the realizations 
  $\R_{DF_K}$ and $\R_{\pst}$ and 
   projecting on the third factor  we get that the above two  realizations are related via a period morphism, i.e., we have 
   a natural equivalence
   $$\rho_{\pst}:\quad V_{\pst}\R_{DF_K}\simeq \R_{\pst}: \quad \MM(K)_{\Qp}\to \Rep_{\pst}(G_K).
   $$

   To sum up, we have a  potentially semistable comparison theorem for Nori's motives. 
  \begin{corollary}For $M\in \MM(K)_{\Qp}$, there is a functorial isomorphism
  $$\rho_{\pst}: \quad \R_{\hk}(M)\otimes_{K^{\nr}_0}\B_{\st}\simeq \R_{\eet}(M)\otimes_{\Qp}\B_{\st}
  $$
  that is compatible with Galois action, Frobenius, and the monodromy operator. Moreover, after passing to  $\B_{\dr}$ via the Hyodo-Kato map (\ref{hk2}), 
  it yields a functorial isomorphism
  $$\rho_{\dr}: \quad \R_{\dr}(M)\otimes_{K}\B_{\dr}\simeq \R_{\eet}(M)\otimes_{\Qp}\B_{\dr}  $$
  that is compatible with filtration.
  \end{corollary}
  \end{construction}
We can illustrate   the above  constructions  by the following, essentially commutative, diagram of exact functors
    $$
\xymatrix@C=22pt@R=14pt{
 && & \Rep(G_K)\ar_{\otimes_{\Qp} \B_{\st}}[rd]\ar@/^17pt/^{\otimes_{\Qp} \B_{\dr}}[rrd] & \\
 && \Rep_{\pst}(G_K)\ar_\iota[ru]
 && M_{\B_{\st}}(\phi,N,G_K) \ar[r] &  MF_{\B_{\dr}}\\
\MM(K)_{\Qp}\ar|{\R_{\hk}}[rrr]
   \ar^{\R_{\pst}}[rru]\ar_{R_{DF_K}}[rrd]
   \ar@/^26pt/^{\R_{\eet}}[rrruu]
   \ar@/_26pt/_{\R_{\dr}}[rrrdd]
 && & M_K(\phi,N,G_K)\ar_-{\otimes_{K_0^{\nr}} \B_{\st}}[ru]\ar^-{F_0}[rd] & \\
 && {DF_K}\ar^-{}[ru]\ar_-{}[rd]\ar[uu]_{V_{\pst}}
  && V_{\ovk}^G \\
 && & V_{\dr}^K \ar@/_36pt/_{\otimes_K\B_{\dr}}[rruuu]\ar_-{F_{\dr}}[ru] &
}
$$
Here $M_{\B_{\st}}(\phi,N,G_K)$ is the exact category of free finite rank $\B_{\st}$-modules equipped with an action of $\phi,N,G_K$ ($\phi$ is an isomorphism, $N$ is nilpotent, and $G_K$-action is continuous - everything being compatible in the usual way and compatible with the same structures on $\B_{\st}$). $MF_{\B_{\dr}}$ is the exact category of  filtered $\B_{\dr}$-modules, and $M_{\B_{\st}}\to M_{\B_{\dr}}$ is the natural functor.

\subsection{$p$-adic realizations of Voevodsky's motives} 
 
\begin{recall}
The category of Voevodsky's motives $DM(K,\Qp)$ with rational coefficients 
 admits several equivalent constructions,
 each interesting in its own. In this section, we will be using the one
 of Morel (see \cite{Mor1}) for a review of which we refer the reader to
 \cite[\textsection 1]{DegMaz}.

By construction, the triangulated category $DM(K,\Qp)$ is stable
 under taking arbitrary coproducts. In this category, each smooth $K$-scheme $X$
 admits a homological motive $M(X)$, covariant with respect to morphism
 of $K$-schemes (and even finite correspondences). Each motive
 can be twisted by an arbitrary integer power of the Tate object
 $\Qp(1)$, and as a triangulated category stable under taking  coproducts,
 $DM(K,\Qp)$ is generated by motives of the form $M(X)(n)$, $X/K$ smooth,
 and $n \in \Z$.
 
The category of \emph{constructible} motives (see also \ref{df:DMh})
 is the thick\footnote{\emph{i.e.} stable by direct factors} triangulated subcategory
 of $DM(K,\Qp)$ generated by the motives $M(X)(n)$, $X/K$ smooth, and $n \in \Z$, without
 requiring stability by infinite coproducts. It is equivalent to
 Voevodsky's category of geometric motives $DM_{gm}(K,\Qp)$
 (\cite[chap. 5]{FSV})
 and can also be described in an elementary way  as follows.
Let $\Qp[\Sma_K]$ be the $\Qp$-linearization of the category of smooth 
 affine $K$-varieties, $K^b(\Qp[\Sma_K])$ its bounded homotopy category. This is
 a triangulated monoidal category, the tensor structure being induced by
 cartesian products of $K$-schemes.
 First we get the geometric $\AA^1$-derived category $D_{\AA^1,gm}(K,\Qp)$
 out of $K^b(\Qp[\Sma_K])$ by the following operations:
\begin{enumerate}
\item Take the Verdier quotient with respect to the triangulated
 subcategory generated by complexes of the form:
\begin{itemize}
\item (\textit{homotopy}) $\hdots \rightarrow 0 \rightarrow \AA^1_X \xrightarrow p X \rightarrow 0 \hdots$,
 for $X \in \Sma_K$, $p$ canonical projection;
\item (\textit{excision}) $\hdots \rightarrow 0 \rightarrow W \xrightarrow{q-k} U \oplus V
\xrightarrow{j+p} X \hdots$,
 for any cartesian square
 $\xymatrix@=10pt{W\ar^k[r]\ar_q[d] & V\ar^p[d] \\ U \ar^j[r] & X}$
 in $\Sma_K$ such that $j$ is an open immersion, $p$ is \'etale and an isomorphism
 above the complement of $j$.
\end{itemize}
\item Formally invert the Tate object $\Qp(1)$, which is the cokernel of
 $\{1\} \rightarrow \mathbb G_m$ placed in cohomological degree $+1$.
\item Take the pseudo-abelian envelope.
\end{enumerate}
Let $\tau$ be the automorphism of $\Qp(1)[1] \otimes \Qp(1)[1]$ in $D_{\AA^1,gm}(K,\Qp)$
 which permutes the factors. Because $\Qp(1)$ is invertible, it induces an
 automorphism $\epsilon$ of $\Qp$ in $D_{\AA^1,gm}(K,\Qp)$ such that
 $\epsilon^2=1$. Then we can define complementary projectors:
$p_+=(1-\epsilon)/2, p_-=(\epsilon-1)/2$,
which cut the objects, and therefore the category, into two pieces:
$$D_{\AA^1,gm}(K,\Qp)_+=\mathrm{Im}(p_+), \quad
 D_{\AA^1,gm}(K,\Qp)_-=\mathrm{Im}(p_-).$$
 Then, according to a theorem of Morel (cf. \cite[16.2.13]{CD3}),
 $\DMgm(K,\Qp) \simeq D_{\AA^1,gm}(K,\Qp)_+$.
\end{recall}

\begin{example}\label{ex:basic_realization}
Let $F$ be an extension field of $\Qp$
 and $\mathscr A$ be a Tannakian $F$-linear category
 with a fiber functor $\omega:\mathscr A \rightarrow V_F$.
Consider a contravariant functor:
$$
R:\quad (\Sma_K)^{op} \rightarrow C^b(\mathscr A).
$$
It automatically extends to a contravariant functor
 $R':K^b(\Qp[\Sma_K])^{op} \rightarrow D^b(\mathscr A)$. 
 The conditions for $R^{\prime}$ to induce a contravariant functor 
 defined on $\DMgm(K,\Qp)$ are easy to state given the
 description of $\DMgm$ given above.
 We will use the following simpler criterion:

We now suppose that the functor $R$ takes its
 values in the bigger category $C^b(\Ind-\mathscr A)$
 but we assume that there exists a functorial isomorphism
$$
H^i\omega R(X) \simeq H^i(X(\mathbf C),F)
$$
and that the product map
$H^i(X(\mathbf C),F) \otimes  H^i(Y(\mathbf C),F) \rightarrow
 H^i(X(\mathbf C) \times Y(\mathbf C),F)$ can be lifted
 to a map $R(X) \otimes R(Y) \rightarrow R(X \times_k Y)$ in
 $C^b(\mathscr A)$.

Then the functor $R'$ uniquely extends to  a realization functor
$$
\tilde R^\vee:\quad \DMgm(K,\Qp)^{op} \rightarrow D^b(\mathscr A)
$$
which is monoidal and such that $H^i(\tilde R^\vee(M(X)))=H^i(R(X))$.\footnote{Note
 in particular that the permutation $\epsilon$ acts by $-1$ on
 singular cohomology.} 
After composing this functor with the canonical duality endofunctor
 of the (rigid) triangulated monoidal category $D^b(\mathscr A)$,
 we get a covariant realization:
$$
\tilde R:\quad \DMgm(K,\Qp) \rightarrow D^b(\mathscr A)
$$
such that $H^i\tilde R(M(X))=H^i(R(X))^\vee$.
 Note also that, by construction, the preceding identification
 can be extended to closed pairs. Also, because $\DMgm(K,\Qp)$ satisfies
 h-descent (see section \ref{ex:DMh_comparison}), it
 can be extended to singular $K$-varieties and pairs of such.
\end{example}

Using this example we can easily build realizations:
\begin{proposition}\label{prop:basic_realization}
Let $F$ be an extension field of $\Qp$
 and $\mathscr A$ be a Tannakian $F$-linear category
 with a fiber functor $\omega:\mathscr A \rightarrow V_F$.
Consider a representation $A^*:\Delta_g \rightarrow \mathscr A$
 such that $\omega A^*$ is isomorphic to the singular representation
 (see Definition \ref{df:Nori_motives}).

Then there exists a canonical covariant monoidal realization:
$$
R_A:\quad \DMgm(K,\Qp) \rightarrow D^b(\mathscr A)
$$
such that for any good pair $(X,Y,i)$,
 $H^iR_A(M(X,Y))=A^i(X,Y)^\vee$ and this identification is
 functorial in $(X,Y,i)$ -- including with respect to boundaries.

Moreover, this construction is funtorial with respect
 to exact morphisms of representations.
%
%if we have an essantially commutative diagram
 %of functors:
%$$
%\xymatrix@C=30pt@R=2pt{
%& \mathscr A\ar^\varphi[dd] \\
%\Delta_g\ar^-{A^*}[ru]\ar_-{B^*}[rd] \\
 %& \mathscr B
%}
%$$
%such that $\varphi$ is exact and monoidal
 %we get an essentially commutative diagram of functors:
%$$
%\xymatrix@C=40pt@R=2pt{
%& D^b(\mathscr A)\ar^\varphi[dd] \\
%\DMgm(K,\Q)\ar^-{R_A}[ru]\ar_-{R_B}[rd] \\
 %& D^b(\mathscr B)
%}
%$$
\end{proposition}
\begin{proof}
Let $X$ be a smooth affine $K$-scheme.
 To any cellular stratification of $X$
  (cf. Corollary \ref{cor1}) $F_{\jcdot}X$, we can associate the complex
$$
R'_A(F_{\jcdot}X):=0 \rightarrow A^0(F_0X) \rightarrow A^1(F_1X,F_0X)
   \rightarrow \hdots \rightarrow A^d(X,F_{d-1}X) \rightarrow 0.
$$
We put: $R'_A(X):=\mathrm{colim}_{F_{\jcdot}X} R'_A(F_{\jcdot}X)$. This defines
 a contravariant functor:
$$
R'_A:\quad  (\Sma_K)^{op} \rightarrow C^b(\Ind-\mathscr A)
$$
which satisfies the assumptions of the previous example.
 Hence we get the proposition by applying the construction
 of this example.
\end{proof}

\begin{remark}\label{rem:realizations_agree}
Consider again a fiber functor $\omega:\mathscr A \rightarrow V_F$
 and a contravariant functor
 $$R:\quad (\Sch_K)^{op} \rightarrow C^b(\Ind-\mathscr A)$$
 such that for any $K$-variety $X$, one has a
 functorial isomorphism $H^i\omega R(X) \simeq H^i(X(\mathbf C),F)$.
 Then we can apply the preceding example to $R|_{\Sma_K}$
 and also the preceding proposition to the unique representation
 $A^*$ induced by $R$ such that
 $A^i(X,Y)=H^i(\mathrm{Cone}(R(X)\rightarrow R(Y)[-1])$.
By applying the construction of the preceding proof,
 we get for any smooth affine $K$-scheme a canonical map
 of complexes
$$
R(X) \rightarrow R_A(X)
$$
which is a quasi-isomorphism. For the functoriality of the construction
 of the previous example, we thus get a canonical isomorphism
 between the two realizations of any Voevodsky's  motive $M$:
$$
\tilde R(M) \xrightarrow{\sim} \tilde R_A(M)
$$
\end{remark}

\begin{remark}\label{rem:motives_singular}
Voevodsky's motives $M(X)$ are homological: they are covariant in $X$.
 In fact, the monoidal category $\DMgm(K,\Qp)$ is rigid: any object
 has a strong dual -- cf. \cite{Riou}. Then for any smooth $K$-variety $X$,
 $M(X)^\vee$ is the cohomological motive of $X/K$.
 Using the notations of the previous proposition,
 because $R_A$ is monoidal and therefore commutes with  strong duals,
 we get: $H^iR_A(M(X)^\vee)=A^i(X)$.

Recall that the category $\DMgm(K,\Qp)$ can be extended to any base and satisfies
 the 6 functors formalism (cf. \cite{CD3}, in particular 16.1.6).
 According to \emph{loc. cit.}, 15.2.4, $M(X)^\vee=f_*(\un_X)$ where
 $f:X \rightarrow \Spec(K)$ is the structural morphism.
 The preceding relation can be rewritten:
$$
H^iR_A(f_*(\un_X))=A^i(X).
$$
Note finally that $f_*$ exists for any $K$-variety $X$.
 One can extend the above identification to this general case using 
 De Jong resolution of singularities and $h$-descent,
 which is true for Voevodsky'srational motives (\cite[14.3.4]{CD3})
 and for Betti cohomology.

There is fully faithful monoidal functor
$$
\mathrm{CHM}(K)_{\Qp}^{op} \rightarrow \DMgm(K,\Qp), \quad h(X) \mapsto M(X)
$$
from the category of Chow motives ($X$ is smooth projective over $K$).
Applying duality on the right hand side, we get a covariant
 fully faithful monoidal functor:
$$
\mathrm{CHM}(K)_{\Qp}\rightarrow \DMgm(K,\Qp), \quad h(X) \mapsto M(X)^\vee=f_*(\un_X).
$$
In view of this embedding, it is convenient to identify the Chow motive $h(X)$
 with the Voevodsky's(cohomological) motive $M(X)^\vee$.
\end{remark}

Let us also state the following corollary 
 which follows from the preceding proposition
 and \cite{De3}:
\begin{corollary}
In the assumptions of the previous proposition,
 for any smooth projective $K$-scheme $X$ of dimension $d$,
 the complex $R_A(h(X))=R_A(M(X)^\vee)$ is split:
 there exists a canonical isomorphism:
$$
R_A(h(X))=\bigoplus_{i=0}^{2d} H^i\big(R_A(h(X))\big)[-i]
 =\bigoplus_{i=0}^{2d} A^i(X)[-i].
$$
\end{corollary}
This decomposition statement follows simply from \emph{loc. cit.} as the derived category
 $D^b(\mathscr A)$ satisfies the assumptions of \emph{loc. cit.}
 and the object $R_A(h(X))$ 
 satisfies the assumption (L.V.) for the 
 map $h(X) \rightarrow h(X)(1)[2]$ given by multiplication
 by the (motivic) first Chern class of an ample invertible bundle
 on $X$.

\begin{example}
In particular, applying the preceding proposition
 to the singular representation,
 we get the classical realization,\footnote{Conjecturally,
 this is more than a realization: it is thought to be an equivalence
 of categories!} due to Nori,
 of (cohomological) Nori's  motives:
$$
\Gamma:\quad \DMgm(K,\Qp) \rightarrow D^b(\MM(K)_{\Qp}).
$$
By definition, and applying the preceding remark, we get for
 any smooth projective (resp. smooth, any) $K$-variety
 $f:X \rightarrow \Spec(K)$:
$$
H^i\Gamma(h(X))=H^i_{mot}(X), \
 \text{resp. } H^i\Gamma(M(X))=H^i_{mot}(X)^\vee, \
 H^i\Gamma(f_*(\un_X))=H^i_{mot}(X).
$$
When $X$ is smooth projective of dimension $d$,
 we also get by the above corollary the decomposition:
$$
\Gamma(h(X))=\bigoplus_{i=0}^{2d} \Hm^i(X)[-i]
$$
Moreover, because of the functorialility statement of
 the proposition, this realization of Voevodsky's  motives
 is the universal (initial) one. 
\end{example}

\begin{num}\label{num:real_Voevodsky}
More interestingly, using either Example \ref{ex:basic_realization}
 or Proposition \ref{prop:basic_realization},
 we can get various p-adic realizations of Voevodsky's  motives,
 and extend the de Rham $p$-adic comparison theorem to the derived situation
 as summarized in the following essentially commutative diagram
 of triangulated monoidal functors:
$$
\xymatrix@C=14pt@R=14pt{
 & && D^b\big(\Rep(G_K)\big)\ar^{\otimes_{\Qp} \B_{\dr}}[rd] & \\
\DMgm(K,\Qp)\ar_{\R\Gamma_{DF_K}}[rrd]\ar^{\R\Gamma_{\pst}}[rr]\ar@/^15pt/^{\R\Gamma_{\eet}}[rrru]
   \ar@/_30pt/_{\R\Gamma_{\dr}}[rrrd]
 && D^b\big(\Rep_{\pst}(G_K)\big)\ar_\iota[ru] & 
 & D^b(MF_{\B_{\dr}}) \\
 & & D^b(DF_K)\ar[u]^{V_{\pst}}\ar[r] & D^b(V_{\dr}^K)\ar_{\otimes_{K} \B_{\dr}}[ru] &
}
$$
where $\iota$ is the canonical functor.\footnote{One should be
 careful that though $\iota$ is induced by a fully faithful functor
 on the corresponding abelian categories, it is a {\it non full}
 faithful functor.}
 The functors $\R\Gamma_{\eet}$, $\R\Gamma_{\pst}$ and $\R\Gamma_{DF_K}$ are
 obtained either from \ref{ex:basic_realization} or equivalently
 from \ref{prop:basic_realization}
 (according to Remark \ref{rem:realizations_agree}) by considering
 respectively the following functors:
\begin{itemize}
\item $X \in \Sma_K, f:X \rightarrow \Spec(K) \mapsto \derR f_*(\Qp)$
 and $(X,Y,i) \mapsto H^i_{\eet}(X_{\ovk}, Y_{\ovk},\Qp)$;
\item $X \in \Sma_K \mapsto \R\Gamma_{\pst}(X_{\ovk},r)
 \simeq\R\Gamma^B_{\pst}(X_{\ovk},\Qp(r))$ \\ 
 and $(X,Y,i) \mapsto H^i_{\eet}( X_{\ovk}, Y_{\ovk},\Qp) \in \Rep_{\pst}(G_K)$;
\item  $(X,Y,i) \mapsto H^i_{DF}(X, Y)$ (see Construction \ref{constr11}).
\end{itemize}
The functor $\R\Gamma_{\dr}$ is obtained by composing $\R\Gamma_{DF_K}$ with the canonical functor $DF_K\to V^K_{\dr}$. 

  For $\epsilon=\eet,\pst,DF_K$, one has defined in the preceding section
 an analoguous exact monoidal realization functor $R_\epsilon$
 from the category of Nori's  motives $\MM(K)_{\Qp}$. This functor
 being exact induces a functor on the (bounded) derived categories
 and according to the functoriality in Proposition
 \ref{prop:basic_realization}, one gets for any Voevodsky motive
 $M \in \DMgm(K,\Qp)$:
\begin{equation} \label{eq:real_Voevodsy&Nori}
\R\Gamma_\epsilon(M)=\R_\epsilon\big(\Gamma(M)\big).
\end{equation}
Same for the de Rham realizations: we have $\R\Gamma_{\dr}(M)=\R_{\dr}\big(\Gamma(M)\big)$.
Therefore, the essential commutativity of the previous diagram
 simply follows from the de Rham  comparison theorem for
 Nori's  motives. More precisely, it yields, for any Voevodsky's  motive
 $M$,  the de Rham comparison isomorphism:
$$
\rho_{\dr}: \quad \R\Gamma_{\dr}(M)\otimes_{K}\B_{\dr}
 \simeq \R\Gamma_{\eet}(M)\otimes_{\Qp}\B_{\dr}
$$
which is a quasi-isomorphism of complexes
 of filtered finite rank  $\B_{\dr}$-modules.

This comparison can be made more precise through the Hyodo-Kato realization,
 as illustrated in the essentially commutative diagram:
$$
\xymatrix@C=22pt@R=14pt{
 && & D^b\big(\Rep(G_K)\big)\ar^{\otimes_{\Qp} \B_{\st}}[rd]\ar@{}|{(1)}[dd] & \\
 && D^b\big(\Rep_{\pst}(G_K)\big)\ar_\iota[ru]
 && D^b\big(M_{\B_{\st}}(\phi,N,G_K)\big) \\
\DMgm(K,\Qp)\ar|{\R\Gamma_{\hk}}[rrr]
   \ar^{\R\Gamma_{\pst}}[rru]\ar_{\R\Gamma_{DF_K}}[rrd]
   \ar@/^26pt/^{\R\Gamma_{\eet}}[rrruu]
   \ar@/_26pt/_{\R\Gamma_{\dr}}[rrrdd]
 && & D^b\big(M_K(\phi,N,G_K)\big)\ar_-{\otimes_{K_0^{\nr}} \B_{\st}}[ru]\ar^-{F_0}[rd]\ar@{}|{(2)}[dd] & \\
 && D^b(DF_K)\ar[ru]\ar[rd]\ar[uu]_{V_{\pst}}
 && D^b\big(V_{\ovk}^G\big). \\
 && & D^b(V_{\dr}^K)\ar_-{F_{\dr}}[ru] &
}
$$
 The Hyodo-Kato realization ${\R\Gamma_{\hk}}$ is obtained by composing $\R\Gamma_{\dr}$ with the projection $DK_F\to M_K(\phi,N,G_K)$. 
  Then the essential commutativity of the part (1) and (2)
 of the above diagram corresponds, respectively, for any Voevodsky's  motive $M$, to the potentially semistable comparison theorem and 
 to the Hyodo-Kato quasi-isomorphism:
\begin{align*}
\rho_{\pst}:& \quad \R\Gamma_{\hk}(M)\otimes_{K^{\nr}_0}\B_{\st}
 \simeq \R\Gamma_{\eet}(M)\otimes_{\Qp}\B_{\st}, \\
\iota_{\dr}:& \quad \R\Gamma_{\hk}(M)\otimes_{K^{\nr}_0}\ovk
 \simeq \R\Gamma_{\dr}(M)\otimes_{K}\ovk.
\end{align*}
Again, the identification \eqref{eq:real_Voevodsy&Nori} holds
 when $\epsilon=\hk$ and the above canonical comparison quasi-isomorphisms 
correspond to  the comparison isomorphisms obtained in the previous section.
\end{num}

\begin{remark}
By construction, for any (smooth) $K$-variety $f:X \rightarrow \Spec(K)$,
 one has a canonical identification:
 $\R\Gamma_{\eet}(f_*(\un_X))=\derR f_*(\Qp)$ where the right hand side
 denotes the right derived functor of the direct image for \'etale $p$-adic
 sheaves.

This implies that the realization functor $\R\Gamma_{\eet}$ constructed above
 coincides with that of \cite[7.2.24]{CD4}, denoted by $\rho_p^*$,
 and equivalently to the one defined in \cite{Ayoub2}.
 In particular, it can be extended to any base and commutes with
 the six functors formalism. This explains the preceding relation and
 why we have prefered the covariant realization rather than the contravariant
 one (see the end of Example \ref{ex:basic_realization}).\footnote{In Section
  \ref{sec:modules}, we will similarly extend the realization functor $\R\Gamma_{\pst}$ to arbitrary $K$-bases (see more precisely Rem. \ref{rem:Rsynt_extends}).}
\end{remark}

\begin{example}\label{ex:syntomic_ssp_DM}
The above realizations  allow us to define syntomic cohomology
 of a motive $M$ in $\DMgm(K,\Qp)$ as
$$
\R\Gamma_{\synt}(M):=\R\Hom_{D(\Rep_{\pst})}(\Qp,\R\Gamma_{\pst}(M))=\R\Hom_{D(DF_K)}(K(0),\R\Gamma_{DF_K}(M)).
$$
In particular, we have the syntomic descent spectral sequence
$$
{}^{\synt}E^{i,j}_2:=H^i_{\st}\big(G_K,H^j\R\Gamma_{\eet}(M)\big)
 \Rightarrow H^{i+j}\R\Gamma_{\synt}(M).
$$
If we apply it to the cohomological Voevodsky's  motive $M(X)^\vee=f_*(\un_X)$
 of any $K$-variety $X$ with structural morphism $f$, 
 we get back the results of Theorem \ref{thm:syntomic_descent}.
 
 An interesting case is obtained by using the (homological) motive with
 compact support $M^c(X)$ in $\DMgm(K,\Qp)$ of Voevodsky
 for any $K$-variety $X$, and its
 dual $M^c(X)^\vee=\underline{\Hom}(M(X),\Qp)$ which belongs to $\DMgm(K,\Qp)$.
 Then $\R\Gamma_{\synt}(M^c(X)^\vee(r))$ is the \emph{$n$-th twisted syntomic complex
  with compact support}
 and we recover  the syntomic descent spectral sequence with compact support from Remark \ref{compact}:
$$
{}^{\synt,c}E^{i,j}_2:=H^i_{\st}\big(G_K,H^j_{\eet,c}(X_{\ovk},\Qp(r))\big)
 \Rightarrow H^{i+j}_{\synt,c}(X,r).
$$
Indeed, in terms of the 6 functors formalism, $M^c(X)^\vee(r)=f_!(\un_X)(r)$
 and the identification relevant to compute the above $E_2$-term follows
 from the previous remark.
 \end{example}

 \subsection{Example I: $p$-adic realizations of the motivic fundamental group}
%\subsubsection{Nori's motivic fundamental group}
     Let $\HMMe(K)_{\Qp}$ denote the category of effective homological Nori's motives, i.e., the diagram category $\scc(\widetilde{\Delta}^{\eff}_g,H_*)$, $H_*:=(H^*)^*:=\Hom(H^*, \Q_p)$, where the diagram $\widetilde{\Delta}^{\eff}_g$ is obtained from the diagram $\Delta^{\eff}_g$ by reversing the edge $f^*$ to $f_*: (X,Y,i)\to (X^{\prime},Y^{\prime},i)$ and changing $\partial$ to $\partial:  (X,Y,i)\to (Y,Z,i-1)$. There is a duality functor $\vee: \HMMe(K)_{\Qp}\to\MM(K)_{\Qp}^{\op}$ respecting the representations $H_*$ and $H^*$ via the usual duality that sends a good pair $(X,Y,i)$ to $(X,Y,i)$. This induces an equivalence on the derived categories $\vee: D^b(\HMMe(K)_{\Qp})\stackrel{\sim}{\to} D^b(\MM(K)_{\Qp})^{\op}$. 
                
In \cite{Cu1} Cushman developed a motivic theory of the fundamental group, i.e., he showed that the unipotent completion of the fundamental group of varieties over complex numbers carries a motivic structure in the sense of Nori. We will recall his main theorem.

$\bullet$ Let ${\mathcal Var}^*_K$ be the category of pairs $(X,x)$, where $X$ is a variety defined over $K$ and $x$ is a $K$-rational base point;  morphism between such pairs are  morphisms between the corresponding varieties defined over $K$ that are compatible with  the base points.

$\bullet$ Let ${\mathcal Var}^{**}_K$ be the category of triples $(X;x_1,x_2)$, where $X$ is defined over $K$ and $x_1, x_2$ are $K$-rational base points. 

For a variety $X$ over ${\mathbf C}$, let $\pi_1(X,x)$ be the fundamental group of $X$ with base point $x$ and let  $\pi_1(X;x_1,x_2)$ be the space of based paths up to homotopy from $x_1$ to $x_2$.  Denote  by $I_{x_2}$ -- the augmentation ideal in $\Q_p[\pi_1(X,x_2)]$ (i.e., the kernel of the augmentation map $\Q_p[\pi_1(X,x_2)]\to \Q_p$) which acts on the right on $\pi_1(X;x_1,x_2)$. The following theorem  \cite[Thm 3.1]{Cu2} shows that the quotient
$\Q_p[\pi_1(X;x_1,x_2)]/I^n_{x_2}$, $n\in\N$, has motivic version $\Pi^{n}(X;x_1,x_2)$ (in the sense of Nori).

\begin{theorem} For every $n\in \N$, there are functors 
\begin{align*}
\Pi^n:{\mathcal Var}^{**}_K \to \HMMe(K)_{\Qp},\quad 
\Pi^n:{\mathcal Var}^{*}_K  \to \HMMe(K)_{\Qp}.
\end{align*} These functors have the following properties.
\begin{enumerate}
\item There is a natural transformation
$$
\Pi^{n+1}(X;x_1,x_2)\to \Pi^{n}(X;x_1,x_2).
$$
\item We have a natural isomorphism of  $\Q_p$-vector spaces
$$\widetilde{H}(\Pi^n(X({\mathbf C});x_1,x_2))\simeq \Q_p[\pi_1(X({\mathbf C});x_1,x_2)]/I^n_{x_2}.$$ 
\item There are natural transformations
\begin{align*}
 \Pi^{n}(X;x_1,x_2) \otimes  \Pi^{n}(X;x_2,x_3)\to
\Pi^{n}(X;x_1,x_3)\\
\Pi^{n+m+1}(X,x_2) \to   \Pi^{m+1}(X,x_2)\otimes
\Pi^{n+1}(X,x_2)
\end{align*}
Via the natural isomorphisms in (2), these transformations  are compatible with the product and coproduct structures as well as with the inversion  in the path space.
\end{enumerate}
This data is equivalent to giving a pro-$\HMMe$ structure on the inverse limit $\Q_p[\pi_1(X({\mathbf C});x_1,x_2)]/I^n_{x_2}$ such that all the obvious maps are motivic, and the completed ideal $I^{\wedge}_{x_2}$ is a sub-motive. 
\end{theorem}

  Dualizing the realization functors  of Nori's motives used in Constructions \ref{constr1}, \ref{constr11} we obtain the following functors
 \begin{align*}
& \Pi^n_{\eet}:\quad {\mathcal Var}^{**}_K\to \Rep(G_K),\quad   \Pi^n_{\eet}:=\R_{\eet}\Pi^n;\\
& \Pi^n_{\hk}:\quad {\mathcal Var}^{**}_K\to M_K(\phi,N,G_K),\quad  \Pi^n_{\hk}:=\R_{\hk}\Pi^n;\\
& \Pi^n_{\dr}:\quad {\mathcal Var}^{**}_K\to V_{\dr}^K,\quad  \Pi^n_{\dr}:=\R_{\dr}\Pi^n.
\end{align*}
These realizations are compatible with change of the index $n$ and with the structure maps that endow these realizations with Hopf algebra  structures.

From Constructions \ref{constr11},\ref{constr2} (again dualizing) we obtain also the following comparison isomorphisms.
\begin{corollary}
\begin{enumerate}
\item There exists the Hyodo-Kato natural equivalence
$$
\iota_{\dr}:\quad \Pi^n_{\hk}(X;x_1,x_2)\otimes_{K^{\nr}_0}\overline K\simeq \Pi^n_{\dr}(X;x_1,x_2)\otimes_K\overline K.
$$
\item There exists a natural equivalence (potentially semistable period isomorphism)
$$
\rho_{\pst}:\quad \Pi^n_{\hk}(X;x_1,x_2)\otimes_{K^{\nr}_0}\B_{\st}\simeq \Pi^n_{\eet}(X;x_1,x_2)\otimes_{\Q_p}\B_{\st}$$
that is compatible with Galois action, Frobenius, the monodromy operator. Extending to  $\B_{\dr}$ and using the Hyodo-Kato equivalence, we get the de Rham period isomorphism
$$
\rho_{\dr}:\quad \Pi^n_{\dr}(X;x_1,x_2)\otimes_{K}\B_{\dr}\simeq \Pi^n_{\eet}(X;x_1,x_2)\otimes_{\Q_p}\B_{\dr}
$$
that is compatible with  filtrations. 
\end{enumerate}
These comparison isomorphisms are compatible with change of the index $n$ and with Hopf algebra  structures.

\end{corollary}
The above comparison statements were proved before in the case of curves in  \cite{Ha}, \cite{AIK},  for varieties with good reduction over slightly ramified base in \cite{Vo}, and  for varieties with good reduction over an unramified base in  \cite{Ol}. The various realizations appearing in these constructions should be naturally isomorphic with ours but we did not check it.

\subsection{Example II: $p$-adic comparison maps with compact support and compatibilities}

\begin{num}\label{num:real_functoriality}
When $\epsilon=\hk, \eet, \dr, {DF_K}, \pst$, we get from the
 preceding section, for any $K$-variety, a complex
$$
\R\Gamma_\epsilon(X):=
 \R\Gamma_\epsilon(M(X)^\vee)=\R\Gamma_\epsilon(M(X))^*
$$
which computes the $\epsilon$-cohomology with enriched
 coefficients. When $\epsilon=\eet, \hk, \dr$ this is the usual complex,
 respectively, of Galois representations, $(\phi,N,G_K)$-modules,
 filtered $K$-vector spaces which computes, respectively, geometric \'etale
 cohomology, Hyodo-Kato cohomology and De Rham cohomology with
 their natural algebraic structures.
 These complexes are related by the comparison isomorphisms
 $\rho_{\dr}$, $\rho_{\pst}$, and $\iota_{\dr}$.

An interesting point is that these complexes, as well as the comparison
 isomorphisms are contravariantly functorial in the homological motive
 $M(X)$.
 Recall Voevodsky's motives are equipped with special covariant
 functorialities.

Let $X$ and $Y$ be $K$-varieties. A \emph{finite correspondence}
 $\alpha$
 from $X$ to $Y$ is an algebraic cycle in $X \times_K Y$
 whose support is finite equidimensional over $X$ and which is
 \emph{special} over $X$ in the sense of
 \cite[8.1.28]{CD3}.\footnote{If $X$ is geometrically unibranch,
  every $\alpha$ whose support is finite equidimensional over $X$
        is special (cf. \cite[8.3.27]{CD3}).
        If $Z$ is a closed subset of $X \times_K Y$
        which is flat and finite over $X$, the cycle associated with $Z$
        is a finite correspondence (cf. \cite[8.1.31]{CD3}).}
        Then by definition, $\alpha$ induces a map $\alpha_*:M(X) \rightarrow M(Y)$.
 
Assume now that $X$ and $Y$ are smooth. Let $f:X \rightarrow Y$ be any
 morphism of schemes of constant relative dimension $d$.
 Then we have the Gysin maps $f^*:M(Y) \rightarrow M(X)(d)[2d]$
  (cf. \cite{Deg6}).
\end{num}
\begin{corollary}
Consider the notations above.

Then $\R\Gamma_\epsilon(X)$ is contravariant with respect to finite
 correspondences and covariant with respect to morphisms of smooth
 $K$-varieties.

Moreover, the comparison isomorphisms  $\rho_{\dr}$, $\rho_{\pst}$, $\iota_{\dr}$
 are natural with respect to these functorialities.
\end{corollary}

\begin{remark}
\begin{enumerate}
\item Note in particular that covariance with respect
 to finite correspondences implies the existence of transfer maps
 $f_*$ for any finite equidimensional morphism $f:X \rightarrow Y$ which
 is special (eg. flat, or $X$ is geometrically unibranch).
\item The syntomic descent spectral sequence and the syntomic
 period map of Example \ref{ex:syntomic_ssp_DM} are natural with
 respect to the functorialities of the corollary.
\item We can deduce from \cite{Deg6} the usual good properties
 of covariant funtoriality (compatibility with composition, projection
  formulas, excess of intersection formulas,...)
\end{enumerate}
\end{remark}

\begin{num}\label{num:real_product}
 \textit{Products}. Consider again the notations
 of the Paragraph \ref{num:real_functoriality}.
 As said previously, 
  from the K\"unneth formula, $\R\Gamma_\epsilon$ is a monoidal functor
         and the comparison isomorphisms are isomorphisms of monoidal functors.

Consider a $K$-variety $X$ with structural morphism $f$.
 Recall from Remark \ref{rem:motives_singular} that
 $M(X)^\vee=f_*(\un_X)$. The functor $f_*$ is left adjoint to a 
 monoidal functor. Therefore it is weakly monoidal and we get a pairing:
$$
\mu:\quad M(X)^\vee \otimes M(X)^\vee=f_*(\un_X) \otimes f_*(\un_X)
 \rightarrow f_*(\un_X)=M(X)^\vee
$$
in $\DMgm(K,\Qp)$.
This induces a cup-product on the $\epsilon$-complexes:
\begin{equation}\label{eq:cup-product}
\begin{split}
\R\Gamma_\epsilon(X) \otimes \R\Gamma_\epsilon(X)
 =\R\Gamma_\epsilon\big(f_*(\un_X)) \otimes \R\Gamma_\epsilon(f_*(\un_X)\big)
 \stackrel{\textit{K}} \simeq
 \R\Gamma_\epsilon(f_*(\un_X) \otimes f_*(\un_X))
 \xrightarrow{\mu_*} &\R\Gamma_\epsilon(f_*(\un_X) \\
 &=\R\Gamma_\epsilon(X),
\end{split}
\end{equation}
where the isomorphism labelled $K$ stands for the structural morphism
 of the monoidal functor $\R\Gamma_\epsilon$ -- and corresponds to
 the K\"unneth formula in $\epsilon$-cohomology. When
 $\epsilon=\eet, \hk, \dr$, we deduce from the definition of this
 structural isomorphism that these products correspond to the natural
 products on the respective cohomology.
As the comparison isomorphisms are isomorphisms of monoidal functors,
 we deduce that they are compatible with the above  cup-products.

From the end of Example \ref{ex:syntomic_ssp_DM},
 we can also define the $\epsilon$-complex of $X$ with compact support:
$$
\R\Gamma_{\epsilon,c}(X)=\R\Gamma_{\epsilon}(f_!(\un_X)).
$$
Because we have a natural map $f_* \rightarrow f_!$ of functors
 (\cite[2.4.50(2)]{CD3}), we also deduce, as usual,  a natural map:
$$
\R\Gamma_{\epsilon,c}(X) \rightarrow \R\Gamma_{\epsilon}(X).
$$
From the 6 functors formalism, we get a pairing in $\DMgm(K,\Qp)$:
$$
\mu_c:\quad f_*(\un_X) \otimes f_!(\un_X) \stackrel{(1)}\simeq
 f_!(f^*f_*(\un_X) \otimes \un_X)=f_!(f^*f_*(\un_X))
 \xrightarrow{(2)} f_!(\un_X)
$$
where the isomorphism (1) stands for the projection formula
 (\cite[2.4.50(5)]{CD3}) and the map (2) is the unit map
 of the adjunction $(f^*,f_*)$. Then, using $\mu_c$ instead
 of $\mu$ in formula \eqref{eq:cup-product}, we get the 
 pairing between cohomology and cohomology with compact support:
\begin{equation}\label{eq:pairing_c}
\R\Gamma_\epsilon(X) \otimes \R\Gamma_{\epsilon,c}(X)
 \rightarrow \R\Gamma_{\epsilon,c}(X).
\end{equation}
Using again the fact that the comparison isomorphisms
 $\rho_{\dr}$, $\rho_{\pst}$, $\iota_{\dr}$ are isomorphisms
 of monoidal functor,
 we deduce that they are compatible with this pairing.
 Let us summarize:
\end{num}
\begin{proposition}
For $*=\emptyset, c$, we have comparison isomorphisms
\begin{align*}
\iota_{\hk,*}:\quad & \R\Gamma_{\hk,*}(X)\otimes_{K_0^{\nr}}\ovk\simeq \R\Gamma_{\dr,*}(X)\otimes_K\ovk, 
\\ \rho_{\pst,*}:\quad &  \R\Gamma_{\hk,*}(X)\otimes_{K_0^{\nr}}\B_{\st}\simeq \R\Gamma_{\eet,*}(X_{\ovk})\otimes_{\Qp}\B_{\st}, \\
\rho_{\dr,*}:\quad &  \R\Gamma_{\dr,*}(X)\otimes_{K}\B_{\dr}\simeq \R\Gamma_{\eet,*}(X_{\ovk})\otimes_{\Qp}\B_{\dr}, 
\end{align*}
that are  compatible with cup-products \eqref{eq:cup-product}
 and with the pairing \eqref{eq:pairing_c}.
\end{proposition}

\section{Syntomic modules}\label{sec:modules}

\subsection{Definition}
In this section we use the dg-algebra $\synspx K$, which represents syntomic cohomology of varieties over $K$ \cite[Appendix]{NN} to define a category of syntomic modules over any such variety. This is our candidate for coefficients systems (of geometric origin) for syntomic cohomology. We prove that in the case of $\Spec K$ itself the category of syntomic coefficients  is (via the period map) a subcategory of potentially semistable representations that is closed under extensions. We call such representations  {\em constructible  representations}.

    Let us first recall the setting of Voevodsky's $h$-motives,
 with coefficients in a given ring $R$ and over any
 noetherian base scheme $S$.
 We let $\Sh(S,R)$ be the category of h-sheaves of $R$-modules
 on $\Sch_S$ -- the category of separated schemes of finite type over $S$. This is a monoidal Grothendieck abelian category
 with generators the free $R$-linear $h$-sheaves represented by any $X$ in $\Sch_S$; we denote them by $R_S^h(X)$.
 In particular, its derived category $\mathcal D(\Sh(S,R))$
 has a canonical structure of a stable monoidal $\infty$-category
 in the sense of \cite[Def. 3.5]{Rob}
 (see also \cite{Lur2}).\footnote{Actually,
 this follows from the existence of a closed monoidal category
 structure on the category of complexes of $\Sh(S,R)$ (cf. \cite{CD1}
 or \cite{CD4}) and from \cite[Sec. 3.9.1]{Rob}.}
 Moreover, it admits infinite direct sums.
Let us define the Tate object as the following
 complex of $R$-sheaves: $R_S(1):=R^h_S(\PP^1_S)/R^h_S(\{\infty\})[-2]$. \\

   The following theorem is an $\infty$-categorical summary
 of a classical construction phrased in terms of model
 categories in \cite{CD4}:
\begin{theorem}\label{thm:exists_DM}
There exists a universal monoidal $\infty$-category 
 $\underline{\mathcal DM_h}(S,R)$ which admits infinite
 direct sums and is equipped with a monoidal $\infty$-functor
$$
\Sigma^\infty:\quad \mathcal D(\Sh(S,R))
 \rightarrow \underline{\mathcal DM}_h(S,R)
$$
such that:
\begin{itemize}
\item \textit{$\AA^1$-Homotopy}: for any scheme $X$ in $\Sch_S$, the induced map
 $\Sigma^\infty R^h_S(\AA^1_X) \rightarrow \Sigma^\infty R^h_S(X)$ is 
 an isomorphism;
\item \textit{$\PP^1$-stability}: the object $\Sigma^\infty R_S(1)$ is $\otimes$-invertible.
\end{itemize}
Moreover, the monoidal $\infty$-category $\underline{\mathcal DM}_h(S,R)$
 is stable and presentable.
\end{theorem}
Concerning the first point, the statement follows from
 the existence of localization for monoidal $\infty$-categories.
 The statement for the second point follows from \cite[4.16]{Rob}
 and the fact that, up to $\AA^1$-homotopy, the cyclic permutation
 on $R_S(1)^{\otimes,3}$ is the identity.

\begin{remark}\label{rem:DM=graded_presheaves}
According to \cite{CD4} and \cite[4.29]{Rob},
 the $\infty$-category $\underline{\mathcal DM}_h(S,R)$
 is associated with an underlying symmetric monoidal model category
 -- this also implies it can be described by a  canonical $R$-linear
 dg-category.  According to the description of this model category, 
 up to quasi-isomorphism, the objects of $\underline{\mathcal DM_h}(S,R)$
 can be understood as $\mathbb N$-graded complexes of $R$-linear
 $h$-sheaves $(\mathcal E_r)_{r \in \mathbb N}$ which satisfy
 the following properties:
\begin{itemize}
% \item (\emph{$h$-descent}) for any $h$-hypercover
%  $W_\bullet \rightarrow X$ and $r \in \bz$,
%  the canonical map $\mathcal E_r(X) \rightarrow
%   \operatorname{Tot}(\mathcal E_r(W_\bullet))$
%       is a quasi-isomorphism;
\item (\emph{Homotopy invariance}) for any integer $r$,
 the $h$-cohomology presheaves $H^*_h(-,\mathcal E_r)$ are
 $\AA^1$-invariant;
\item (\emph{Tate twist}) there exists a (structural) quasi-isomorphism
 $\mathcal E_r \rightarrow \uHom(R_S(1),\mathcal E_{r+1})$.
\end{itemize}
One should be careful however that,
 in order to get the right \emph{symmetric} monoidal structure
 on the underlying model category, one has to consider in addition
 an action of the symmetric group of order $r$ on $\mathcal E_r$,
 in a way compatible with the structural isomorphism associated
 with Tate twists. The corresponding objects are called
 \emph{symmetric Tate spectra}.\footnote{See \cite[Sec. 5.3]{CD3} for the 
 construction in motivic homotopy theory.}
\end{remark}

\begin{example}\label{ex:synt_sp}
Let $S=\Spec(K)$ and $R=\Qp$.
Consider the $h$-sheaf associated with the presheaf of dg-$\Qp$-algebras
$$
X \mapsto (\R\Gamma_{\synt}(X_h,r) \simeq \R\Gamma_{\synt}(X,r))
$$
defined in \ref{def1} (see Theorem \ref{compsynpH} for the isomorphism).
Because of \cite{NN},
 it satisfies the homotopy invariance and Tate twist properties
 stated above;
 thus as explained in Appendix B of \cite{NN}, 
 it canonically defines an object $\mathcal E_{syn}$
 of $\underline{\mathcal DM}_h(K,\Qp)$.
 Moreover, the dg-structure allows to get a canonical ring structure
 on this object, which corresponds to a strict structure
 (the diagrams encoding commutativity and associativity are commutative
  not only up to homotopy).
\end{example}

For any scheme $X$ in $\Sch_S$,
 we put $M_S(X):=\Sigma^\infty R_S^h(X)$, called the (homological)
 $h$-motive associated with $X/S$. 
\begin{definition}\label{df:DMh}
We define the stable monoidal $\infty$-category of $h$-motives
 $\mathcal DM_h(S,R)$
 (resp. constructible $h$-motives $\mathcal DM_{h,c}(S,R)$)
 over $S$
 with coefficients in $R$ as the sub-$\infty$-category 
 of $\underline{\mathcal DM}_h(S,R)$, spanned by infinite direct sums
 of objects of the form $M_S(X)(n)[i]$
 (resp. objects of the form $M_S(X)(n)[i]$)
 for a smooth $S$-scheme $X$ and integers $(n,i) \in \bz^2$.

We let $DM_h(S,R)$ (resp. $DM_{h,c}(S,R)$) be the
 associated homotopy category, as a triangulated monoidal
 category.
\end{definition}

\begin{example}\label{ex:DMh_comparison}
When $R$ is a $\bq$-algebra
 (resp. $R$ is a $\bz/n$-algebra where $n$ is invertible on $S$),
 $DM_h(S,R)$ is equivalent to the triangulated monoidal category
 of rational mixed motives (resp. derived category of $R$-sheaves
 on the small site \'etale of $S$): see \cite{CD4}, Th. 5.2.2 (resp. Cor. 5.4.4).
 In particular, $\mathcal DM_h(S,R)$
 is presentable by a monoidal model category.
\end{example}

The justification of the axioms of $\AA^1$-homotopy and $\PP^1$-stability
 added to the derived category of $h$-sheaves comes from the following theorem:
\begin{theorem}[\cite{CD4}]
The triangulated categories $DM_h(S,R)$ for various schemes $S$ are
 equipped with Grothendieck 6 functors formalism and satisfy
  the absolute purity property.
 If one restricts to quasi-excellent schemes $S$
 and morphisms of finite type,
 the subcategories $DM_{h,c}(S,R)$ are stable under the 6 operations,
 and satisfy Grothendieck-Verdier duality.
\end{theorem}
We refer the reader to \cite[A.5]{CD3} of \cite[Appendix A]{CD4} for
 a summary of Grothendieck 6 functors formalism and Grothendieck-Verdier
 duality.

Let us now take the notations of Example \ref{ex:synt_sp}.
 We view $\mathcal E_{syn}$ in the model category
 underlying $\underline{\mathcal DM}_h(K,\Qp)$, equiped
 with its structure of (commutative) dg-algebra.
 According to \cite[7.1.11(d)]{CD3}, one can assume that
 $\mathcal E_{syn}$ is cofibrant (by taking a cofibrant resolution
 in the category of dg-algebras according to \emph{loc. cit.}).
 Given any morphism $f:S \rightarrow \Spec(K)$, we put
$$
\mathcal E_{syn,S}:=\mathrm Lf^*(\mathcal E_{syn})
$$
which is again a dg-algebra because $f^*$ is monoidal.
According to the construction of \cite[Sec. 7.2]{CD3},
 the category $\usMod_S$
 of modules over this dg-algebra is endowed
 with a monoidal model structure, and therefore with
 a structure of monoidal $\infty$-category.
The free $\synsp$-module functor induces an adjunction of
$\infty$-categories:
$$
R_{\synt}:\underline{\mathcal DM}_h(S,\Qp) \leftrightarrows 
 \usMod_S:\mathcal O_{\synt}.
$$
Given any $S$-scheme $X$, and any integer $n \in \Z$,
 we put $\synspx S(X)(n):=R_{\synt}(M_S(X)(n))$ .
\begin{definition}
Using the above notations, we define
 the $\infty$-category of \emph{syntomic modules}
 (resp. \emph{constructible syntomic modules})
 over $S$ as the $\infty$-subcategory of $\usMod_S$ stable under taking  infinite direct sums
 (resp. finite direct sums) and generated by $\synspx S(X)(n)[i]$ for
 a smooth $S$-scheme $X$ and integers $(n,i) \in \Z^2$.
 
We denote it by $\sMod_S$ (resp. $\sMod_{c,S}$)
 and let $\smod_S$ (resp. $\smod_{c,S}$) be its associated homotopy category.
 This is a monoidal triangulated category.
\end{definition}
In particular, we get an adjunction of triangulated categories:
$$
R_{\synt}:DM_h(S,\Qp) \leftrightarrows \smod_S:\mathcal O_{\synt},
$$
such that $R_{\synt}$, called the realization functor, is monoidal
 and sends constructible  motives to constructible syntomic modules.

\begin{remark}\label{rem:smod&compact}
By definition, the triangulated category $\smod_S$ 
 (resp. $DM_h(S,\Qp)$) is generated 
 by the objects of the form $\synspx S(X)(n)$ (resp. $M_S(X)(n)$)
 for a smooth $S$-scheme $X$ and an integer $n \in \Z$.
By construction, the functor $\mathcal O_{\synt}$ commutes
 with arbitrary direct sums. Thus, because $M_S(X)(n)$ is
 compact\footnote{Recall
  an object $M$ of a triangulated category $\mathcal T$ is compact when
  the functor $\Hom_\mathcal T(M,-)$ commutes with arbitrary direct sums.}
 in $DM_h(S,\Qp)$ (see \cite[15.1.4]{CD3}), we deduce that
 $\synspx S(X)(n)$ is compact.
This implies that a syntomic module is constructible
 if and only if it is compact.\footnote{This corresponds to the description
 of perfect complexes of a ring as compact objects of the derived category.}

Note also that $\smod_S$ is a \emph{compactly generated}
 triangulated category.
\end{remark}

Essentially using the previous theorem
 and the good properties of the forgetful functor $\mathcal O_{\synt}$
 -- we get the following result:
\begin{theorem}\label{thm:h-realization}
The triangulated categories $\smod_S$ for various schemes $S$ are
 equipped with Grothen\-dieck 6 functors formalism and satisfy
  the absolute purity property.
 If one restricts to quasi-excellent $K$-schemes $S$
 and morphisms of finite type,
 the subcategories $\smod_{c,S}$ are stable under the 6 operations,
 and satisfy Grothendieck-Verdier duality.

If one restricts to $K$-varieties $S$, the syntomic (pre-)realization
 functors:
$$
R'_{\synt}:\quad DM_{h,c}(S,\Qp) \rightarrow \smod_{c,S},
$$
for various $S$, commute with the 6 operations and in particular with duality.
\end{theorem}
See Corollary \ref{cor:syntomic_real&modules} for
 the computation of this functor over the base field $K$.
\begin{proof}
The first assertion comes from \cite[Prop. 7.2.18]{CD3},
 which also implies that the motivic category $\smod$ is
 separated.\footnote{Recall from \cite[2.1.7]{CD3} that this means the following: for any $K$-schemes $X$, $Y$
 and surjective morphism $f:Y \rightarrow X$ of finite type,
 the base change functor $f^*:\smod_X \rightarrow \smod_Y$
 is conservative.}
 Thus the second assertion comes
 from \cite[Th. 4.2.29]{CD3}.\footnote{As we work over a field of
 characteristic $0$, the absolute purity property is easy to get.
 Thus the premotivic triangulated category $\smod$ is compatible
 with Tate twist in the sense of Def. 4.2.20 of \cite{CD3}.}
 The last assertion is \cite[Th. 4.4.25]{CD3}.
\end{proof}
\begin{remark}
To get a feeling for the category $\smod_{c,S}$ the reader might want to recall a more classical case of coefficients defined by de Rham cohomology. Let $K={\mathbf C}$ be the field of complex numbers; let $\se_{\dr}$ be the commutative ring spectrum representing de Rham cohomology $X\mapsto \R\Gamma_{\dr}(X)$, for varieties $X$ over $K$.  We have 
$$H^n_{\dr}(X)=\R\Hom_{DM_h(K,{\mathbf C})}(M(X),\se_{\dr}[n]).
$$
We can define, in a way analogous to what we have done above, the category of constructible de Rham coefficients $\se_{\dr}-\operatorname{mod}_{c,S}$, for varieties $S$ that are smooth over $K$. By \cite[Example 17.2.22]{CD3} (using the Riemann-Hilbert correspondence) or by \cite[Theorem 3.3.20]{Drew} (more directly, using the isomorphism between Betti and de Rham cohomologies) this category is equivalent to the bounded derived category of analytic regular holonomic $\sd$-modules on $S$ that are constructible, of geometric origin.
\end{remark}
\begin{num}\label{num:duality}
Recall the Grothendieck-Verdier duality property
 means that for any regular $K$-scheme $S$
 and any separated morphism of finite type $f:X \rightarrow S$,
 the syntomic module $M_X=f^!(\synspx S)$ is dualizing for
 the category of constructible syntomic modules over $X$.
In other words, the functor
\begin{equation}\label{eq:duality_synt_module}
D_X:=\uHom(-,M_X):\quad \left(\smod_{c,X}\right)^{op}
 \rightarrow \smod_{c,X}
\end{equation}
is an anti-equivalence of monoidal triangulated categories.
 Moreover, it exchanges usual functors with exceptional functors:
 given any separated morphism of finite type $p:Y \rightarrow X$,
 one has: $D_Yp^*=p^!D_X$ and $D_Xp_*=p_!D_Y$.
\end{num}

\subsection{Comparison theorem}

\begin{num}\label{num:rep&fiber_functor}
Consider the abelian category $\Rep_{pst}(G_K)$ of
 potentially semistable representations  and the coinvariants  functor
$$
\omega_{!}:\quad \Rep_{pst}(G_K) \rightarrow V_{\Qp}^f 
$$
where the right hand side is the category of finite dimensional
 $\Qp$-vector spaces.
 It admits a right adjoint denoted by
 $\omega^!$
 which to a finite dimensional $\Qp$-vector space $V$ associates
 the representation $V$ with trivial action of $G_K$. 
 It is obviously exact and monoidal.
 One could also put $\omega^*=\omega^!$ because it also admits a right
 adjoint $\omega_*$ which to a potentially semistable representation
 $V$ associates the $\Qp$-vector $V^{G_K}$ of $G_K$-invariants.
 The situation can be pictured as follows:
$$
\xymatrix@=40pt{
\Rep_{pst}(G_K)\ar@<5pt>^-{\omega_!}[r]\ar@<-5pt>_-{\omega_*}[r]
 & V_{\Qp}^f.\ar|-{\omega^!=\omega^*}[l]
}
$$

   It will be convenient for what follows to enlarge the category
 $\Rep_{pst}(G_K)$. Consider the category
$$
\Rep_{pst}^\infty(G_K):=\Ind-\Rep_{pst}(G_K)
$$
of ind-objects.
 Thus, for us,
 an infinite potentially semistable representation $V$ will be
 a $\Qp$-vector space $V$ with an action of $G_K$ which is
 a filtering union of sub-$\Qp$-vector spaces stable under the action of  $G_K$
 which are potentially semistable representations of $G_K$.
The category $\Rep_{pst}^\infty(G_K)$ is an abelian (symmetric closed) monoidal
 category which contains $\Rep_{pst}(G_K)$ as a full abelian
 thick subcategory. Moreover, it is a Grothendieck abelian category
 -- it admits infinte direct sums and filtering colimits are exact.
 The above diagram of functors extends to this larger
 category. Note in particular that according to this definition,
 Formula \eqref{def1} can be rewritten:
\begin{equation} \label{def1bis}
V_{\pst}\theta^{-1}:\quad \R\Gamma_{\synt}(X,r)\stackrel{\sim}{\to}\R\omega_* \R\Gamma_{\pst}(X_{\ovk},r).
\end{equation}
\end{num}

Due to the Drew's thesis \cite{Drew} together with our main
 construction (\textsection \ref{num:main_construction}), we get
 the following computation of syntomic modules over $K$:
\begin{theorem}\label{thm:compute_synt_modl}
There exist a canonical pair of adjoints of triangulated categories:
$$
\rho^*:\smod_{K}
 \leftrightarrows \mathrm D\big(\Rep^\infty_{\pst}(G_K)\big):\rho_*
$$
such that $\rho^*$ is monoidal and
 which can be promoted to an adjunction of stable
 $\infty$-categories. Moreover, the functor $\rho^*$ is fully
 faithful and induces by restriction a monoidal fully faithfull
 triangulated functor:
$$
\rho^*:\smod_{c,K}
 \rightarrow \mathrm D^b\big(\Rep_{\pst}(G_K)\big)
$$
such that for any $K$-variety $X$
 with structural morphism $f$, there exists a canonical
 quasi-isomorphism of complexes of $G_K$-representations:
\begin{equation} \label{eq:compute_synt_modl}
\rho^*\left(f_* \synspx X(r)\right)
 \simeq \R\Gamma_{\pst}(X_{\ovk},r).
\end{equation}
\end{theorem}
\begin{proof}
We will apply Theorem 2.2.7 and Proposition 2.2.21 of \cite{Drew}.
To be consistent with the notations of \emph{loc. cit.}, we take $B=\Spec(K)$
 and put $\TT_0=\Rep_{\pst}(G_K)$,  $\TT=\Rep_{\pst}^\infty(G_K)$.

Consider the functor
 $\tilde{\mathcal E}_{syn}:X \mapsto \R\Gamma_{\pst}(X_{\ovk},0)$ (recall that $\R\Gamma_{\pst}(X_{\ovk},0)\simeq \R\Gamma_{\eet}(X_{
\ovk},\Q_p(0))$ as Galois representations).
 This is a presheaf of dg-$\Qp$-algebras on $K$-varieties
 with values in $\TT_0$. 
Then $\tilde{\mathcal E}_{syn}$ satisfies the axioms of a mixed Weil
 $\TT_0$-theory
 in the sense of \cite[2.1.1]{Drew}:
 the axiom (W1) comes from the fact $\tilde{\mathcal E}_{syn}$ satisfies
 $h$-descent which is stronger than Nisnevich descent,
 (W2), (W3) comes from homotopy invariance of geometric $p$-adic Hodge cohomology
 and the computation of the syntomic cohomology of $K$,
 (W4) comes from the projective bundle formula for geometric $p$-adic Hodge cohomology,
 and (W5) was proved in Lemma \ref{lm:kunneth_abs_synt}.
 Then we can apply 2.2.7 and 2.2.21 of \emph{loc. cit.} to
 $\tilde{\mathcal E}_{syn}$
 and this gives the theorem.

Let us explain this in more detail.
 First, Drew generalizes Theorem \ref{thm:exists_DM},
 to the category  $SH_{\Rep_{\pst}(G_K)}(S)$ of Nisnevich sheaves with values
 in the category of ind-representations $\TT$, seen as an enriched category
 over $\TT$ -- morphisms are not simply sets but ind-representations.
 This defines the $\Rep_{\pst}(G_K)$-enriched stable homotopy category
 over any base scheme $S$.
 Drew proves that this category is a stable monoidal $\infty$-category
 -- actually it is defined by a monoidal model category -- 
 that we will denote here by $\mathcal D_{\AA^1}(K,\TT)$.
 We will  denote by $\mathcal D_{\AA^1}(K,\Qp)$
  the usual monoidal $\infty$-category of $\AA^1$-homology,
         obtained by replacing $\TT$ with the category
         of $\Qp$-vector spaces--
 and the associated homotopy category still satisfies
 the 6 functors formalism (cf. \emph{loc. cit.},
 Prop. 1.6.7).\footnote{Essentially, its object are graded
 presheaves on the category of smooth $S$-scheme with values
 in $\TT$ satisfying homotopy invariance, Tate twist,
 as in Remark \ref{rem:DM=graded_presheaves}, but
 we have to add the Nisnevich descent property.}

Then applying Theorem  2.1.4 of \emph{loc. cit.} to the presheaf
 $\tilde{\mathcal E}_{syn}$ we get
 that the geometric $p$-adic Hodge cohomology is representable
 in $SH_{\Rep_{\pst}(G_K)}(S)$ by a commutative monoid 
 $\tilde{\mathcal E}_{syn}$
 in the underlying model category
 -- in our case the corresponding object is simply the collections
  of presheaves $X \mapsto \R\Gamma_{\pst}(X_{\ovk},r)$, as a $\N$-graded 
        dg-algebra indexed by $r$, seen as presheaves on $\Sm_K$ (the category of smooth $K$-varieties) with values
        in $\TT$.

Then Drew shows that one can define a monoidal $\infty$-category
 of modules over the dg-algebra $\tilde{\mathcal E}_{syn}$ which
 is enriched over $\TT$, that we will denote here
 by $\smodx{\tilde{\mathcal E}_{syn}}_K$.
 It follows that we have the following interpretation of the K\"{u}nneth formula: by  Theorem 2.2.7 of \emph{loc. cit.}  the
 functor 
$$
\tilde \rho:\quad 
 \smodx{\tilde{\mathcal E}_{syn}}_K \xrightarrow \sim D(\TT),\quad 
 M \mapsto
 \derR Hom_{\tilde{\mathcal E}_{syn}}^\TT(\tilde{\mathcal E}_{syn},M),
$$
where $Hom_{?}^\TT$ indicates the enriched Hom
  (with values in complexes of $\TT$), is an equivalence of monoidal triangulated categories.
 Recall that any smooth $K$-variety $X$ defines a canonical
 $\tilde{\mathcal E}_{syn}$-module $\tilde{\mathcal E}_{syn}(X)$.
 It follows from the construction that,
 for any smooth $K$-variety $X$ and any integer $r \in \Z$,
 there exists a canonical quasi-isomorphism:
\begin{equation} \label{eq:pf_compute_synt_modl}
\derR Hom_{\tilde{\mathcal E}_{syn}}^\TT(\tilde{\mathcal E}_{syn}(X),\tilde{\mathcal E}_{syn}(r)) \simeq \R\Gamma_{\pst}(X_{\ovk},r)
\end{equation}
functorial in $X$.

  Now we descend. According to \emph{loc. cit.}, 1.6.8, the pair of adjoint functors
 $(\omega^*,\omega_*)$  induces an
 adjunction of stable $\infty$-categories:
$$
\derL \omega^*:\mathcal D_{\AA^1}(K,\Qp)
 \leftrightarrows \mathcal D_{\AA^1}(K,\TT):\derR \omega_*
$$
such that $\derL \omega^*$ is monoidal. Then Drew defines
 (\emph{loc. cit.}, 2.2.13)
 the absolute cohomology associated with the enriched
 mixed Weil cohomology $\tilde{\mathcal E}_{\synt}$
 as
 $\derR \omega_*(\tilde{\mathcal E}_{\synt}),$
 seen as a monoid
 in $\mathcal D_{\AA^1}(K,\Qp)$ -- recall $\R\omega_*$ is
 weakly monoidal. According to this definition,
 Formula \eqref{def1bis}, and the definition recalled in Example
 \ref{ex:synt_sp}, we get:
$$
\mathcal E_{\synt} \simeq \R\omega_*(\tilde{\mathcal E}_{syn}),
$$
the absolute cohomology associated with
 $\tilde{\mathcal E}_{syn}$.
 According to this definition, we deduce from the adjunction
 $(\derL \omega^*,\derR \omega_*)$ an adjunction of
 stable $\infty$-categories:
$$
\derL\tilde \omega^*:\smod_K \leftrightarrows \smodx{\tilde{\mathcal E}_{syn}}_K:
 \R \tilde \omega_*
$$
whose left adjoint, $\derL \tilde \omega^*$, is monoidal.
Therefore, one gets the first two statements of the Theorem
 by putting:
$$
\rho^*=\tilde \rho \circ \derL\tilde \omega^*,\quad 
 \rho_*=\tilde \omega_* \circ \derR\tilde \rho^{-1}.
$$

Moreover, Prop. 2.2.21 of \emph{loc. cit.} tells us that
 $\derL\tilde \omega^*$ is an equivalence of categories
 if one restricts to constructible objects on both sides
 (\emph{i.e.}, generated by, respectively, the objects of
 the form
 $\mathcal E_{syn}(X)(r)$ and $\tilde{\mathcal E}_{syn}(X)(r)$
 for a smooth $K$-scheme $X$ and an integer $r \in \Z$).
  The fact that
 $\rho^*$ is fully faithful is a formal consequence of this result
 together with the fact that $\smod_K$ is compactly generated
 (cf. Rem. \ref{rem:smod&compact}).

Recall that, for any smooth $K$-variety $X$ with structural morphism
 $f:X \rightarrow \Spec(K)$, one gets:
$$
\tilde{\mathcal E}_{syn}(X)
 =\derL\tilde \omega^*({\mathcal E}_{syn}(X))
 =\derL\tilde \omega^*(f_!f^!{\mathcal E}_{syn,K})
 =\derL\tilde \omega^* D_K(f_*f^*{\mathcal E}_{syn,K})
 =\derL\tilde \omega^* D_K(f_*{\mathcal E}_{syn,X}),
$$
where $D_K$ is the Grothendieck-Verdier duality operator
 on constructible syntomic modules over $K$ defined
 in Paragraph \ref{num:duality}.
 Thus, in the case when $X$ is a smooth $K$-variety,
 Formula \eqref{eq:compute_synt_modl} follows from
 this identification, the definition of $\rho^*$,
 and \eqref{eq:pf_compute_synt_modl}.
One removes the assumption that $X$ is smooth 
 using the fact that the quasi-isomorphism
 \eqref{eq:compute_synt_modl} can  be extended
 to diagrams of smooth $K$-varieties and that
 both the left and the right  hand side satisfies (by definition) 
 cohomological descent for the h-topology.

\end{proof}

\begin{remark}\label{rem:concrete_syntomic_modl/K}
As a consequence,
 the category of constructible syntomic modules over $K$
 can be identified with a full triangulated subcategory $\D$ of
 the derived category $D^b\big(\Rep_{\pst}(G_K)\big)$.

It is easy to describe this subcategory:
 using resolution of singularities, all objects of
 $\smod_{c,K}$ are obtained by taking
 iterated extensions\footnote{Recall: in a triangulated category $\T$,
 an object $M$ is an extension of $M''$
 by $M'$ if there exists a distinguished triangle
 $M' \rightarrow M \rightarrow M'' \rightarrow M'[1]$ in $\T$.}
 or retracts of syntomic modules of
 the form $f_*(\se_{\synt,X})(r)$ for 
 a smooth projective morphism $f:X \rightarrow \Spec(K)$
 and an integer $r \in \Z$
 (this is an easy case of the general result \cite[4.4.3]{CD3}).
 So $\D$ is the full subcategory of $D^b\big(\Rep_{\pst}(G_K)\big)$
 whose objects are obtained by taking retract of iterated extensions
 of complexes of the form $\R\Gamma_{\pst}(X_{\ovk},r)$ 
 for $X/K$ smooth projective and $r \in \Z$.

Similarly, the (essential image of the) category
 of (not necessarily constructible)
 syntomic modules over $K$ can be identified with
 the smallest full triangulated subcategory of
 $D\big(\Rep^\infty_{\pst}(G_K)\big)$
 stable under taking (infinite) direct sums and which contains complexes
 of the form $\R\Gamma_{\pst}(X_{\ovk},r)$
 with the same assumptions as above. 
\end{remark}

 Composing the syntomic realization functor over $K$
 with the fully faithful functor $\rho^*$ above, we get:
\begin{corollary}\label{cor:syntomic_real&modules}
The syntomic (pre-)realization functor of Theorem \ref{thm:h-realization}
 in the case $S=\Spec(K)$ defines a triangulated monoidal
 realization functor:
$$
R_{\synt}:\quad DM_{gm}(K,\Qp)
 \simeq DM_{h,c}(K,\Qp) \xrightarrow{R'_{\synt}}
 \smod_{c,K} \xrightarrow{\rho^*} \mathrm D^b\big(\Rep_{\pst}(G_K)\big).
$$
It coincides with the functor $\R\Gamma_{\pst}$
 defined in Paragraph \ref{num:real_Voevodsky}.
\end{corollary}
\begin{proof}
Only the last statement requires a proof. By definition,
 $\R\Gamma_{\pst}$ is the functor defined on $\DMgm(K,\Qp)$
 applying Example \ref{ex:basic_realization}
 to the functor which to a smooth affine $K$-variety $X$
 associates the complex $\R\Gamma_{\pst}(X_{\ovk},r)$.
 Thus the statement follows from the description of the functor
 $\rho^*$ in the above proof and the
 identification \eqref{eq:pf_compute_synt_modl}. 
\end{proof}

\begin{remark} \label{rem:Rsynt_extends}
The corollary means in particular that the realization
 $R'_{\synt}$ of Theorem \ref{thm:h-realization} does indeed 
 extends the realization $\R\Gamma_{\pst}$ to arbitrary $K$-bases
 in a way compatible with the 6 operations.
\end{remark}
\begin{corollary}
For a variety $f: X\to \Spec(K)$, we have a natural quasi-isomorphism
$$
\R\Gamma_{\sh}(X,r)=\R\Hom_{\smod_X}(\se_{\synt,X},\se_{\synt,X}(r)).
$$
\end{corollary}
\begin{proof}Since, by the above theorem,  $\rho^*(f_*\se_{\synt,X}(r))\simeq \R\Gamma_{\pst}(X_{\ovk},r)$, 
we have
\begin{align*}
\R\Hom_{\smod_X} & (\se_{\synt,X},\se_{\synt,X}(r)) =\R\Hom_{\smod_X}(f^*\se_{\synt,K},\se_{\synt,X}(r))\\
 & =
\R\Hom_{\smod_K}(\se_{\synt,K},f_*\se_{\synt,X}(r))=\R\Hom_{D(\Rep_{\pst}(G_K)}(\Qp,\R\Gamma_{\pst}(X_{\ovk},r))\\
 & \simeq R\Gamma_{\sh}(X,r),
\end{align*}
as wanted.
\end{proof}
This means that we can define syntomic cohomology of a syntomic module in the following way.
\begin{definition}
Let $ X$ be a variety over $K$ and $\sm\in\smod_X$. {\em Syntomic cohomology of $\sm$} is the complex
$$
\R\Gamma_{\sh}(X,\sm)=\R\Gamma_{\synt}(X,\sm):=\R\Hom_{\smod_X}(\se_{\synt,X},\sm).
$$
\end{definition}
This definition is compatible with the definition of syntomic cohomology of Voyevodsky's motives from Example \ref{ex:syntomic_ssp_DM}. That is, for $M\in DM_{\gm}(K,\Qp)$, we have a canonical quasi-isomorphism $$\R\Gamma_{\synt}(\Spec(K),R^{\prime}_{\synt}(M))\simeq \R\Gamma_{\synt}(M). $$ This follows easily from Theorem \ref{thm:compute_synt_modl} and Corollary \ref{cor:syntomic_real&modules}.
\begin{remark}
Syntomic cohomology with coefficients was studied before in \cite{N1}, \cite{N2}, \cite{Ts0}, \cite{Ban}. The coefficients used there could be called "syntomic local systems". They are variants of the crystalline and semistable local systems introduced by Faltings \cite{Fa0}, \cite{Fa1}. There exists also a notion of "de Rham local systems". Those were introduced by Tsuzuki in his (unpublished) thesis \cite{Ts} and later by Scholze \cite{Sch} in the rigid analytic setting.

In all these cases, syntomic local systems have a de Rham avatar and an \'etale one. These two avatars  are related by relative Fontaine theory and their cohomologies (de Rham, \'etale, and syntomic) satisfy  $p$-adic comparison isomorphisms. We hope that this is also the  case for the syntomic coefficients introduced here and we will discuss it in a forthcoming paper.
\end{remark}

\subsection{Geometric and constructible representations}

\begin{definition}\label{df:motivic_rep}
Keep the notations of the previous section.
We define the category  $\Rep_{gm}(G_K)$ (resp. $\Rep_{Ngm}(G_K)$, resp. $\Rep_{c}(G_K)$) of \emph{geometric} (resp. \emph{Nori's geometric}, resp. \emph{constructible})
 {$p$-adic representations of $G_K$}
 as the essential image of the following (composite) functor:
\begin{align*}
& DM_{gm}(K,\Q_p) \xrightarrow{R_{\synt}} \mathrm D^b\big(\Rep_{\pst}(G_K)\big)
 \xrightarrow{H^0} \Rep_{\pst}(G_K), \\
  resp. \  & \R_{\pst}:\quad \MM(K)_{\Qp}\to\Rep_{\pst}(G_K),\\ 
resp. \  & \smod_{c,K}
 \xrightarrow{\rho^*} \mathrm D^b\big(\Rep_{\pst}(G_K)\big)
 \xrightarrow{H^0} \Rep_{\pst}(G_K).
\end{align*}
\end{definition}
Thus a geometric $G_K$-representation can be described as the 
 geometric \'etale $p$-adic cohomology of a Voevodsky's  motive over $K$ with its natural Galois action and Nori's geometric $G_K$-representation - as the geometric \'etale $p$-adic cohomology of a Nori's motive.
By Corollary \ref{cor:syntomic_real&modules}, a geometric 
 $G_K$-representation is  constructible and by 
 the compatibility of realizations of Nori's and Voevodsky's motives (\ref{eq:real_Voevodsy&Nori}) geometric representation is Nori's geometric.
  So we have the following inclusions of categories
 \begin{equation}
 \label{inclusions}
  \Rep_{gm}(G_K)\subset \Rep_{Ngm}(G_K)\subset \Rep_{c}(G_K)\subset \Rep_{\pst}(G_K).
\end{equation}
We do not know much about these subcategories.
 For example we even do not have a conjectural description of them
 in purely algebraic terms (for example in terms of $(\phi,N,G_K)$-modules)
 -- this contrasts very much with the case of number fields, see
  \cite{FonMaz}.

Here are  few trivial facts:
\begin{itemize}
\item All three subcategories are stable under taking tensor products
 and twists.
\item All three categories contain representations
 of the form $H^i_{\eet}(X_{\ovk},\Q_p(r))$
 for any integers $i,r \in \N \times \Z$ and any $K$-variety $X$
 (possibly singular). They also contain kernel of projectors
 of these particular representations when the projector is induced by
 an algebraic correspondence modulo rational equivalence for $X/K$
 projective smooth, and any finite correspondence for an arbitrary
 $X/K$.
\end{itemize}
 We do not know  if any of these subcategories are stable under taking sub-objects,
  quotients, or even direct factors.

   The following fact is the only nontrivial result about stability.
\begin{proposition}
\label{only}
The category $\Rep_c(G_K)$ contains all potentially
 semistable extensions of representations of the form
 $H^i_{\eet}(X_{\ovk},\Q_p(r))$ for $X/K$ smooth and projective,
 $i \in \N$, $r \in \Z$.
\end{proposition}
\begin{proof}
Let  $\D$ be the essential image of the functor
 $\rho^*:\smod_{c,K} \rightarrow \mathrm D^b\big(\Rep_{\pst}(G_K)\big)$.
 Note that $\D$ is stable under taking retracts, suspensions, and extensions
 (see Remark \ref{rem:concrete_syntomic_modl/K}).
 We first prove that for any smooth projective morphism
 $f:X \rightarrow \Spec(K)$ and any integer $r \in \Z$,
 the representation $H^i_{\eet}(X_{\ovk},\Q_p(r))$ belongs to $\D$.

The complex or representations
 $\R\Gamma_{\pst}(X_{\ovk},r) \simeq \derR f_*(\Qp)(r)$
 belongs to $\D$ (according to the end of
 Theorem \ref{thm:compute_synt_modl}).
Moreover, using \cite[4.1.1]{De2} and \cite{De1},
 there exists an isomorphism in $D^b\big(\Rep_{\pst}(G_K)\big)$:
$$
\derR f_*(\Qp)(r) \simeq \bigoplus_{i \in \Z} \derR^i f_*(\Qp)(r)[-i].
$$
This means that $\derR^i f_*(\Qp)(r)$ is the kernel of a projector
 of $\derR f_*(\Qp)(r)$, thus belongs to $\D$ because the later is stable
 under taking retracts.

Thus the result follows, using the fact that $\D$ is stable
 under taking extensions in $D^b\big(\Rep_{\pst}(G_K)\big)$.
\end{proof}

\begin{remark}
The preceding proof shows that the essential image $\D$ of
 constructible syntomic modules in complexes of pst-representations
 contains arbitrary truncations of the complexes
 $\R\Gamma_{\pst}(X_{\ovk},r)$.
 A natural question would be to determine if, more generally, 
 $\D$ is stable under taking truncation. This would immediately
 imply that $\Rep_c(G_K)$ is a thick abelian subcategory
 of $\Rep(G_K)$ (\emph{i.e.}it is stable under taking sub-objects and quotients)
 and that $\D$ is the category of bounded complexes of
 pst-representations whose cohomology groups are constructible
 in the above sense.
\end{remark}

\begin{remark}
 In the sequence of inclusions 
 \begin{equation*}
  \Rep_{gm}(G_K)\subset \Rep_{Ngm}(G_K)\subset \Rep_{c}(G_K)\subset \Rep_{\pst}(G_K)
\end{equation*}
we believe that the first one is an equality and the following two are strict. We can support this belief with the following observations. 
 The first inclusion should be  an equality since the category of Nori's motives is expected to be the heart of a motivic $t$-structure on $DM_{gm}(K,\Q_p)$ (see \cite[p. 374]{BKa}). The second inclusion should be strict by the philosophy of weights: by Proposition \ref{only},  we allow all potentially semistable extensions as extensions of certain geometric representations in the constructible category but in the geometric category such extensions should satisfy a  weight filtration condition.  For properties of geometric representations coming from abelian varieties over $\Q_p$ see the work of Volkov \cite{MV0}, \cite{MV}.
 
 For the third inclusion, take $k=F_q$, the finite field with $q=p^s$ elements. Let $V\in\Rep_c(G_K)$ be a constructible representation. Then, by the Conjecture of purity of the weight filtration, the $\phi$-module  $D_{\pst}(V)$ is an extension of "pure" $\phi$-modules, i.e., $\phi$-modules such that, for a number $a\geq s$, 
 $\phi^a$  has eigenvalues that are $p^a$ - Weil numbers\footnote{A $p^a$ - Weil number is an algebraic integer such that all its conjugates have absolute value $\sqrt{p^a}$ in ${\mathbf C}$} (cf., \cite[Conjecture 2.6.5]{Ill}). But there are crystalline representations that do not have this property. For example,  any unramified character $\chi: G_{K_0}\to\Q_p$, $Fr\mapsto \mu\in\Q_p^*$, such that $\mu$ is not a $p^a$ -Weil number for any $a\geq 0$ (such a $\mu$ exists by a uncountability of $\Qp$).
\end{remark}


\begin{thebibliography}{BE00}
\bibitem{AIK} F.~Andreatta, A.~Iovita, M.~Kim, {\em A $p$-adic non-abelian criterion for good reduction of curves}, to appear in Duke Math. Journal.
\bibitem{DA} D.~Arapura, {\em An abelian category of motivic sheaves},  Adv. Math. 233 (2013), 135Ð195.
\bibitem{Ayoub}J.~Ayoub,
\newblock {\em Les six op\'erations de {G}rothendieck et le formalisme des  cycles \'evanescents dans le monde motivique ({I})}, volume 314 of {\em  Ast\'erisque}.
\newblock Soc. Math. France, 2007.
\bibitem{Ayoub2} J.~Ayoub,
\newblock La r\'ealisation \'etale et les op\'erations de {G}rothendieck.
\newblock {\em Ann. Sci. \'Ec. Norm. Sup\'er. (4)}, 47(1):1--145, 2014.
\bibitem{AB} J.~Ayoub, L.~Barbieri-Viale, {\em 1-motivic sheaves and the Albanese functor}. J. Pure Appl. Algebra 213 (2009), no. 5, 809Ð839.
\bibitem{Ban} K.~Bannai, {\em Syntomic cohomology as a $p$-adic absolute Hodge cohomology}, Math. Z. 242/3 (2002), 443-480.
\bibitem{B0} A.~Beilinson, {\em On the derived category of perverse sheaves}, in: K-theory, Arithmetics, and Geometry, Springer LNM 1289, 1-25 (1987). \bibitem{BE0} A.~Beilinson, {\em Notes on absolute Hodge cohomology.} Applications of algebraic $K$-theory to algebraic geometry and number theory, 
Part I, II (Boulder, Colo., 1983), 35-68, Contemp. Math., 55, Amer. Math. Soc., Providence, RI, 1986. 
\bibitem{BE00} A.~Beilinson, {\em $p$-adic periods and derived de Rham cohomology}, J. Amer. Math. Soc. 25 (2012), 715-738.
\bibitem{BE1} A.~Beilinson, {\em On the crystalline period map},  preprint
arXiv:1111.3316.
\bibitem{BS} A.~Beilinson, private communication, March 2014.
\bibitem{BBD}A.~Beilinson, J.~Bernstein, P.~ Deligne, {\em  Faisceaux pervers}. (French) [Perverse sheaves] Analysis and topology on singular spaces, I 
(Luminy, 1981), 5-171, Ast\'erisque, 100, Soc. Math. France, Paris, 1982.
\bibitem{BB} O.~Ben-Bassat, J.~Block, {\em Cohesive DG categories I: Milnor descent}, arxiv/1201.6118.
\bibitem{Bh} B.~Bhatt, {\em $p$-adic derived de Rham cohomology}, preprint, arXiv:1204.6560.
\bibitem{CCM} B.~Chiarellotto, A.~Ciccioni, N.~Mazzari, {\em Cycle classes and the syntomic regulator},
Algebra and Number Theory 7 (2013), no. 3, 533--566. arXiv:1006.0132.
\bibitem{CD1} D.-C. Cisinski and F.~D{\'e}glise,
\newblock Local and stable homological algebra in {G}rothendieck abelian
  categories.
\newblock {\em HHA}, 11(1):219--260, 2009.
\bibitem{CD2}
D.-C. Cisinski and F.~D{\'e}glise.
\newblock Mixed weil cohomologies.
\newblock {\em Adv. in Math.}, 230(1):55--130, 2012.
\bibitem{CD3} D.-C. Cisinski and F.~D{\'e}glise,
\newblock Triangulated categories of mixed motives.
\newblock arXiv:0912.2110v3, 2012.
\bibitem{CD4} D.-C. Cisinski and F.~D{\'e}glise,
\newblock {\'E}tale motives.
\newblock Compositio Mathematica (to appear), 2015.
\bibitem{CF} P.~Colmez, J.-M.~Fontaine, {\em Construction des repr\'esentations $p$-adiques
semi-stables}, Invent. Math. 140 (2000), no. 1, 1-43.
\bibitem{Cu1} M.~Cushman, {\em The motivic fundamental group}, PhD thesis, Department of Mathematics, University of Chicago, (2000).
\bibitem{Cu2} M.~Cushman, {\em Morphisms of curves and the fundamental group}. Contemporary trends in algebraic geometry and algebraic topology (Tianjin, 2000), 39D54, Nankai Tracts Math., 5, World Sci. Publ., River Edge, NJ, 2002.
\bibitem{Deg6}
F.~D{\'e}glise.
\newblock Around the {G}ysin triangle {I}.
\newblock In {\em Regulators}, volume 571 of {\em Contemporary Mathematics},
  pages 77--116. 2012.
\bibitem{DegMaz} F.~D{\'e}glise and N.~Mazzari, The rigid syntomic spectrum.
 {\em Journal de l'Institut math\'ematique de Jussieu},
  FirstView:1--47, 12 2014.
\bibitem{De1} P.~Deligne,  {\em 
Th\'er\`eme de Lefschetz et crit\`eres de d\'eg\'en\'erescence de suites spectrales.} Publications
Math\'ematiques de l'IHES 35.1 (1968) 107-126.
\bibitem{De2} P.~Deligne, {\em La conjecture de Weil. II. (French) [Weil's conjecture. II]} Inst. Hautes \'Etudes Sci. Publ. Math. No. 52 (1980), 137Ð252.
\bibitem{De3}P.~Deligne, D\'ecompositions dans la cat\'egorie d\'eriv\'ee.
 In {\em Motives ({S}eattle, {WA}, 1991)}, volume~55 of {\em Proc.
  Sympos. Pure Math.}, pages 115--128. Amer. Math. Soc., Providence, RI, 1994.
\bibitem{Drew}
B.~Drew.
\newblock R\'ealisations tannakiennes des motifs mixtes triangul\'es.
\newblock
  \url{http://www.math.univ-paris13.fr/~drew/drew_realisations-tannakiennes.pdf},
  2013.
\bibitem{Dr} V.~Drinfeld, {\em DG quotients of DG categories}. J. Algebra 272 (2004), no. 2, 643-691. 
\bibitem{Fa0} {G.~Faltings}, {\em  Crystalline cohomology and $p$-adic Galois
representations}, Algebraic analysis, geometry and
number theory (J.~I. Igusa ed.), Johns Hopkins University Press,
Baltimore, 1989, 
25--80. 
\bibitem{Fa1}  { G.~Faltings}, {\em Almost \'etale extensions},
Cohomologies $p$-adiques et applications arithm\'etiques, II. Ast\'erisque No. 279
  (2002), 185--270.
\bibitem{F2} J.-M.~Fontaine, {\em Repr\'esentations $p$-adiques semi-stables}.  P\'eriodes $p$-adiques (Bures-sur-Yvette, 1988). Ast\'erisque (1994), no. 223, 113-184.
\bibitem{F3} J.-M.~Fontaine, {\em Repr\'esentations l-adiques potentiellement semi-stables}. P\'eriodes $p$-adiques (Bures-sur-Yvette, 1988). Ast\'erisque (1994),
no. 223, 321-347.
\bibitem{FonMaz}
Jean-Marc Fontaine and Barry Mazur, \emph{Geometric {G}alois representations},
  Elliptic curves, modular forms, \& {F}ermat's last theorem ({H}ong {K}ong,
  1993), Ser. Number Theory, I, Int. Press, Cambridge, MA, 1995, pp.~41--78.
\bibitem{FPR} J.-M.~Fontaine, B.~Perrin-Riou,
{\em  Autour des conjectures de Bloch et Kato: cohomologie galoisienne et valeurs de fonctions L. (French) [Concerning the conjectures of Bloch and Kato: Galois 
cohomology and values of L-functions]} Motives (Seattle, WA, 1991), 599--706, Proc. Sympos. Pure Math., 55, Part 1, Amer. Math. Soc., Providence, RI, 1994.
\bibitem{HMF} A.~Huber, {\em Mixed motives and their realization in derived categories}. Lecture Notes in Mathematics, 1604. Springer-Verlag, Berlin, 1995. 
\bibitem{Ha} M.~Hadian, {\em  Motivic Fundamental Groups and Integral Points. Duke Math}. J. 160 (2011), 503-565.
\bibitem{HW} A.~Huber, J.~ Wildeshaus, {\em Classical motivic polylogarithm according to Beilinson and Deligne}. Doc. Math. 3 (1998), 27--133.
\bibitem{HMS} A.~Huber, S.~M\"{u}ller-Stach, {\em On the relation between Nori motives and Kontsevich periods}, Preprint 2011.
\bibitem{HNo} A.~Huber et al.,  {\em Notes from Summer School Alpbach: motives , periods, and transendence}, http://www.phys.ethz.ch/~jskowera/alpbach-2011.pdf.
%\bibitem{FI} F.~Ivorra, {\em Perverse, Hodge and motivic realizations of \'etale motives}, preprint, 2014.
\bibitem{Ill} L.~Illusie, {\em Crystalline cohomology}. Motives (Seattle, WA, 1991), 43--70, Proc. Sympos. Pure Math., 55, Part 1, Amer. Math. Soc., Providence, RI, 1994.
\bibitem{BKa} B.~Kahn, {\em  Algebraic K-theory, algebraic cycles and arithmetic geometry}. Handbook of K-theory. Vol. 1, 2, 351--428, Springer, Berlin, 2005.
\bibitem{KK} K.~Kimura, {\em Nori's construction and the second Abel-Jacobi map}, Math. Res. Lett. 14 (2007), no. 6, 973-981.
\bibitem{ML} M.~Levine, {\em Mixed motives}. Handbook of K-theory. Vol. 1, 2, 429Ð521, Springer, Berlin, 2005.
\bibitem{Lur1} J.~Lurie,
\newblock {\em Higher topos theory}, volume 170 of {\em Annals of Mathematics
  Studies}.
\newblock Princeton University Press, Princeton, NJ, 2009.
\bibitem{Lur2} J.~Lurie,
\newblock Higher algebra.
\newblock \url{http://www.math.harvard.edu/~lurie/papers/higheralgebra.pdf},
  2014.
\bibitem{Mor1} F.~Morel,
\newblock On the motivic {$\pi_0$} of the sphere spectrum.
\newblock In {\em Axiomatic, enriched and motivic homotopy theory}, volume 131
  of {\em NATO Sci. Ser. II Math. Phys. Chem.}, pages 219--260. Kluwer Acad.
  Publ., Dordrecht, 2004.
\bibitem{N} A.~Neeman, {\em The derived category of an exact category}, J. Algebra 135 (1990), no. 2, 388-394.
\bibitem{NN} J.~Nekov\'a\v{r}, W.~Nizio\l, {\em Syntomic cohomology and $p$-adic regulators for varieties over $p$-adic fields}, preprint 2013.
\bibitem{N1} W.~Nizio\l, {\em Duality in the cohomology of crystalline local systems}, Compositio Math. 109 (1997), no. 1, 67--97.
\bibitem{N2} W.~Nizio{\l}, {\em Cohomology of crystalline smooth sheaves}, Compositio Math. 129 (2001), no. 2, 123--147.
\bibitem{Ni2} W.~Nizio{\l}, {\em Semistable Conjecture via K-theory}, Duke Math. J. 141 (2008), no. 1, 151-178.
\bibitem{Ni3} {W.~Nizio\l}, {\em On uniqueness of $p$-adic period morphisms}, Pure Appl. Math. Q. 5 (2009), no. 1, 163--212.
\bibitem{N1} M~Nori, Lectures at TIFR, unpublished, 32 pages.
\bibitem{Ol} M.~Olsson, {\em Towards non-abelian $p$-adic Hodge theory in the good reduction case.}
Mem. Amer. Math. Soc. 210 (2011), vi+157.
\bibitem{Qu} D.~Quillen, {\em Higher algebraic K-theory. I.} Algebraic K-theory, I: Higher K-theories 
(Proc. Conf., Battelle Memorial Inst., Seattle, Wash., 1972), pp. 85-147. Lecture Notes in Math., Vol. 341, Springer, Berlin 1973.
\bibitem{Riou} J.~Riou,
\newblock Dualit\'e de {S}panier-{W}hitehead en g\'eom\'etrie alg\'ebrique.
\newblock {\em C. R. Math. Acad. Sci. Paris}, 340(6):431--436, 2005.
\bibitem{Rob} M.~Robalo,
\newblock Noncommutative motives {I}: A universal characterization of the
  motivic stable homotopy theory of schemes.
\newblock arXiv:1206.3645v3, 2013.
\bibitem{SS} P.~Schapira, J.-P.~Schneiders, {\em Derived category of filtered objects},  arXiv:1306.1359.
\bibitem{Sn} J.-P.~Schneiders, {\em Quasi-abelian categories and sheaves}, M\'em. Soc. Math. Fr. (N.S.) ${\mathbf 76}$, 1999. 
 \bibitem{Sch} {P.~Scholze}, {\em $p$-adic Hodge theory for rigid-analytic varieties}. Forum Math. Pi 1 (2013), e1, 77 pp. 
 \bibitem{Ta} G.~Tabuada, {\em Homotopy Theory of dg-categories via Localizing Pairs and Drinfeld's DG Quotient},
Homology, Homotopy, and Applications, vol. 12,1, 2010, 187-219.
\bibitem{Ts0} T.~Tsuji, {\em Syntomic complexes and $p$-adic vanishing cycles}, J. reine angew. Math. {bf 472} (1996), 69--138.
\bibitem{Ts1} {T.~Tsuji}, {\em p-adic \'etale and crystalline cohomology
in the semistable reduction case}, Invent. math. {\bf 137} (1999), 233--411.
\bibitem{Ts} N.~Tsuzuki, {\em  Variation of $p$-adic de Rham structures}, preprint, (1991).
\bibitem{MV0} M.~Volkov, {\em Les repr\'esentations $\ell$-adiques associ\'ees aux courbes elliptiques sur $\Q_p$}, J. reine angew. Math. 535 (2001), 65--101.
\bibitem{MV} M.~Volkov, {\em A class of $p$-adic Galois representations arising from abelian varieties over $\Q_p$}, Math. Ann. 331 (2005), 889--923.
\bibitem{FSV} V.~Voevodsky, A.~Suslin, and E.~M. Friedlander,
\newblock {\em Cycles, Transfers and Motivic homology theories}, volume 143 of
  {\em Annals of Mathematics Studies}.
\newblock Princeton Univ. Press, 2000.
\bibitem{Vo} V. Vologodsky, {\em Hodge structures on the fundamental group and its applications to $p$-adic
integration.} Mosc. Math. J. 3 (2003), 205-247.
\bibitem{Yam} K.~Yamada, {\em Log rigid syntomic cohomology for strictly semistable schemes},  preprint arXiv:1505.03878.
\end{thebibliography}
 \end{document}